\documentclass[11pt,twoside]{amsart}

\usepackage{pdfsync}
\usepackage[parfill]{parskip}    
\usepackage{graphicx} 
\DeclareGraphicsExtensions{ .tiff, .jpg, .png}
\usepackage{amssymb} 
\usepackage{epstopdf} 
\usepackage[latin1,applemac]{inputenc} 
\usepackage[english]{babel}
\usepackage{esint}

\usepackage{geometry} 
\newcommand{\dd}[2] {\ensuremath{\frac{\partial {#1}}{\partial {#2}}}}

\newcommand{\xx}{ \mbox{$\textbf{x}$}}

\newcommand{\uu}{ \mbox{$\textbf{u}$}}

\newcommand{\etz}{\mbox{$\epsilon \rightarrow 0 \: $}}

\newcommand{\uue}{ \mbox{$\textbf{u}^{\epsilon}$}}

\newcommand{\ueind}[1]{ \mbox{$u^{\epsilon}_{#1}$}}
\newcommand{\veind}[1]{ \mbox{$v^{\epsilon}_{#1}$}}
\newcommand{\uind}[1]{ \mbox{$u_{#1}$}}
\newcommand{\vind}[1]{ \mbox{$v_{#1}$}}

\newcommand{\qq}{ \mbox{$\mathbb{Q}$} }
\newcommand{\rr}{ \mbox{$\mathbb{R}$} }
\providecommand{\cspace}[2]{ \mbox{$C^{#1}_{#2}$}}
\providecommand{\cspaceof}[3]{ \mbox{$C^{#1}_{#2}\left(#3\right)$}}
\providecommand{\lspace}[2]{ \mbox{$L^{#1}_{#2}$}}

\providecommand{\hspace}[1]{  \mbox{$H^{#1}$}}

\providecommand{\wspace}[2]{  \mbox{$W^{#1}_{#2}$}}

\providecommand{\sspace}[1]{  \mbox{$\mathcal{S}^{#1}$}}
\newcommand{\supp}{ \mbox{$\mathrm{supp}$} }
\newcommand{\dist}{ \mbox{$\mathrm{dist}$} }

\newcommand{\osc}{ \mbox{$ \mathrm{osc}$} }

\newcommand{\D}{ \mbox{$\Delta$} }
\newcommand{\puccip}{\mbox{$\mathcal{M}^{+}$}}
\newcommand{\puccin}{\mbox{$\mathcal{M}^{-}$}}

\newcommand{\tr}{\mbox{$\mathrm{Tr}$}}

\newcommand{\Dij}{ \mbox{$D_{ij}$}}

\newcommand{\holder}{H\mbox{$\ddot{\mathrm{o}}$}lder  }

\providecommand{\diff}[1]{  \mbox{$\mathrm{d#1}$}  }
\newcommand{\dx}{ \mbox{$\mathrm{d\xx}$} }
\newcommand{\dy}{ \mbox{$\mathrm{dy}$} }

\providecommand{\abs}[1]{\left\vert#1\right\vert}
\providecommand{\norm}[2]{\left\Vert#1\right\Vert_{#2}}

\newtheorem{prop}{Proposition}[section]

\newtheorem{lemma}[prop]{Lemma}
\newtheorem{coro}[prop]{Corolary}
\newtheorem{theor}[prop]{Theorem}
\newtheorem{definition}[prop]{Definition}
\newtheorem{remark}[prop]{Remark}



\title[A free boundary problem arising from segregation of populations]{A free boundary problem arising from segregation of populations with high competition}

\author{Veronica Quitalo}

\date{\today}

\address{Department of Mathematics, The University of Texas at Austin, 1 University Station C1200, Austin, TX 78712} \email{vquitalo@math.utexas.edu}



\thanks{The author is supported by CoLab,  UT Austin PhD programs and UT Austin.}

\begin{document}

\maketitle

\begin{abstract}
 In this work, we show how to obtain a free boundary problem as the limit of a fully non linear elliptic system of equations that models population segregation (Gause-Lotka-Volterra type). We study the regularity of the solutions. In particular, we prove Lipschitz regularity across the free boundary. The problem is motivated by the work done by Caffarelli, Karakhanyan and Fang-Hua Lin for the linear case.
  \end{abstract}

\section{Introduction}

The problem we study in the present paper is motivated by the  Gause-Lotka-Voltera model of extinction or coexistence of species that live in the same territory, can diffuse, and have high competition rates.

Consider the equation, 
\begin{eqnarray*}
\dd{\uind{i}}{t} =\underbrace{ d_{i} \D\uind{i}}_{\mbox{diffusion term}} + R_{i}\uind{i}- a_{i}\uind{i}^{2}- \sum_{i \neq j} b_{ij}\uind{i} \uind{j} & \mbox{in} \: \Omega,
\end{eqnarray*}
which models populations of different species in competition, where
\begin{itemize}
\item[ ] $\uind{i} (x, t)$ is the density of the population $i$ at time $t$ and position $x;$ 
\item[ ] $R_{i}$ is the intrinsic rate of growth of species $i$;
\item[ ] $d_{i}$ is the diffusion rate for species $i$;
\item[ ] $a_{i}$ is a positive number that characterizes the intraspecies competition for the species $i$;
\item[ ] $b_{ij}$  is a positive number that characterizes the interspecies competition between the species $i $ and $j$.
 \end{itemize}

 In the papers  \cite{mimura_effect_1991, shigesada_effects_1984} this model was studied initially without diffusion. These papers studied how species can survive or get extinct with time, depending on the interactions among them.  Upon adding diffusion, Mimura, Ei  and Fang proved that  the existence of a stable solution depends on the shape of the domain and on the relations between the coefficients in the equation. The characterization for two species  has been proved to be easier, while the three species interactions remain to be  fully understood in these papers. 

In the sequence of papers by  Dancer and Du \cite{dancer_positive_1995-1, dancer_positive_1995, dancer_competing_1995}  the authors decided to first understand better the steady case (time independent) in order to obtain results for the parabolic problem. In these papers, one can find sufficient and necessary conditions for the existence of positive solutions $(\uind{1},\uind{2})$ and  $(\uind{1},\uind{2},\uind{3})$ with explicit conditions on the coefficients $R_{i}, a_{i}, b_{ij}$   for the following problem: 
\begin{equation}
\label{second model}
\left\{
\begin{split}
- & \D\uind{i} = R_{i}\uind{i}- a_{i}\uind{i}^{2}- \sum_{i \neq j} b_{ij}\uind{i}\uind{j} \quad \mbox{in}\:  \Omega,\\
&\uind{i}=0  \quad \mbox{on} \: \partial \Omega,\\
&\uind{i} > 0 \quad \mbox{ in}\:  \Omega,\\
\end{split}
\right.
\end{equation}
with $i=1,2$ and $i=1,2,3, $ respectively.

The  spatial segregation obtained in the limit as $b_{ij} \rightarrow \infty$  of the competition-diffusion system was  associated with a free boundary problem by Dancer, Hilhorst, Mimura, and Peletier  in \cite{dancer_spatial_1999} (i.e. in the case of high competition between the species). In  \cite{dancer_free_2003} 
the existence and uniqueness of the solution to Problem (\ref{second model}) with just two populations  has been  studied  using variational methods.

Later in \cite{conti_neharis_2002}, Conti, Terracini and Verzini proved that the limit problem is related with the  optimal partition problem in $N $ dimensional domains. Since then,  several papers by Conti, Felli, Terracini and Verzini \cite{conti_asymptotic_2005, conti_regularity_2005, conti_coexistence_2006, conti_minimal_2008, conti_variational_2003} studied with a general formulation, the existence, uniqueness and regularity  for the asymptotic limit of the following system,
\begin{equation*}
\left\{
\begin{split}
-&  \D\ueind{1} = f(\ueind{1})- \frac{1}{\epsilon} \ueind{1}\ueind{2}  \quad\mbox{in}\: \Omega,\\
-&  \D\ueind{2} = f(\ueind{2})-  \frac{1}{\epsilon} \ueind{1}\ueind{2} \quad \mbox{in} \:  \Omega,\\
&\uind{i}=\phi_{i}   \quad \mbox{on} \: \partial \Omega, \:i=1,2.\\
\end{split}
\right.
\end{equation*}
where $\phi_{i} (x)\phi_{j} (x) =0, $ for $i \neq j.$  In these papers, the existence of a limit pair of functions $(\uind{1}, \uind{2})$ such that $(\ueind{1}, \ueind{2}) \rightarrow (\uind{1}, \uind{2})$  when $\epsilon \rightarrow 0$   is shown  to have a tight connection with two different mathematical problems.  Namely, 
\begin{enumerate}
\item[a)] to find the solution of a free boundary problem characterized by the conditions:
\begin{eqnarray*}
\left\{
\begin{split}
- & \D\uind{i} = f(\uind{i})\chi_{\{\uind{i} >0\}} \quad i=1,2, \\
&\uind{i}(x)> 0 \quad \mbox{in}\:  \Omega, \:  i=1,2, \\
&\uind{1}(x)\,\uind{2}(x)=0 \quad  \mbox{in}\:  \Omega,\\
&\uind{i}=\phi_{i} \quad  \mbox{on} \: \partial \, \Omega,  \: i=1,2, \\
\end{split}
\right.
\end{eqnarray*}

\item[b)] to find the solution for a optimal partition problem. 
\end{enumerate} 

The existence and uniqueness of solution for a type of free boundary problem of the form
$$
\left\{
\begin{split}
- & \D u= f(u) \chi_{\{u >0\}} \\
&u(x)> 0 \quad \mbox{in}\:  \Omega,\\
&u=0  \quad  \mbox{on} \: \partial \Omega\\
\end{split}
\right.
$$
with $u$ bounded, was studied using variational methods by Dancer  \cite{dancer_uniqueness_2006}. 

Then the regularity of  solutions for the free boundary problem
\begin{equation}
\label{eq: caff problem}
\left\{
\begin{split}
&\D\ueind{i} = \frac{1}{\epsilon} \ueind{i} \sum_{i\neq j} \ueind{j} \quad  \mbox{in} \:  \Omega,  \:  i=1, \ldots, d,  \\
&\ueind{i}>0  \quad  \mbox{in} \:  \Omega, \:  i=1, \ldots, d,  \\
&\ueind{i}(x)=\phi_{i}(x)\geq 0 \quad  \mbox{on} \:  \partial \, \Omega, \:  i=1, \ldots, d, \\
&\phi_{i}\,\phi_{j}=0 \quad   \mbox{on} \:  \partial \, \Omega,\: i \neq j \:\\
\end{split}
\right.
\end{equation}
was studied   by Caffarelli, Karakhanyan and Lin in  \cite{caffarelli_singularly_2008, caffarelli_geometry_2009} with the  viscosity approach.   

More specifically, in  \cite{caffarelli_geometry_2009} the authors proved that  the singular perturbed elliptic system  (\ref{eq: caff problem})
has as limit, when $\epsilon \rightarrow 0, $ the following free boundary problem 
\begin{equation}
\label{eq: caff limit problem}
\left\{
\begin{split}
&\D\uind{i} =0 \quad \mbox{when } \uind{i}>0, \: i=1, \ldots, d, \\
&\D(\uind{i}-\sum_{i \neq j} \uind{j})\leq 0 \quad \mbox{in} \:  \Omega, \: i=1, \ldots, d,  \\
&\uind{i}(x)>0 \quad \mbox{in} \:  \Omega,  \: i=1, \ldots, d,  \\
&\uind{i}\,\uind{j}=0 \quad   \mbox{in} \: \Omega, \: i \neq j, \\
&\uind{i}=\phi_{i} \quad   \mbox{on} \: \partial \Omega,  \: i=1, \ldots, d.\\
\end{split}
\right.
\end{equation}
They also proved that the  limit solutions $\uind{i}$ are \holder continuous and have linear growth from a free boundary point.  Also  that the set of interfaces $\{ x: \uu(x)=0\} $ consists of two parts: a singular set of Hausdorff dimension $n-2$; and a family of analytic surfaces, level surfaces of harmonic functions.

The goal of this paper is to generalize  the regularity results  for the system (\ref{eq: caff problem}) and (\ref{eq: caff limit problem}), presented in Sections 1 and 2 of \cite{caffarelli_geometry_2009}, to the following nonlinear elliptic system of equations 
\begin{eqnarray*}
\left\{
\begin{split}
& \puccin  (\ueind{i}) = \frac{1}{\epsilon}  \ueind{i} \sum_{j \neq i} \ueind{j}, \quad   \mbox{in} \:  \Omega,  \: i=1, \ldots, d,  \\
&\ueind{i}>0 \quad  \mbox{in} \:  \Omega, \: i=1, \ldots, d,  \\
&  \ueind{i} = \phi_{i} \quad  \mbox{on} \: \partial \Omega, \: i=1, \ldots, d,\\
 & \phi_{i}\, \phi_{j}=0 \quad \mbox{on} \: \partial \Omega, \:  i \neq j, \:     \\
  \end{split}
  \right.
\end{eqnarray*}
  where $\puccin$ denotes the extremal Pucci operator (see (\ref{pucci def})),  and to characterize the analogous limit problem (\ref{eq: caff limit problem}) for this case. We also address the existence of solutions of this system.

We have chosen this problem, besides its intrinsic mathematical interest,  in order to study a  model that takes into account  diffusion with preferential directions, so we are able to  model situations with maximal diffusion. The choice of the operator is also  related with its  natural  comparison with a non-divergence linear operator with measurable coefficients.

  The paper is organized as follows: in Section \ref{sec: main results} we present the main theorems  as well as some definitions and background. The proofs of these results are presented in different sections for the reader's convenience. Thus,  Section \ref{subsec: existence}  is dedicate to  the existence of solutions $\uue$.  The proof of \holder regularity up to the boundary for an equation of the type $\puccin(u) = f(x)$ with \holder boundary values in Lipschitz domain is a subsection of this section.  Then, in Section  \ref{sec: holder}, we prove \holder regularity uniform in $\epsilon$  for \uue. The characterization of the limit problem as a free boundary problem has its proof  in Section \ref{sec: limit}. The main result, the linear decay from the free boundary, is proved in Section \ref{sec: lipschitz}. The appendices contain some technical proofs and results that could distract the reader from the essential ideas if they were to be presented in  their respective sections. Thus, Appendix \ref{general pucci} has some elementary general properties of the Pucci operators. Appendix \ref{app: F S inequality} has the Fabes and Strook inequality used in Section  \ref{sec: holder}. Appendices \ref{apd for lips: subhar}, \ref{apd for lips: monotonicity} and \ref{apd for lips: barriers}  contain the Alt-Caffarelli- Friedman Monotonicity formula and the proofs of some results used in the proof of Theorem \ref{theorem: lips} in Section \ref{sec: lipschitz}.

\section{Main Results. Background and some definitions}
\label{sec: main results}

 Let $\Omega \subset \rr^{n}$ be a bounded domain where $d$ populations  co-exist. 
 Consider the following system of fully nonlinear elliptic equations with Dirichlet boundary data for
\begin{equation}
\label{eq: main problem}
\left\{
\begin{split}
&  \puccin  (\ueind{i}) = \frac{1}{\epsilon}  \ueind{i} \sum_{j \neq i} \ueind{j} , \qquad i=1, \ldots, d, \quad  \mbox{in} \: \Omega,   \\
  & \ueind{i} = \phi_{i},  \quad i=1, \ldots, d, \quad  \mbox{on} \quad \partial \Omega,
  \end{split}
  \right.
 \end{equation}
 where $\ueind{i}, (i=1, \ldots, d$) are non-negative functions defined in  $\Omega$ that can be seen as  a density of the population $i,$ and  the parameter $ \frac{1}{\epsilon}$ characterizes the level of competition between species.

Each $\phi_{i}$ is a non-negative \holder  continuous function defined on $ \partial \Omega$ such that $\phi_{i} (x)\phi_{j}(x)=0$ for $i\neq j,$ meaning that they have disjoint supports.

 Here $\puccin $ denotes the  extremal Pucci operator, defined as
\begin{eqnarray}
\label{pucci def}  
\puccin  (\omega) :=  \inf_{A \in \mathcal{A}_{\lambda, \Lambda}}    a_{ij}D_{ij}(\omega (x))=\Lambda \sum_{  e_{i}<0 } e_{i}  + \lambda \sum_{ e_{i}>0 } e_{i},
\end{eqnarray}
where  $\mathcal{A}_{\lambda, \Lambda}$ is the set of  symmetric $n \times n$ real matrices   with eigenvalues   in $[\lambda, \Lambda],$ for some fixed constants $0<\lambda<\Lambda, $ and $e_{i} $ are the eigenvalues of the matrix  $D^{2} \omega (x).$

We assume that $\ueind{i}$ are bounded, $0 \leq \ueind{i} \leq N, $ for all $i.$ Note that $\lambda \D \omega \geq \puccin (\omega),$ thus $\ueind{i} $  are subharmonic, for all $i.$

\begin{remark}
Observe that if $u$ is continuous and subharmonic in the viscosity sense then $u$ is subharmonic in the distributional sense, meaning that 
\begin{eqnarray*}  
\int_{\Omega}\D u \: \phi \: \dx := \int_{\Omega}  u\:  \D \phi \: \dx \geq 0 \qquad \forall \phi \geq 0, \: \phi \in C^{\infty}_{0}.
\end{eqnarray*}  
\end{remark}

Our results in this paper are the following:

\begin{theor}[Existence]
\label{theorem: existence} 
Let $\epsilon >0$ constant, and $\Omega$ be a Lipschitz domain. Let $\phi_{i} $ be a non-negative H\"older continuous functions defined on $\partial \Omega.$ Then there exist continuous functions  $(\ueind{1}, \cdots, \ueind{d})$ depending on the parameter $\epsilon$ such that $\ueind{i}$ is a  viscosity solution of  
\begin{eqnarray*} 
 \puccin  (\ueind{i}) = \frac{1}{\epsilon}  \ueind{i} \sum_{j \neq i} \ueind{j}, \quad \mbox{for all} \quad i=1, \ldots, d.
 \end{eqnarray*} 
  \end{theor}

\begin{theor}[Regularity of solutions]
\label{theorem: regularity of solutions} 
Let $\epsilon $ and $\phi_{i} $ be as in Theorem \ref{theorem: existence}. 
Let $\uue =(\ueind{1}, \cdots, \ueind{d})$ be solutions of Problem (\ref{eq: main problem}) in $B_{1}(0)$. Then there exist a constant $\alpha$, $0<\alpha<1 $, such that for any $\epsilon$,  $\uue \in \left( \cspaceof{\alpha}{}{B_{1}(0)} \right)^{d} $ and
\begin{eqnarray*}  
\Vert \uue\Vert_{\left(\cspace{\alpha}{}{\left( B_{\frac{1}{2}} \right)} \right)^{d}}\leq C(N),
\end{eqnarray*}  
 with $N= \sup_{j} \norm{\ueind{j}}{L^{\infty}(B_{1}(0))}$ and $C(N)$ independent of $\epsilon.$ 
\end{theor}

In the  limit as $\epsilon \rightarrow 0,$ 
 this model forces  the populations to segregate,  meaning that in the limit the supports of the functions are disjoint 
 and 
 \begin{eqnarray*}  
 \frac{\ueind{i} \ueind{j}}{\epsilon} \rightharpoonup \mu \quad \mbox{in the sense of measures, when} \: \epsilon \rightarrow 0.
\end{eqnarray*}  
 The measure $\mu $ has support on the free boundary. Recall that the support of a measure $\mu$ is the complementary of the set 
$
\{ E: E \: \mbox{the biggest open set such that} \: \mu (E)=0 \}.
$

\begin{theor}[Characterization of the limit problem]
\label{theorem: limit problem}
Let  $\phi_{i} $ be as in Theorem \ref{theorem: existence}. If  \\$\uu \in \left( \cspace{\alpha}{} \right)^{d}$  is the limit of solutions of (\ref{eq: main problem}), then
\begin{enumerate}
\item $\puccin \left(    \uind{i} - \sum_{k\neq j}  \uind{k} \right) \leq 0;$
\item $(\supp  \:u_{i})^{o} \cap \left(\supp \:(\sum_{k \neq i} u_{k})\right)^{o} = \emptyset$ for $i=1, \ldots d;$
\item $\puccin  (u_{i}) = 0, $ when $ u_{i}(x)>0,$ for $ x \in \Omega \quad i=1, \ldots, d;$
\item $\uind{i}(x)= \phi_{i}(x),$ for $x \in \partial \Omega,   \quad i=1, \ldots, d. $
\end{enumerate}
\end{theor}

\begin{theor}[Lipschitz regularity for the free boundary problem]
\label{theorem: lips}
If  $\uu \in (\cspace{\alpha}{}{(B_{1}(0))})^{d}$  is the limit of solutions of (\ref{eq: main problem}) in $B_{1}(0),$ and $x_{0}$ belongs to the set  $\partial \,( \supp \, \uind{ 1})\cap B_{\frac{1}{2}}(0),$   then, without loss of generality, the growth of $u_{1}$ near the boundary of its support  is controlled in a linear way and $\uind{1}$ is Lipschitz. More precisely, there exist a universal constant $C$ such that for any  solution $\uu, $ for any point $x_{0}$ on the free boundary:
\begin{enumerate}
\item $\sup_{B_{R}(x_{0})} \uind{1} \leq C\, R,$
\item $\Vert \uind{1} \Vert_{Lip(B_{R}(x_{0}))} \leq C, $
\end{enumerate}
where $C=C(n,\Vert \uu\Vert_{L^{2}(B_{1})})$ and $R\leq \frac{1}{4}.$
\end{theor}

The proofs are developed in the following sections. Although, the last three results are similar in spirit to the ones proved in \cite{caffarelli_geometry_2009} for the elliptic linear system of equations, our proofs use different techniques.

 \subsection{Background and some definitions}

In this part we present the definition of viscosity solutions, the spaces $\underbar{S}(\lambda, \Lambda, f)$ and $\overline{S}(\lambda, \Lambda, f)$ and some results from the fully nonlinear elliptic theory (see \cite{caffarelli_fully_1995} for the proofs and more detail).

Here $\puccip $ will denote the  positive extremal Pucci operator, 
\begin{eqnarray*}   
 \puccip (\omega) =  \sup_{A \in \mathcal{A}_{\lambda, \Lambda}}    a_{ij}D_{ij}(\omega)=\lambda \sum_{  e_{i}<0 } e_{i}  + \Lambda \sum_{ e_{i}>0 } e_{i}.
\end{eqnarray*}  
%

\begin{definition}
Let  $f$ be a continuous function defined in $\Omega$ and $0<\lambda< \Lambda$ two constants.
 We denote by 
$\underbar{S}(\lambda, \Lambda, f)$ 
the space of continuous functions $u $ defined in $\Omega$ that are viscosity  subsolutions  of 
 $\puccip(u)=f(x)$  in $\Omega$, meaning that if $x_{0} \in \Omega$,  $A $ is a neighborhood of $x_{0}$, and  $P$ a paraboloid (polynomial of degree 2) that touches $u$ from above at $x_{0}, $ i.e.
\begin{eqnarray*}  
 P(x) \geq u(x) \qquad x \in A \qquad \mbox{and}  \qquad P(x_{0}) =u(x_{0}),
\end{eqnarray*}  
then 
\begin{eqnarray*}  
\puccip(P(x_{0})) \geq f(x_{0}).
\end{eqnarray*}  
In similar way, we denote by 
$\overline{S}(\lambda, \Lambda, f)$ 
the space of continuous functions $u $ defined in $\Omega$ that are viscosity supersolutions  of 
 $\puccin(u)=f(x)$  in $\Omega$, meaning that if $x_{0} \in \Omega$,  $A $ is a neighborhood of $x_{0}$, and  $P$ a paraboloid (polynomial of degree 2) that touches $u$ from below at $x_{0}, $ i.e.
\begin{eqnarray*}  
 P(x) \leq u(x) \qquad x \in A \qquad \mbox{and}  \qquad P(x_{0})= u(x_{0}),
\end{eqnarray*}  
then 
\begin{eqnarray*}  
\puccin(P(x_{0})) \leq f(x_{0}).
\end{eqnarray*}  
\end{definition}

\begin{remark}
As in \cite{caffarelli_fully_1995} we will denote by $S^{\star}(\lambda, \Lambda, f)$ the set of viscosity solutions $$ \underbar{S}(\lambda, \Lambda,-\abs{f}) \cap \overline{S}(\lambda, \Lambda, \abs{f}).$$
\end{remark}

\begin{definition}
Let $F: \mathcal{S}\times \Omega \rightarrow \rr,$  where $ \mathcal{S}$ is the space of real $n \times n$ symmetric matrices  and $\Omega \subset \rr^{n}.$ We say that 

\begin{itemize}
\item[-] $F$ is a uniform elliptic operator if there are two positive constants $\lambda \leq \Lambda,$ called ellipticity constants, such that, for any $M \in \mathcal{S}$ and $x \in \Omega $ 
\begin{eqnarray*}
\lambda \norm{N}{} \leq F(M+N) -F(M) \leq \Lambda \norm{N}{}
\end{eqnarray*}
 for all $ N \in \mathcal{S} $ non-negative definite matrix, where  $\norm{N}{}=\sup_{\abs{x}=1}{\abs{Nx}}$ is the value of the maximum eigenvalue of $N$ if $N\geq 0.$

\item[-] $F$ is concave (convex) if it is concave (convex) as a function of $M \in \mathcal{S} .$
\end{itemize}
\end{definition}

Now, we recall the comparison principle for viscosity solutions,  Corollary 3.7  in \cite{caffarelli_fully_1995},  that states that a viscosity subsolution that is negative on the boundary has to remain negative in whole domain, and that a viscosity supersolution that is positive on the boundary has to remain positive in whole domain:

\begin{prop}
\label{prop: comparison principle}
Assume that $u \in C(\overline{\Omega}).$ Then,
\begin{enumerate}
\item $u \in \underbar{S}(\lambda, \Lambda, 0)$ and $u \leq 0$ on $\partial \Omega$  imply $u\leq 0$ in $\Omega.$
\item $u \in \overline{S}(\lambda, \Lambda, 0)$ and $u \geq 0$ on $\partial \Omega$  imply $u\geq 0$ in $\Omega.$
\end{enumerate}
  \end{prop}

The following compactness result (Proposition 4.11 in \cite{caffarelli_fully_1995}) follows from the closedness of the family of viscosity solutions of Problem (\ref{eq:general})  under the uniform convergence and the Ascoli-Arzela theorem.

\begin{prop}
\label{prop: compactness }
Let $\{  F_{k}\}_{k \geq 1}$ be a sequence of uniformly elliptic operators with ellipticity constants $\lambda, \, \Lambda$  and let $\{u_{k}\}_{k \geq 1} \subset C (\Omega)$ be viscosity solutions in $\Omega$ of 
\begin{equation}
\label{eq:general}
F_{k}(D^2 u_{k},x) = f(x)
\end{equation}
Assume that $F_{k}$ converges uniformly in compact sets of $ \mathcal{S}\times \Omega$ to $F,$ where $\mathcal{S}$ is the space of real  symmetric matrices, and that $u_{k}$ is uniformly bounded in  compact sets of $\Omega$.
Then there exist $u \in C(\Omega)$ and a subsequence of $\{u_{k}\}_{k \geq 1} $ that converges uniformly to $u$ in compact sets of $\Omega$. Moreover, $F(D^2 u, x)= f(x)$ in the viscosity sense in $\Omega.$ 
  \end{prop}

Below is the $L^{\epsilon}$ Lemma that follows from  Lemma 4.6 in \cite{caffarelli_fully_1995}, using a standard covering argument. Note that it is enough to consider $f^{+}$ instead of $\abs{f}$ due to the  Alexandroff-Bakelman-Pucci estimate, Theorem 3.2  in \cite{caffarelli_fully_1995}.

 \begin{lemma}
 \label{le lemma}
 
 If $u \in \overline{S}(f^{+})$ in $B_{1}(0),$ $u \in \cspace{}{}{(\overline{B_{1}(0)}) },$  $f$ a is continuous and bounded function in $B_{1}(0),$ and they satisfy:
\begin{enumerate}
\item  $\inf_{B_{\frac{1}{2}}(0) } u(x) \leq  1  $
\item $u(x) \geq 0$ in $B_{1}(0)$
\item $\norm{f^{+}}{\lspace{n}{}{(B_{1}(0))}}\leq \epsilon_{0}$  
\end{enumerate}
Then, if $\epsilon_{0}$ is sufficiently small, there exist $d$ and $\epsilon$ positive universal constants such that:  
$$\abs{\{x \in B_{\frac{1}{4}}(0):  u(x) \geq    t   \}} \leq d t^{-\epsilon}, \quad \mbox{for all} \quad t >0.$$ 

  \end{lemma}

Now, we recall the inequality that gives interior \holder regularity and that follows from the Harnack inequality for viscosity solutions:

\begin{prop}
\label{prop: osc decay}
Let $\omega  \in \overline{S}(\lambda, \Lambda, \abs{f}) \cap \underbar{S}(\lambda, \Lambda, -\abs{f}) $ with $f$ a continuous and bounded function in $B_{1}(0)$. Then, there exists a universal constant $\mu <1$ such that
\begin{eqnarray*}  
\osc_{B_{\frac{1}{2}}(0)} \omega \leq \mu \osc_{B_{\frac{1}{2}}(0)} \omega + \norm{f}{L^{n}(B_{1}(0))}.
\end{eqnarray*}  
  \end{prop}

The interior \holder regularity that we will use is a particular case of Theorem 7.1 in \cite{caffarelli_fully_1995},  and Sobolev embedding:

\begin{prop}
\label{prop: interior holder}
Let $\omega $ be a bounded viscosity solution  of $\puccin(\omega)=f(x)$ in $B_{1}(0),$
with $f$ a continuous bounded function in $B_{1}(0)$. Then there exists a positive constant $C$ depending only on $n, \lambda, \Lambda$ such that $\omega \in \wspace{2,p}{}{(B_{\frac{1}{2}}(0))},$ for any $p<\infty,$ and so $ \omega \in \cspace{1,\alpha}{}{(B_{\frac{1}{2}}(0))}$ for any $\alpha <1,$ and we have
\begin{eqnarray*}
\norm{\omega}{C^{1,\alpha}(B_{\frac{1}{2}}(0))}\leq C\left(  \norm{\omega}{L^{\infty}(B_{1}(0))}+ \norm{f}{L^{p}(B_{1}(0))}\right).
\end{eqnarray*}
  \end{prop}

\begin{remark}
$\:$
\begin{enumerate}
\item The same result under the same hypothesis is also valid for a general uniformly elliptic operator, concave or convex.
\item Observe that if $f \in C^{\alpha}$ then $\omega \in C^{2,\alpha}.$
\end{enumerate}
\end{remark}

\section{Existence (Theorem \ref{theorem: existence})}
\label{subsec: existence}

To prove the existence theorem, Theorem \ref{theorem: existence},  we will need a fixed point argument that can be found in  \cite{gilbarg_elliptic_2001}, pg 280. We recall the result here for the sake of completeness:

\begin{prop}
\label{prop: fixed point}
Let $\sigma$ be a closed, convex subset of a Banach space $B.$ Let $T:  \sigma \rightarrow \sigma$ be a continuous function such that $T(\sigma )$ is a pre-compact set. Then $T$ has a fixed point.
  \end{prop}

To apply the fixed point theorem, we need   an existence result and regularity up to the boundary for a  Bellman-type equation. In the following results, we denote by $G$ the operator
\begin{eqnarray*}
G[\omega_{i}]:=G(D^{2} \omega_{i}, x)=
\inf_{
\begin{split}
& \quad a_{st} \in \qq  \\ 
&\quad \lbrack a_{st} \rbrack
 \in \mathcal{A}_{\lambda, \Lambda}\\
\end{split}}
\left( a_{st}D_{st} \omega_{i}-\frac{1}{\epsilon} \omega_{i} \sum_{j \neq i} \ueind{j}\right)=\puccin(\omega_{i}) -\frac{1}{\epsilon} \omega_{i} \sum_{j \neq i} \ueind{j},
\end{eqnarray*}
with $\ueind{j}$ fixed positive continuous functions and  $\mathcal{A}_{\lambda, \Lambda}$ the set of  symmetric $n \times n$ real matrices   with eigenvalues   in $[\lambda, \Lambda],$ for $0<\lambda<\Lambda. $  The existence result is Theorem 17.18 in \cite{gilbarg_elliptic_2001},  but we also state it below in the adequate form for our purpose. 

\begin{prop}
\label{prop: general existence}
Let $\Omega $ be a bounded domain in $\rr^{n}$ satisfying the exterior sphere condition for  all $x \in \partial \Omega.$
Let $\ueind{j},$  $j \neq i,$  be given functions, and $a_{st} $ symmetric matrices. 
 Suppose that, for all $s, t, j$, there exists a positive constant $\mu$ such that:
\begin{itemize}
\item[] $a_{st} \in C^{2}(\Omega)$ and $\norm{a_{st}}{C^{2}(\Omega)}\leq \mu \lambda$;
\item[] $\frac{1}{\epsilon}\ueind{j} \in C^{2}(\Omega) $ and $\norm{ \frac{1}{\epsilon}\ueind{j} }{C^{2}(\Omega)}\leq \mu \lambda;$
\end{itemize}
and
\begin{itemize}
\item[] $0\leq \lambda \abs{\xi}^{2}\leq a_{st} \xi_{s} \xi_{t} \leq \Lambda \abs{\xi}^{2}, \quad \mbox{and} \quad \ueind{j} \geq 0.$
\end{itemize}
Then, for any $\phi_{i} \in C(\partial \Omega),$ there exists a unique solution $\omega_{i} \in C^{2} (\Omega) \cap C(\overline{\Omega})$  of 
\[
\left\{ 
\begin{array}{ll} 
 G (D^{2} \omega_{i}, x)=0, & \mbox{in } \: \Omega\\
\omega_{i}= \phi_{i} &\mbox{on } \: \partial \Omega.
\end{array} 
\right. 
\]
  \end{prop}

 We also need a generalization of the comparison principle that the reader can find on page 443 in \cite{gilbarg_elliptic_2001} and comment on page 446:

\begin{prop}
\label{prop: general comparison principle}
Let $u, \, v \in C(\overline{\Omega})\cap C^{2} (\Omega)$. If $G[u] \geq G[v] $ in $\Omega$ and $u  \leq v$ in $\partial \Omega, $ then $u \leq v$ in $\Omega.$
  \end{prop}

\subsection{\holder regularity up to the boundary.}

The next Proposition is the \holder regularity up to the boundary for a viscosity solution of an equation of the type  $\puccin(\omega) = f(x) $ in a Lipschitz domain. The proof we present here uses the comparison principle and an inductive construction of barriers. A different proof, by Luis Escauriaza, can be found in   \cite{Escauriaza_apriori_1993}, Lemma 3. Escauriaza uses the Fabes and Stroock inequality to estimate the \holder norm of the solution up to the boundary  in terms of the $\lspace{q}{}{}$ norm of the right-hand side. 

This result is an improvement of Proposition 4.12 and 4.13 in \cite{caffarelli_fully_1995}  for general Lipschitz domains.

\begin{prop}
\label{prop: holder up boundary}
Let $\Omega$ be a Lipschitz domain. Let $\omega \in \cspace{2}{}{(\Omega)}\cap \cspace{}{}{(\overline{\Omega})}$ be a viscosity solution of
\begin{equation}
\label{aux problem}
\left\{
\begin{split}
&\puccin(\omega) = f(x)  \quad  \mbox{in} \quad \Omega,   \\
& \omega = \phi   \quad  \mbox{on} \quad \partial \Omega,\\
  \end{split}
 \right.
\end{equation}
with $f \geq 0 $ and $f \in C(\overline{\Omega})\cap C^{2} (\Omega), $ $\phi \in  \cspace{\beta}{}{(\partial \Omega)}.$ Then $\omega \in \cspace{\gamma}{}{(\overline{\Omega})},$ where $\gamma = \min (\frac{\beta}{2}, \alpha) $ and $\alpha$ is the universal \holder exponent for the interior regularity.
 
  \end{prop}

The proof of this Proposition follows the same lines of the proof of Proposition 4.13, in \cite{caffarelli_fully_1995}. Once the interior regularity and the regularity for an arbitrary point in the boundary is guaranteed  the proof is basically the interplay of these two results depending on how close two points are compared to the maximum of their distance to the boundary. The interior regularity comes from Proposition \ref{prop: interior holder}.  To prove the regularity for an arbitrary point on the boundary,  Proposition \ref{prop: holder point boundary }, we need the couple of Lemmas that  follow.   

The first Lemma establishes  the decay of the solution in concentric balls centered at an external point in the outside cone. The proof uses a standard comparison argument and the use of a barrier function.

\begin{lemma}
\label{lemma: intermediate bound}
Let $\Omega$ be an Lipschitz domain and $C$ an external cone centered at $x_{0} \in \partial \Omega$, with some universal opening. Let $y \in C $ be the center of the  balls  
$B^{1} \subset  B^{2} \subset B^{3}$ such that $B^{1} \in C$,  
and that   \mbox{$ \dist (\partial B^{2} \cap \Omega, x_{0}) > \delta>0.$}  
Let $u$ be a solution of $\puccip (u) \geq 0$ in the viscosity sense in $\Omega$, such that
$u \leq 1$ on $B^{3} \cap \Omega$ 
and 
$u \leq 0$ on  $ B^{3} \cap \partial \Omega$.
Then, there exist $\lambda >0$ such that
\begin{eqnarray*}  
u (x)\leq \lambda <1 \quad \mbox{in} \quad B^{2} \cap \Omega.
\end{eqnarray*}  
 
  \end{lemma}

\begin{proof}
Since the domain is Lipschitz there exists a cone $C$ with opening equal to $\rho,$ such that for any point of the boundary we can place the cone with opening $\rho$ and vertex at that point  such that $C \cap \Omega = \emptyset.$

\begin{figure}[h]
\begin{center}
\includegraphics[scale=0.4]{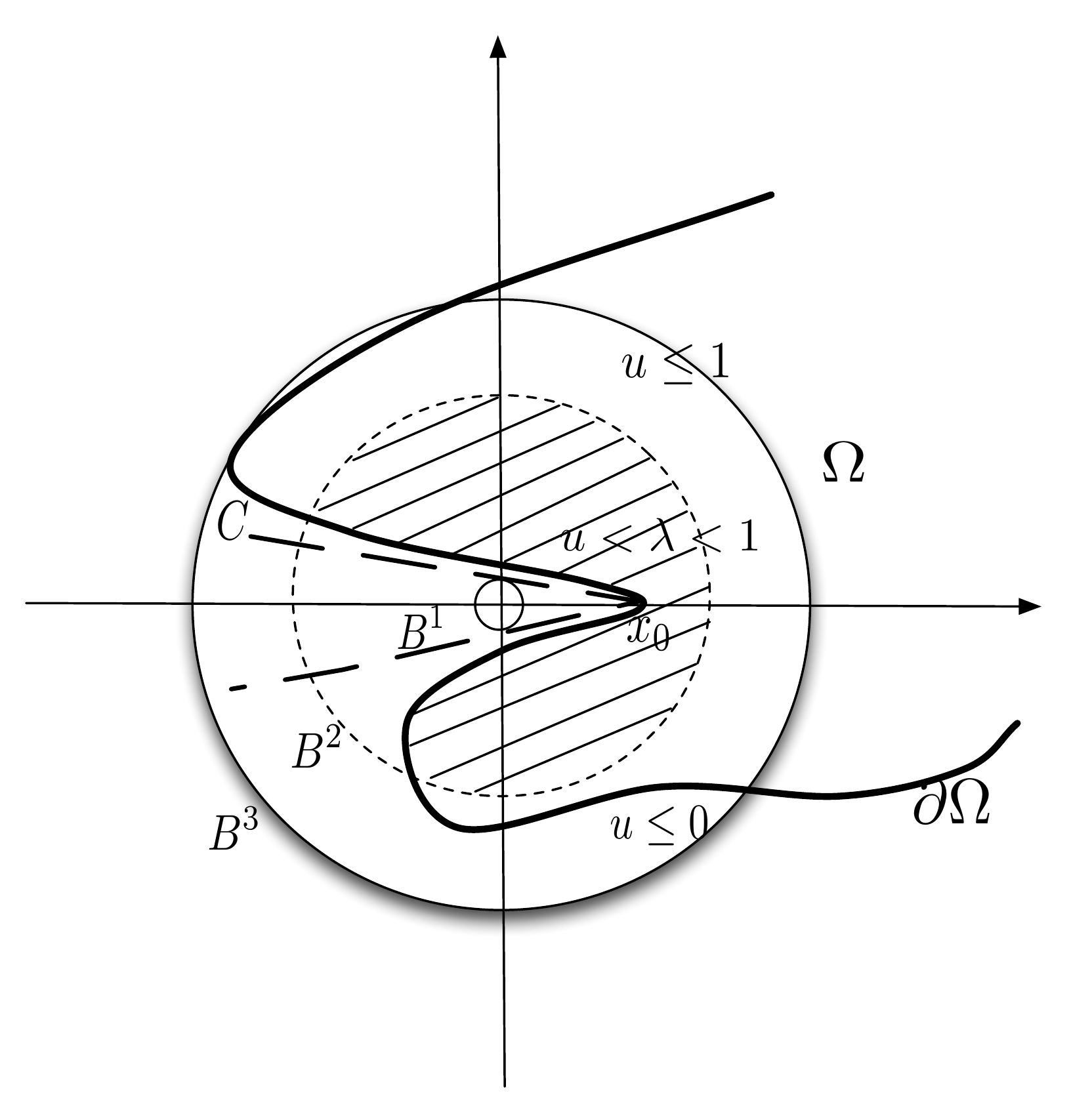}
\caption{Reference decay.}
\label{fig:reference_decay}
\end{center}
\end{figure}

Without loss of generality, take the cone with origin at $x_{0} \in \partial \Omega$ and with axis $e_{n}$, 
$$ C=\left\{ x: (x-x_{0})\cdot e_{n}<  - \rho \sqrt{\sum_{i=1}^{n-1} (x-x_{0})_{i}^{2}} \right\},
$$
and consider the origin as the center of the balls $B^{1}, B^{2}, B^{3}$ as illustrated in Figure \ref{fig:reference_decay}.

Applying Lemma \ref{lemma: barrier super} with $M=\frac{1}{r}$ where $r$ is the radius of $B^{3}$ and $\frac{ar}{b}$
equal to the radius of $B^{1}$ we obtain a supersolution $\psi$, $\puccip(\psi)\leq 0$ in $B^{3}\backslash B^{1},$ such that $\psi(x) =0 $ in $\partial B^{1}$ and $\psi (x)=1$ in $\partial B^{3}.$ And so
\begin{eqnarray*}   
u (x) \leq \psi (x) \quad \mbox{on} \quad \partial( B^{3}\cap \Omega)
\end{eqnarray*}  
Applying the comparison principle stated in Proposition \ref{prop: comparison principle} we can conclude that 
\begin{eqnarray*}  
u (x) \leq \psi (x) \quad \mbox{in} \quad B^{3}\cap \Omega
\end{eqnarray*}

\begin{figure}[h]
\begin{center}
\includegraphics[scale=0.4]{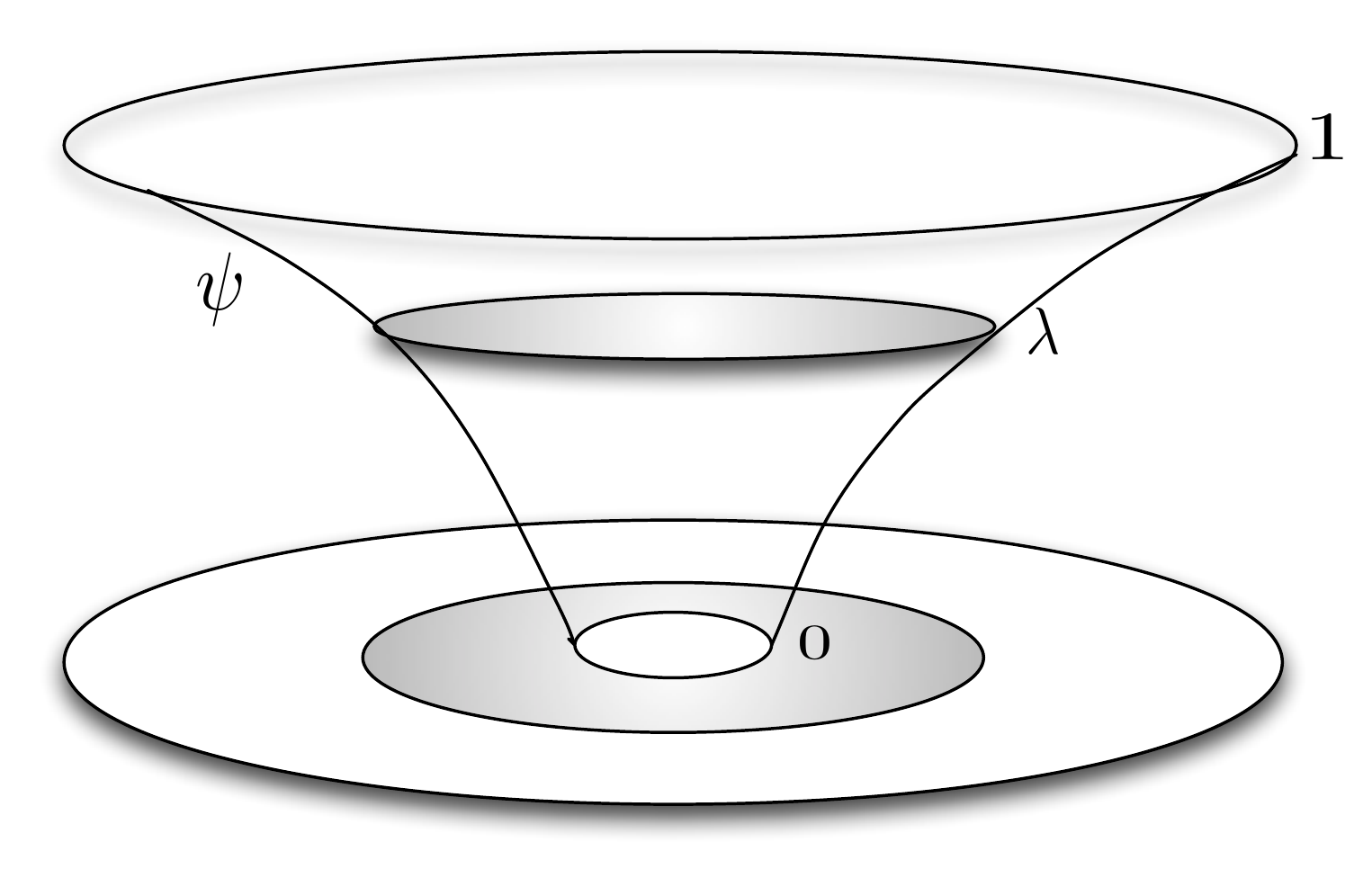}
\caption{Upper bound for $u$ due to the supersolution $\psi$. }
\label{fig: barrier bound}
\end{center}
\end{figure}

Let $\lambda=\psi (x) <1 $ for $x \in \partial B^{2} \cap \Omega.$ As $\psi$ is an increasing function, see Figure \ref{fig: barrier bound},  we can conclude that 
for $x \in B^{2}\cap \Omega$ we have that $u(x) \leq \lambda<1$ as we claim. 
\end{proof}

The next Lemma is important for the iteration construction. Basically we prove that if the boundary data is bounded from above in half of the unit ball centered on a boundary point, and the function is bounded in the unit ball  then in one fourth of the ball, the function decays by a fixed value.

\begin{lemma}
\label{lemma: reference decay }
Let $\Omega$ be a Lipschitz domain and $C$ an external cone with universal opening. Let $0 \in \partial \Omega $ be the origin of the cone. 
Let $v$  be a solution of $\puccip (u) \geq 0$ in the viscosity sense in $\Omega$, such that 
$v \leq 1$ on $B_{1}(0) \cap \Omega,$ $v(0)=0$
and 
$v (x) \leq \left(\frac{1}{2}\right)^{\beta}$ on  $ B_{\frac{1}{2}}(0) \cap \partial \Omega$,
for some $\beta >0$. 
Then, there exists a constant  $\left(\frac{1}{2}\right)^{\beta}<\mu < 1$  such that
\begin{eqnarray*}  
v(x) \leq \mu  \quad \mbox{in} \quad B_{\frac{1}{4}}(0) \cap \Omega.
\end{eqnarray*}

  \end{lemma}

\begin{proof}

By hypothesis,  
\begin{eqnarray*}  
v(x) \leq \left( \frac{1}{2}\right)^{\beta}, \quad  x \in \partial \Omega \cap  B_{\frac{1}{2}}(0).\end{eqnarray*}    
Let 
 $$\omega (x)= \frac{v(\frac{x}{2}) - (\frac{1}{2})^{\beta}}{1 -(\frac{1}{2})^{\beta} } \quad \mbox{for} \quad x \in B_{1}(0) \cap \Omega_{\frac{1}{2}},$$
 where $\Omega_\frac{1}{2}= \{2 x: x \in \Omega \}.$ 
 $\omega$ 
 satisfies all the hypotheses of Lemma \ref{lemma: intermediate bound} with:
 \begin{enumerate}
 \item $x_{0}=0;$
 \item $C$ the uniform external cone with axis without loss of generality equals to $e_{n}$  axis;
 \item $B^{1}= B_{r}(y)$ with $y=(0, \cdots, 0, - \frac{1}{10})$ and $r \leq \dist (y, \partial C), r= \frac{\rho}{10};$
  \item $B^{2}= B_{\frac{7}{10}}(y)$;
 \item $B^{3}= B_{\frac{4}{5}}(y)$.   
\end{enumerate}
Observe that 
$B^{3} \subset B_{1}(0)$ and that 
$ B_{\frac{1}{2}}(0)  \subset    B^{2}.$
Then  by Lemma \ref{lemma: intermediate bound} 
\begin{eqnarray*}  
\omega (x )\leq \lambda <1 \quad \mbox{in} \quad B^{2} \cap \Omega_{\frac{1}{2}}
\end{eqnarray*}   
and so we have also that 
\begin{eqnarray*}   
\frac{v(\frac{x}{2}) - (\frac{1}{2})^{\beta}}{1 -\left(\frac{1}{2}\right)^{\beta} } \leq \lambda \Leftrightarrow v(\frac{x}{2}) \leq \lambda (1 -\left(\frac{1}{2}\right)^{\beta} ) + \left(\frac{1}{2}\right)^{\beta}= \underbrace{\lambda + (1-\lambda)\left(\frac{1}{2}\right)^{\beta}}_{\mu}  \qquad \mbox{for} \quad x \in  B_{\frac{1}{2}}(0)  \cap \Omega_{\frac{1}{2}},
\end{eqnarray*}  
$\left(\frac{1}{2}\right)^{\beta}\leq \mu \leq 1.$
Then, we obtain that 
\begin{eqnarray*}  
v(z) \leq \mu \qquad z \in B_{\frac{1}{4}}(0) \cap \Omega.
\end{eqnarray*}  
\end{proof}

Now we will be able to prove an iterative decay illustrated in Figure \ref{fig:iteractivedecay}:

\begin{center}
\hspace{0.5in}
\begin{figure}[h]
\includegraphics[scale=0.4]{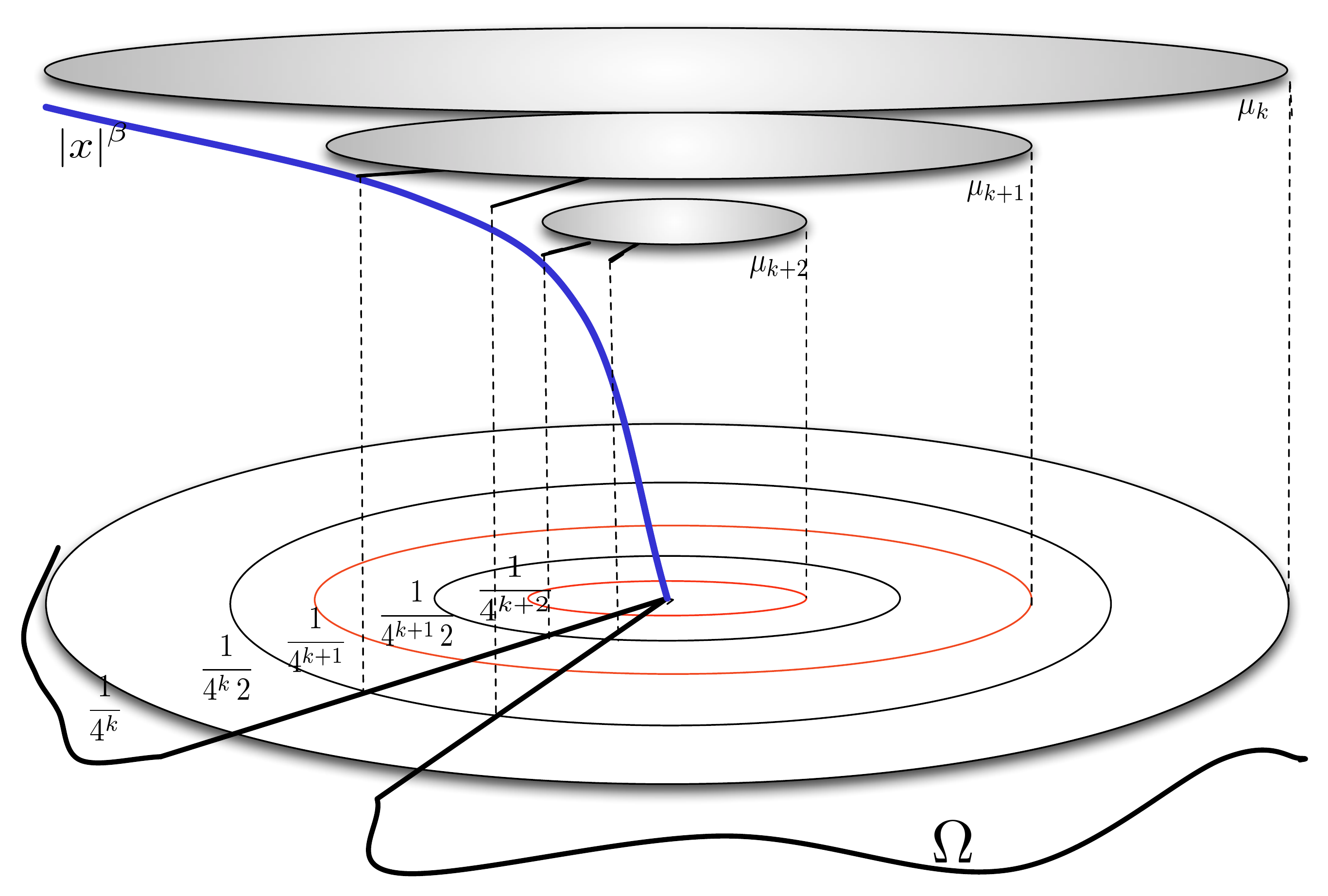}
\caption{Iterative decay.}
\label{fig:iteractivedecay}
\end{figure}
\end{center}

\begin{lemma}
\label{lemma: iteration decay }
Let $\Omega$ be a Lipschitz domain and $C$ an external cone with universal opening. Let $0 \in \partial \Omega $ be the origin of the cone. 
Let $v$  be a solution of $\puccip (u) \geq 0$ in the viscosity sense in $\Omega$, such that 
$v \leq \mu_{k}$ on $B_{\frac{1}{4^{k}}}(0) \cap \Omega,$ ($\mu_{0}=1$), $v(0)=0$
and 
$v (x) \leq \left(\frac{1}{4^{k}\,2}\right)^{\beta}$ on  $ B_{\frac{1}{4^{k}\, 2}}(0) \cap \partial \Omega$. 
Then, there exist a constant, $\mu_{k+1},$  $  \mu_{k+1}:= \lambda  \mu_{k} + (1-\lambda)\left(\frac{1}{4^{k}\,2}\right)^{\beta}$ for some $\lambda \in (0,1)$ universal,  such that
$$
v(x) \leq \mu_{k+1} \quad \mbox{in} \quad B_{\frac{1}{4^{k+1}}}(0) \cap \Omega.
$$

 \end{lemma}

\begin{proof}
By scaling and dilation, define for $x \in B_{1}(0)$
\begin{eqnarray*}  
\omega(x)= \frac{v(\frac{x}{4^{k}})- \left( \frac{1}{4^k\, 2}\right)^{\beta}}{ \mu_{k}-\left(\frac{1}{4^{k}\, 2}\right)^{\beta}}.
\end{eqnarray*}  
Since $\omega $ satisfies the hypotheses of Lemma \ref{lemma: reference decay }, considering this dilation and scaling,
we see that
\begin{eqnarray*}  
\omega (x )\leq \lambda <1 \quad \mbox{in} \quad B_{\frac{1}{4}} (0) \cap \Omega
\end{eqnarray*}  
and so, like in the previous proof,  we have also that 
\begin{eqnarray*}  
\frac{v(\frac{x}{4^{k}})- \left( \frac{1}{4^k\, 2}\right)^{\beta}}{ \mu_{k}-\left(\frac{1}{4^{k}\, 2}\right)^{\beta}} \leq \lambda \Leftrightarrow  v(\frac{x}{4^{k}}) \leq \underbrace{\lambda \mu_{k}+ (1-\lambda)\left(\frac{1}{4^{k}\, 2}\right)^{\beta}}_{\mu_{k+1}}  \qquad \mbox{for} \quad x \in  B_{\frac{1}{4}}(0)  \cap \Omega.
\end{eqnarray*}  
Therefore,
\begin{eqnarray*}  
v(y)\leq \mu_{k+1}  \qquad \mbox{for} \quad y \in  B_{\frac{1}{4^{k+1}}}(0)  \cap \Omega,
\end{eqnarray*}  
with
$\left(\frac{1}{4^{k}\, 2}\right)^{\beta}\leq \mu_{k+1} \leq \mu_{k}.$ This finishes the proof.
\end{proof}

\begin{remark}
Observe that the decay at each step of this iteration is constant and equal to $C_{0} \left(\frac{1}{4^k}\right)^{\alpha}$ for $\alpha$ much smaller than $\beta$ and $C_{0}$ a large positive constant. 
In fact, there exit constants $C_{0}$ and $\alpha$ such that $\mu_{k}\leq C_{0} \left(\frac{1}{4^k}\right)^{\alpha}.$ By induction, for $k=0$ the result is true by Lemma \ref{lemma: reference decay } for $C_{0}\geq 1$. Assuming the result valid for a general $k$, we have that 
\begin{eqnarray*}  
\mu_{k+1}=\lambda \mu_{k} + (1-\lambda) \left(\frac{1}{4^k \,2}\right)^{\beta} \leq  \frac{1}{4^{\alpha (k+1)}} \left(    4^{\alpha}  \lambda C_{0}+ (1-\lambda) \frac{1}{2^{\beta-2\alpha} 4^{k (\beta-\alpha)}} \right)
\end{eqnarray*}  
Take $\alpha =\epsilon \beta$ and such that $  4^{\alpha}  \lambda<1 $ then 
\begin{eqnarray*}  
\left(    4^{\alpha}  \lambda C_{0}+ (1-\lambda) \frac{1}{2^{\beta-2\alpha} 4^{k (\beta-\alpha)}} \right) \leq C_{0}
\end{eqnarray*}  
for $C_{0}$ large constant since,
\begin{eqnarray*}  
  4^{\alpha}  \lambda C_{0}+ (1-\lambda) \frac{1}{2^{\beta-2\alpha} 4^{k (\beta-\alpha)}} =   \underbrace{4^{\alpha}  \lambda}_{ <1} C_{0}+ (1-\lambda) \frac{1}{2^{\beta(1-2\epsilon )} 4^{k \beta(1-\epsilon )}} \leq C_{0} 
 \end{eqnarray*}  
 \begin{eqnarray*}  
\Leftrightarrow C_{0} \geq (1-\lambda) \frac{1}{(1 -  4^{\alpha}  \lambda)2^{\beta(1-2\epsilon )} 4^{k \beta(1-\epsilon )}}.  
\end{eqnarray*}

\end{remark}

\begin{remark}
\label{remark: valid for puccip}
Note that Lemmas \ref{lemma: intermediate bound} - \ref{lemma: iteration decay } are valid for  $v$ a viscosity solution of problem (\ref{aux problem}).
\end{remark}

\begin{prop}[\holder regularity up to the boundary]
\label{prop: holder point boundary }
Let $\Omega$ be an Lipschitz domain and $C$ an external cone with universal opening that only depends on the domain. Let $x_{0} \in \partial \Omega $ be the origin of the cone. 
Let $v$ be a viscosity solution of problem (\ref{aux problem}) such that 
$\abs{v(x)} \leq 1$ on $B_{1}(x_{0}) \cap \Omega,$ $v(x_{0})=0$
and 
$\abs{v (x)} \leq \abs{x}^{\beta}$ on  $ B_{1}(x_{0}) \cap \partial \Omega$. 
Then, there exist constants  $C>0$ and $\alpha<< \beta$ such that
 \begin{eqnarray*}  
\sup_{x \in \overline{B}_{1}(x_{0}) \cap \overline{\Omega}}\frac{\abs{v(x) -v(x_{0})}}{\abs{x-x_{0}}^{\alpha}}\leq C 
\end{eqnarray*}  
and $C=C(\norm{f}{\lspace{\infty}{}}).$
  \end{prop}

\begin{remark}
The constant $C$ in the previous  proposition would depend on $\epsilon$ if this result was to be applied to our main problem. The uniform  \holder regularity in $\epsilon$ will be proved in the next section and does not depend on this result though.
\end{remark}

\begin{proof}
Assume by translation invariance that $x_{0}=0.$ Note that $v(0)=v(x_{0})=0.$
So if we prove that 
\begin{eqnarray*}  
\abs{ v(x)} \leq C \abs{x}^{\alpha},
\end{eqnarray*}  
the result follows. Observe that on the boundary the regularity of the boundary data gives the result directly.
Let $k$ be such that 
\begin{eqnarray*}  
\frac{1}{4^{k+1}} \leq  \abs{x} \leq \frac{1}{4^k }.
\end{eqnarray*}  
 Using Lemma \ref{lemma: reference decay } followed by Lemma \ref{lemma: iteration decay } (see Remark \ref{remark: valid for puccip}) we can assume that for $x $ such that $\abs{x}=\rho < \frac{1}{4^k} $ we have:
\begin{eqnarray*}  
v(x) \leq \mu_{k}.
\end{eqnarray*}  
Then, taking in account the previous remark, we also have,
\begin{eqnarray*}  
v(x) \leq C_{0}\left( \frac{1}{4^{k}}\right)^{\alpha} 
\end{eqnarray*}  
but then 
\begin{eqnarray*}  
v(x)\leq C_{0}4^{\alpha} \left( \frac{1}{4^{k+1}}\right)^{\alpha} \leq C_{0} 4^{\alpha} \abs{x}^{\alpha}.
\end{eqnarray*}  
To obtain the other inequality observe that $-v(x)\leq 1.$ 
So if we consider instead of $-v$ the function
\begin{eqnarray*}  
\omega(x)=\frac{\norm{f}{\lspace{\infty}{}} }{2n\lambda} \abs{x}^{2} -v(x).
\end{eqnarray*}  
we have that,  for $x \in B_{1}(0),$ 
\begin{eqnarray*}  
\puccip (\omega(x)) \geq \puccin \left(\frac{\norm{f}{\lspace{\infty}{}} }{2n\lambda} \abs{x}^{2}\right) - \puccin (v) \geq \norm{f}{\lspace{\infty}{}} -f(x) \geq 0.
\end{eqnarray*}  
Observe  that we have as well, for   $ x \in \partial \Omega \cap B_{1}(0)$ that
\begin{eqnarray*}  
\omega(x) = \frac{\norm{f}{\lspace{\infty}{}} }{2n\lambda} \abs{x}^{2} - \phi (x) \leq \left( \frac{\norm{f}{\lspace{\infty}{}} }{2n\lambda}  + C \right)\abs{x}^{\beta} .
\end{eqnarray*}  
Therefore, we can apply the comparison principle for $\omega$ and a barrier function as in Lemma \ref{lemma: reference decay }.  
Repeating the same construction as in the proof of  Lemma \ref{lemma: reference decay } and Lemma \ref{lemma: iteration decay } we obtain also that
\begin{eqnarray*}  
\omega (x) \leq \overline{C} \abs{\rho}^{\alpha}
\end{eqnarray*}  
and so, in an analogous way, we obtain that 
\begin{eqnarray*}  
-v(x) \leq \tilde{C} \abs{\rho}^{\alpha}.
\end{eqnarray*}  
with $\tilde{C}$ depending of ${\norm{f}{\lspace{\infty}{}} }.$ This completes the proof.
\end{proof}

With this result that is the  analogous of Proposition 4.12 in  \cite{caffarelli_fully_1995}, the proof of Proposition \ref{prop: holder up boundary} follows as the proof of Propositon 4.13 in  \cite{caffarelli_fully_1995}.

\subsection{Proof of Theorem \ref{theorem: existence}}

In this section we finally present the proof of Theorem \ref{theorem: existence} stated in Section \ref{sec: main results}.

\begin{proof}

Let $B$ be the Banach space of bounded continuous vector-valued functions defined on a domain $\Omega$ with the norm
\begin{eqnarray*}
\norm{ (\ueind{1}, \ueind{2}, \cdots, \ueind{d}) }{B}=\max_{i}\left( \sup_{x \in \Omega} \abs{\uind{i}(x)}\right)
\end{eqnarray*}
Let  $\sigma$ be the subset of bounded continuous functions that satisfy prescribed boundary data, and are bounded from above and from below as is stated below:
\begin{eqnarray*}
\sigma=\{ (\ueind{1}, \ueind{2}, \cdots, \ueind{d}):  \ueind{i} \: \mbox{is continuous},\:  \ueind{i}(x)= \phi_{i} (x) \: \mbox{when} \: x \in \partial \Omega,  \quad 0 \leq \ueind{i} (x) \leq \sup_{i} \norm{\phi_{i}}{\lspace{\infty}{}} \}
\end{eqnarray*}
$\sigma $ is a closed and convex subset of $B$. 
Let $T$ be the operator that is defined in the following way:
$T \left((\ueind{1}, \ueind{2}, \cdots, \ueind{d})\right)= (\veind{1}, \veind{2}, \cdots, \veind{d})$ 
if 
$(\ueind{1}, \ueind{2}, \cdots, \ueind{d})$ and $(\veind{1}, \veind{2}, \cdots, \veind{d})$  are such that,
\begin{eqnarray*}
\left\{
\begin{split}
&  \puccin (\veind{i})=\frac{1}{\epsilon} \sum_{j\neq i} \veind{i} \ueind{j} \quad i=1, \ldots, d
\qquad  \mbox{in} \quad \Omega   \\
 & \veind{i} = \phi_{i},  \quad i=1, \ldots, d, \qquad  \mbox{in} \quad \partial \Omega,\\
  \end{split}
 \right.
\end{eqnarray*}
in the viscosity sense, where $\ueind{j},$ $j\neq i$ are fixed.
Let  $g(\uue)=\frac{1}{\epsilon} \sum_{j\neq i}  \ueind{j}$ and observe that each of the previous equations of the vector $\uue=(\ueind{1}, \ueind{2}, \cdots, \ueind{d}) $ take the form:
$$ \puccin (\veind{i})= \veind{i} \: g(\uue). $$
Observe that if $T $ has a fixed point, then 
$$
T\left((\ueind{1}, \ueind{2}, \cdots, \ueind{d})\right)= (\ueind{1}, \ueind{2}, \cdots, \ueind{d})
$$
meaning that   $\ueind{i} = \phi_{i}$  on the boundary and that
\begin{eqnarray*}
\puccin (\ueind{i})=\frac{1}{\epsilon} \sum_{j\neq i} \ueind{i} \ueind{j} \qquad i=1, \ldots, d
\end{eqnarray*}
 in $\Omega$, which proves the existence as desired.

So in order for $T$ to have a fixed point we need to prove that it satisfies the hypotheses of Proposition \ref{prop: fixed point}:

\begin{enumerate}
\item ${\it{ T(\sigma) \subset \sigma:}} $ We need to prove that there exists a regular solution for each one of the equations  
\begin{eqnarray*}
\left\{
\begin{split}
& \puccin (\veind{k})- \veind{k} \: g(\uue)=0, \quad  k=1, \ldots, d, \quad  \mbox{in} \quad \Omega   \\
&  \veind{k} = \phi_{k}, \qquad k=1, \ldots, d, \quad  \mbox{in} \quad \partial \Omega.\\
  \end{split}
 \right.
\end{eqnarray*}

Observe that if such a regular solution exists, the comparison principle is valid by Proposition \ref{prop: general comparison principle}. 

To use Proposition \ref{prop: general existence} we can rewrite the differential equation in the form
\begin{eqnarray*}
F(\veind{k})= \inf_{
\begin{split}
&\: a_{ij} \in \qq  \\ 
& \: \lbrack a_{ij} \rbrack
 \in \mathcal{A}_{\lambda, \Lambda}\\
\end{split}} 
\left( a_{ij}\Dij \veind{k}- \veind{k} \: g(\uue)\right)=0, \qquad  k=1, \ldots, d.
\end{eqnarray*}
since, by density, taking the infimum in $\qq $ is equal to taking the infimum over $\rr.$ Let $\left(\Omega_{l}\right)_{l \in \mathbb{N}}$ be a family of smooth domains contained in $\Omega$ such that $\Omega_{l}\nearrow \Omega$ as $l \rightarrow \infty.$ Since  we do not have the desired regularity on the coefficients of $\veind{k}$ we will first consider the boundary value problem in each smooth domain $\Omega_{l}$ with the following regularized equation,
\begin{eqnarray*}
F^{\delta}(\veind{k})= \inf_{
\begin{split}
& \: a_{ij} \in \qq  \\ 
& \: \lbrack a_{ij} \rbrack
 \in \mathcal{A}_{\lambda, \Lambda}\\
\end{split}} \left( a_{ij}\Dij \veind{k}- \veind{k} \: \left(g(\uue) \star \rho_{\delta} \right)\right)=0, \qquad  k=1, \ldots, d,
\end{eqnarray*}
where $\rho_{\delta} $ is an approximation of the identity. And we  will consider suitable  boundary data  that converges to the original boundary data when we approach the original domain.

So now, by Proposition \ref{prop: general existence}, there exist a unique solution 
$(\veind{k})^{\delta}_{l} \in \cspace{2}{}{(\Omega)}\cap \cspace{0}{}{(\overline{\Omega)}}$ 
 with 
\begin{eqnarray*}
(\veind{k})^{\delta}_{l}=(\phi_{k})_{l} \quad \mbox{on} \; \partial \Omega_{l}, \quad k=1, \cdots,d.
\end{eqnarray*}

Taking the limit in $l$ we obtain the existence of $(\veind{k})^{\delta},$ solutions in $\Omega$ with boundary data equals to $\phi_{k}.$ 
From Proposition \ref{prop: holder up boundary} we also have that there exists a universal exponent $\gamma$ such that $(\veind{k})^{\delta}$ is \holder continuous up to the boundary, $(\veind{k})^{\delta} \in \cspace{\gamma}{}{(\overline{\Omega})} $.

As we have $F^{\delta} \rightarrow F$ uniformly due to the properties of identity approximation, we can conclude that 
\begin{eqnarray*}
F^{\delta}(\veind{k}^{\delta}) \stackrel{\delta \rightarrow 0}{\longrightarrow} F(\veind{k})
\end{eqnarray*}
by Proposition  \ref{prop: compactness }, this is, that  $\veind{k} \in \cspace{\gamma}{}{(\overline{\Omega})}$ is the solution for our problem $ F(\veind{k})=0 $ in $\Omega,$ for all $k$.

We need also to prove that for each $k,$ $0\leq \veind{k}\leq \sup_{i} \norm{\phi_{i}}{\lspace{\infty}{}}.$

Suppose, by contradiction that  there exists $x_{0}$ such that $ \veind{k} (x_{0})< 0.$ Since $ \veind{k} = \phi_{k}$   on $\partial \Omega$ and $\phi_{k}$ is non-negative, $x_{0}$ must be an interior point. But, if $ \veind{k} $ has an interior minimum attained at $x_{0}$ then there exists a paraboloid $P$ touching $\veind{k}$ from above at $x_{0}$ and such that
\begin{eqnarray*}
\puccin( P(x_{0})) > 0. 
\end{eqnarray*}
Since, $g(\uue)\geq 0,$ on the other hand we have that
\begin{eqnarray*}
 \puccin (\veind{k}(x_{0}))= \veind{k}(x_{0}) \: g(\uue(x_{0})) \leq 0,
\end{eqnarray*}
and so we have a contradiction. 

Analogously, suppose, by contradiction that  there exists $x_{0}$ such that $ \veind{k} (x_{0})> \sup_{i} \norm{\phi_{i}}{\lspace{\infty}{}}.$ Then, by the same reason as before $x_{0}$ must be an interior point and 
\begin{eqnarray*}
\puccin( \veind{k} (x_{0})) < 0. 
\end{eqnarray*}
Since, $g(\uue)\geq 0,$ we have that
\begin{eqnarray*}
 \puccin (\veind{k}(x_{0}))= \veind{k}(x_{0}) \: g(\uue(x_{0})) \geq 0,
\end{eqnarray*}
and so we have a contradiction. 

\item {\it $T$ is continuous:} Let us assume that for each fixed $\epsilon$ we have  $\left((\ueind{1})_{n}, \cdots, (\ueind{d})_{n}\right) \rightarrow \left(\ueind{1}, \cdots, \ueind{d}\right)$ in $[\cspace{}{}{(\Omega)}]^{d}$ meaning that  when $n $ tends to $\infty$,
\begin{eqnarray*}
\max_{ 1<i<d} \norm{(\ueind{i})_{n} - \ueind{i} }{L^{\infty}{}} \rightarrow 0.
\end{eqnarray*}
We need to prove that 
\begin{eqnarray*}
\norm{T\left((\ueind{1})_{n}, \cdots, (\ueind{d})_{n}\right)- T\left(\ueind{1}, \cdots, \ueind{d}\right)}{[C(\Omega)]^{d}} \rightarrow 0
\end{eqnarray*}
when $n\rightarrow \infty$. Since,
\begin{eqnarray*}
&& T\left((\ueind{1})_{n}, \cdots, (\ueind{d})_{n}\right)=((\veind{1})_{n}, \cdots, (\veind{d})_{n})
\end{eqnarray*}
if we prove that there exists a constant $C,$ independent of $i,$  so that we have the estimate 
$$
\norm{(\veind{i})_{n}- \veind{i}}{L^{\infty}} \leq C  \max_{j} \norm{(\ueind{j})_{n}- \ueind{j}}{L^{\infty}{}}
 $$
the result follows. For all $x \in  \Omega$, let $\omega_{n}$ be the function
$$
\omega_{n}(x)= (\veind{i})_{n} (x)- \veind{i}(x),
$$ 
and suppose by contradiction that there exists $y \in \Omega$ such that 
\begin{eqnarray}
\label{contradition h}
\omega_{n}(y) > r^{2} K  \max_{j} \norm{(\ueind{j})_{n}- \ueind{j}}{L^{\infty}{}},
\end{eqnarray}
for some large $K >0,$ where $r $ is such that $\Omega \subset B_{r}(0)$. 
We want to prove that this is impossible if $K$ is sufficiently large. Let $h_{n}$ be the concave radially symmetric function,
\begin{eqnarray*}
h_{n}(x)=\gamma (r^{2}- \abs{x}^{2}),
\end{eqnarray*}
with $\gamma=K \max_{i} \norm{(\ueind{i})_{n}- \ueind{i}}{L^{\infty}{}} $. Observe that:
\begin{enumerate}
\item $h_{n}(x)=0$  on $\partial B_{r}(0);$ 
\item $h_{n}(x)\leq r^{2}\,K\, \max_{j} \norm{(\ueind{j})_{n}- \ueind{j}}{L^{\infty}{}} $ for all $x$ in $B_{r}(0)$;
\item  $0= \omega_{n}(x) \leq h_{n}(x)$  on $\partial \Omega,$ since $(\veind{i})_{n}$ and $\veind{i}$ are solutions with the same boundary data.
\end{enumerate}
So,  since we are assuming (\ref{contradition h}), there exists a negative minimum of $ h_{n}-\omega_{n}$. Let $x_{0}$ be the point where the negative  minimum value of $ h_{n}-\omega_{n}$ is attained, $h_{n}(x_{0})-\omega_{n}(x_{0}) \leq 0.$
Then  we have that for any matrix $A \in \mathcal{A}_{\lambda, \Lambda},$
\begin{eqnarray*}
a_{ij} \Dij \left( \left(h_{n}-\omega_{n}\right)(x_{0})\right)\geq 0 \quad \mbox{and} \quad a_{ij}\Dij h_{n}(x)=\sum_{i} - 2\, \gamma \, a_{ii} \leq 0.
\end{eqnarray*}
Moreover,
\begin{eqnarray*}
\puccip(\omega_{n}) &\geq& \puccin((\veind{i})_{n}(x)) -\puccin(\veind{i})\\
& =& \frac{1}{\epsilon}\left( \left( (  \veind{i})_{n} - \veind{i} \right)  \sum_{j \neq i} (\ueind{j})_{n}  -  \veind{i} \sum_{j \neq i} \left(   \ueind{j}-(\ueind{j})_{n} \right) \right)\\
&\geq&    \frac{1}{\epsilon}\left( \left( (  \veind{i})_{n} - \veind{i} \right)  \sum_{j \neq i} (\ueind{j})_{n}  -  \veind{i} (d-1) \norm{(\ueind{j})_{n}- \ueind{j}}{L^{\infty}{}}\right)
\end{eqnarray*}
adding and substrating $ \frac{1}{\epsilon}\veind{i} \sum_{j \neq i} (\ueind{j})_{n} .  $ 
Then,  if $a_{ij}^{\omega}$ are the coefficients associated to $\omega_{n}, $ i.e.   
\begin{eqnarray*}
a_{ij}^{\omega_{n}}\Dij \omega_{n} (x) = \puccip(\omega_{n}),
\end{eqnarray*}
then
\begin{eqnarray*}
\qquad 0 &\leq &a_{ij}^{\omega_{n}}\Dij \left( h_{n}-\omega_{n}\right)(x_{0})\\  &\leq &
\sum_{i} - 2\, \gamma \, a^{\omega_{n}}_{ii}(x_{0}) - \frac{1}{\epsilon}\left( \left( (  \veind{i})_{n} - \veind{i} \right)(x_{0})  \sum_{j \neq i} (\ueind{j})_{n}(x_{0})   -  \veind{i} (x_{0}) (d-1) \norm{(\ueind{j})_{n}- \ueind{j}}{L^{\infty}{}}\right) \\
&\leq& \sum_{i} - 2\, K  \max_{j} \norm{(\ueind{j})_{n}- \ueind{j}}{L^{\infty}{}}\, a^{\omega_{n}}_{ii}(x_{0}) + \frac{1}{\epsilon}  \veind{i}(x_{0})(d-1) \norm{(\ueind{j})_{n}- \ueind{j}}{L^{\infty}{}},
\end{eqnarray*}
because $0 < h_{n}(x_{0})\leq \omega_{n}(x_{0})$ and $\sum_{j \neq i} (\ueind{j})_{n}(x_{0})\geq 0$ and so
\begin{eqnarray*}
- \frac{1}{\epsilon}\left( (  \veind{i})_{n} - \veind{i} \right)(x_{0})  \sum_{j \neq i} (\ueind{j})_{n}(x_{0})  \leq 0.
\end{eqnarray*}
Taking $K > \frac{d-1}{\lambda \epsilon} \sup_{i} \norm{ \phi_{i}}{L^{\infty}}, $ we obtain that
\begin{eqnarray*}
0 &\leq &a_{ij}^{\omega_{n}}\Dij \left( h_{n}-\omega_{n}\right)(x_{0}) < 0
\end{eqnarray*} 
 which is a contradition.

\item {\it $T(\sigma) $ is precompact:} This a consequence of (1), since the solutions to the equation are \holder  continuous, and this set is precompact in $\sigma$.

\end{enumerate}
This concludes the proof.
\end{proof}

\section{Regularity of solutions (Theorem \ref{theorem: regularity of solutions})}
\label{sec: holder}

Following the proof presented in \cite{caffarelli_geometry_2009},
we will prove the decay of the oscillation of the solution under certain particular hypotheses and then use this result to prove the uniform $C^{\alpha}$ regularity for a solution of the system of equations.

\subsection{Some Lemmas}

We first establish some conditions under which the maximum of a function in a smaller ball is lower.

\begin{lemma}
\label{lemma: cases of decay}
Let $\uue$ be a solution of Problem (\ref{eq: main problem}). Let $M_{i}= \max_{x \in B_{1}(0)} \ueind{i} (x)$  and $O_{i} = \mathrm{osc}_{x \in B_{1}(0)} \ueind{i}(x)$.  
If for some positive constant $\gamma_{0}$ one of the following hypotheses is verified
\begin{enumerate}
\item  $\abs{\{x \in B_{\frac{1}{4} } (0):  \ueind{i}(x) \leq M_{i} - \gamma_{0} O_{i} \}} \geq \gamma_{0} $
\item $\abs{\{x \in B_{\frac{1}{4} } (0): \puccin (\ueind{i}(x) )\geq \gamma_{0} \: O_{i}\}} \geq \gamma_{0} $
\item $\abs{\{x \in B_{\frac{1}{4} } (0): \puccin ( \ueind{i}(x) ) \geq \gamma_{0} \: \ueind{i}(x) \}}  \geq \gamma_{0} $
\end{enumerate}
then there exist a small positive constant $c_{0}= c_{0} (\gamma_{0})$ such that  the following decay estimate is valid:
$$
\max_{x \in B_{\frac{1}{4}}(0)} \ueind{i} (x) \leq M_{i} - c_{0} \: O_{i}
$$
and so 
$$
\osc_{B_{\frac{1}{4}}(0)} \ueind{i} (x) \leq \tilde{c}_{0} \: O_{i}, \qquad \mbox {with} \: \quad \tilde{c}_{0}<1.
$$
\end{lemma}

\begin{proof} $\:$ 
\begin{enumerate}
\item By contradiction, assume that for all $c_{0}$ small, exist a point $x_{0} \in B_{\frac{1}{4} } (0)$ such that 
$$
\ueind{i} (x_{0}) \geq M_{i} - c_{0} \: O_{i}
$$
and let 
$$
v_{i}(x) = \frac{M_{i} -\ueind{i} (x)  }{O_{i}}.
$$
$v_{i}$ satisfy the following properties, with $f_{i}(x)= - \frac{ \ueind{i} (x) }{\epsilon c_{0} \: O_{i}} \sum_{j} \ueind{j} (x) $:
\begin{enumerate}
\item  $\inf_{B_{\frac{1}{2}}(0) } \frac{ \vind{i}(x)}{c_{0}  } \leq \inf_{B_{\frac{1}{4}}(0) } \frac{ \vind{i}(x)}{c_{0}  } \leq 1  $
\item $ \puccin (\frac{\vind{i}}{c_{0}}) \leq  f_{i}(x)$
\item $\vind{i}(x) \geq 0$ in $B_{1}(0)$
\item $\norm{f_{i}^{+}}{\lspace{n}{}{}}=0$  since $f_{i}^{+} (x)= 0.$
\end{enumerate}

Then we can apply the $L^{\epsilon}-$ Lemma stated as Lemma \ref{le lemma},  to conclude that 
$$\abs{\{x \in B_{\frac{1}{4}}(0):  \frac{ \vind{i}(x)}{c_{0}  } \geq    t   \}} \leq d t^{-\delta}, $$
for $d, \delta$ universal constants and for all $t >0.$
If $t = \frac{\gamma_{0}}{c_{0}}$  then we have
 $$\abs{\{x \in B_{\frac{1}{4}}(0):  \vind{i}(x)  \geq   \gamma_{0}      \}} \leq d (\frac{\gamma_{0}}{c_{0}})^{-\delta}$$
  But, taking in account the  hypothesis we have:
\begin{eqnarray*}
  \gamma_{0} \leq \abs{\{x \in B_{ \frac{1}{4} } (0):  \vind{i}(x)  \geq   \gamma_{0}      \}} \leq d (\frac{\gamma_{0}}{c_{0}})^{-\delta} 
\end{eqnarray*}
 Since $c_{0}$ is arbitrary, if $c_{0}^{\delta} < \frac{\gamma_{0}^{\delta +1}}{d}$ we have  the  contradiction.

\item 
\label{case b}
If $\ueind{i}$ is  a solution for $\puccin(\ueind{i})= f(\ueind{i})$ there exists a symmetric matrix with coefficients $a_{ij}(x)$ with 
$$
\lambda \abs{\xi}^{2}\leq a_{ij}\xi_{i}\xi_{j} \leq \Lambda \abs{\xi}^{2}
$$
such that 
\begin{eqnarray*}
a_{kl}(x) D_{kl}\ueind{i}(x) = f(\ueind{i}(x))
\end{eqnarray*}
For that particular matrix consider the linear problem with measurable coefficients in $B_{1}(0):$
$$
L_{a} \left( v(x)\right) = g(x) 
$$
and let $G(x, \cdot)$ be the respective Green function on  $B_{1}(0)$ such that for all $x \in B_{1}(0), $
\begin{eqnarray*}
v(x)= -\int_{B_{1}(0)} G(x,y) g(y) \dy + \mbox{boundary terms}
\end{eqnarray*}
(see page 415 in \cite{Escauriaza_apriori_1993}).
We conclude 
\begin{eqnarray*}
M_{i} - \ueind{i}(x) \geq -\int_{B_{1}(0)} G(x,y)(-f(\ueind{i}))\dy \geq \int_{A_{i}} (G(x,y))(f(\ueind{i}))\dy,
\end{eqnarray*}
for $A_{i}:=\{x \in B_{\frac{1}{4} } (0): \puccin (\ueind{i}(x) )\geq \gamma_{0} \: O_{i}\},$ since 
\begin{eqnarray*}
a_{kl}(x) D_{kl} (M_{i}-\ueind{i}(x)) = -f(\ueind{i}(x)),
\end{eqnarray*}
and the boundary values are positive and $ G(z, y) f(\ueind{i}(y)) \geq 0$ for all $y.$ Since by hypothesis 
$$\abs{A_{i}}:=\abs{\{x \in B_{\frac{1}{4} } (0): \puccin (\ueind{i}(x) )\geq \gamma_{0} \: O_{i}\}} \geq \gamma_{0} $$ and due to  Fabes-Strook inequality, (see Lemma \ref{lemma: F S ineq} and Theorem 2 in \cite{fabes_lp_1984} for more details)
\begin{eqnarray*}
  \int_{A_{i}}  G(z, y) \gamma_{0} \: O_{i} \dy \geq \gamma_{0} \: O_{i} \left( \frac{\abs{A_{i}}}{\abs{B_{\frac{1}{4}}(0)}}\right)^{\gamma}  \int_{B_{\frac{1}{4}}(0)}  G(z, y)  \dy. 
\end{eqnarray*}
We claim that:
\begin{equation}
\label{claim green}
\frac{1}{2n\Lambda} \int_{B_{\frac{1}{4}}(0)}  G(z, y) 2n\Lambda \dy \geq C.
\end{equation}
So again by hypothesis and due to the claim (that we will prove later) we have
\begin{eqnarray*}
\gamma_{0} \: O_{i} \left( \frac{\abs{A_{i}}}{\abs{B_{\frac{1}{4}}(0)}}\right)^{\gamma}  \int_{B_{\frac{1}{4}}(0)}  G(z, y)  \dy  \geq \gamma_{0} \: O_{i} \left( \frac{\gamma_{0}}{\abs{B_{\frac{1}{4}}(0)}}\right)^{\gamma}  C
\end{eqnarray*}
so 
$$
M_{i}- \ueind{i}(x) \geq \gamma_{0} \: O_{i} \left( \frac{\gamma_{0}}{\abs{B_{\frac{1}{4}}(0)}}\right)^{\gamma}  C= c_{0} O_{i}.
$$
with $c_{0}<1.$

To prove claim (\ref{claim green}) we argue that 
\begin{eqnarray*}
 \int_{B_{\frac{1}{4}}(0)}  G(z, y) 2n\Lambda \dy \geq \int_{B_{\frac{1}{4}}(0)}  G(z, y) 2 \left(\sum_{i} a_{ii}\right) \dy= \frac{1}{4^{2}}- \abs{z}^{2} \geq C
\end{eqnarray*}
for $z $  interior, since
$a_{ij}\Dij (\frac{1}{4^{2}}- \abs{z}^{2} ) = -2 \sum_{i} a_{ii}$ and $-2n\Lambda \leq -2 \sum_{i} a_{ii} \leq -2n \lambda$
\item  Let 
\begin{eqnarray*}
A_{i}=\left\{x \in B_{\frac{1}{4} } (0): \puccin ( \ueind{i}(x) ) \geq \gamma_{0} \: \ueind{i}(x) \right\}
 \end{eqnarray*}
and 
\begin{eqnarray*}
H_{i}=\left\{x \in A_{i}:  \ueind{i}(x)  \leq \frac{M_{i}}{2} \right\} \end{eqnarray*}
and consider the two possible cases: 
$(a) \: \abs{A_{i} \backslash H_{i}} \geq \frac{1}{2} \abs{A_{i}}$ and $(b) \: \abs{A_{i} \backslash H_{i}} < \frac{1}{2} \abs{A_{i}}$.

\begin{enumerate}
\item If $ \abs{A_{i} \backslash H_{i}} \geq \frac{1}{2} \abs{A_{i}}$ then as
\begin{eqnarray*}
 \left\{  x \in A_{i} \backslash H_{i}:   \puccin ( \ueind{i}(x) ) \geq \gamma_{0} \: \frac{M_{i}}{2} \right\} \subset  \left\{x \in  B_{\frac{1}{4} } (0): \puccin ( \ueind{i}(x) ) \geq \frac{\gamma_{0} \: O_{i}}{2} \right\} 
\end{eqnarray*}
since $\frac{O_{i}}{2} \leq \frac{M_{i}}{2},$  we can conclude that
\begin{eqnarray*}
\abs{\left\{x \in B_{\frac{1}{4} } (0): \puccin ( \ueind{i}(x) ) \geq \gamma_{0} \:\frac{ O_{i}}{2} \right\} }  \geq  \abs{A_{i}\backslash H_{i}} \geq \frac{1}{2}\abs{A_{i}} \geq \frac{\gamma_{0}}{2}.
\end{eqnarray*}
 Then we have the decay by (\ref{case b}) with $\gamma_{0}$ replaced by $\frac{\gamma_{0}}{2}.$
\item If $ \abs{A_{i} \backslash H_{i}} < \frac{1}{2} \abs{A_{i}}$ then as
\begin{eqnarray*}
H_{i}=\left\{x \in A_{i}:  \ueind{i}(x)  \leq \frac{M_{i}}{2} \right\} \subset \left\{x \in B_{\frac{1}{4} } (0):  \ueind{i}(x)  \leq M_{i} - \beta_{0} O_{i}\right\} 
\end{eqnarray*}
for $\beta_{0} \leq \frac{M_{i}}{2 O_{i}}$ and  
\begin{eqnarray*}
\abs{H_{i}} = \abs{A_{i}\backslash \left(A_{i} \backslash H_{i}\right)} = \abs{A_{i}}- \abs{\left(A_{i} \backslash H_{i}\right)} \geq \frac{\abs{A_{i}}}{2} \geq \frac{\gamma_{0}}{2} 
\end{eqnarray*}
we have
\begin{eqnarray*}
\abs{ \left\{x \in B_{\frac{1}{4} } (0):  \ueind{i}(x)  \leq M_{i} - \tilde{ \gamma_{0}} O_{i}\right\} } \geq \frac{ \tilde{ \gamma_{0}}}{2}
\end{eqnarray*}
for $\tilde{\gamma_{0}}=\min(\beta_{0}, \gamma_{0})$.  The decay follows by (1).
\end{enumerate}
\end{enumerate}
\end{proof}

Next Lemma states that if all the oscillations are tiny compared to just one that remains big, then the largest oscillation has to decay due to an increase of the minimum in a smaller ball.

\begin{lemma}
\label{lemma: oscillations argument}
Let $\uue$ be a solution of Problem (\ref{eq: main problem}) in $B_{1}(0)$. Let 
$$O^{1}_{i} = \osc_{x \in B_{1}(0)}  \ueind{i}(x). $$
 Assume that for some $\delta >0,$ sufficiently small, 
\begin{eqnarray*}
\sum_{j\neq 1} O^{1}_{j} \leq \delta O^{1}_{1}.
\end{eqnarray*}
then  $O^{1}_{1}$ must decay in $B_{\frac{1}{2}}(0)$, that is, there exist $\mu <1$ such that
\begin{eqnarray*}
O^{\frac{1}{2}}_{1} \leq \mu O^{1}_{1}.
\end{eqnarray*}
\end{lemma}

\begin{figure}[h]
\begin{center}
\includegraphics[scale=0.4]{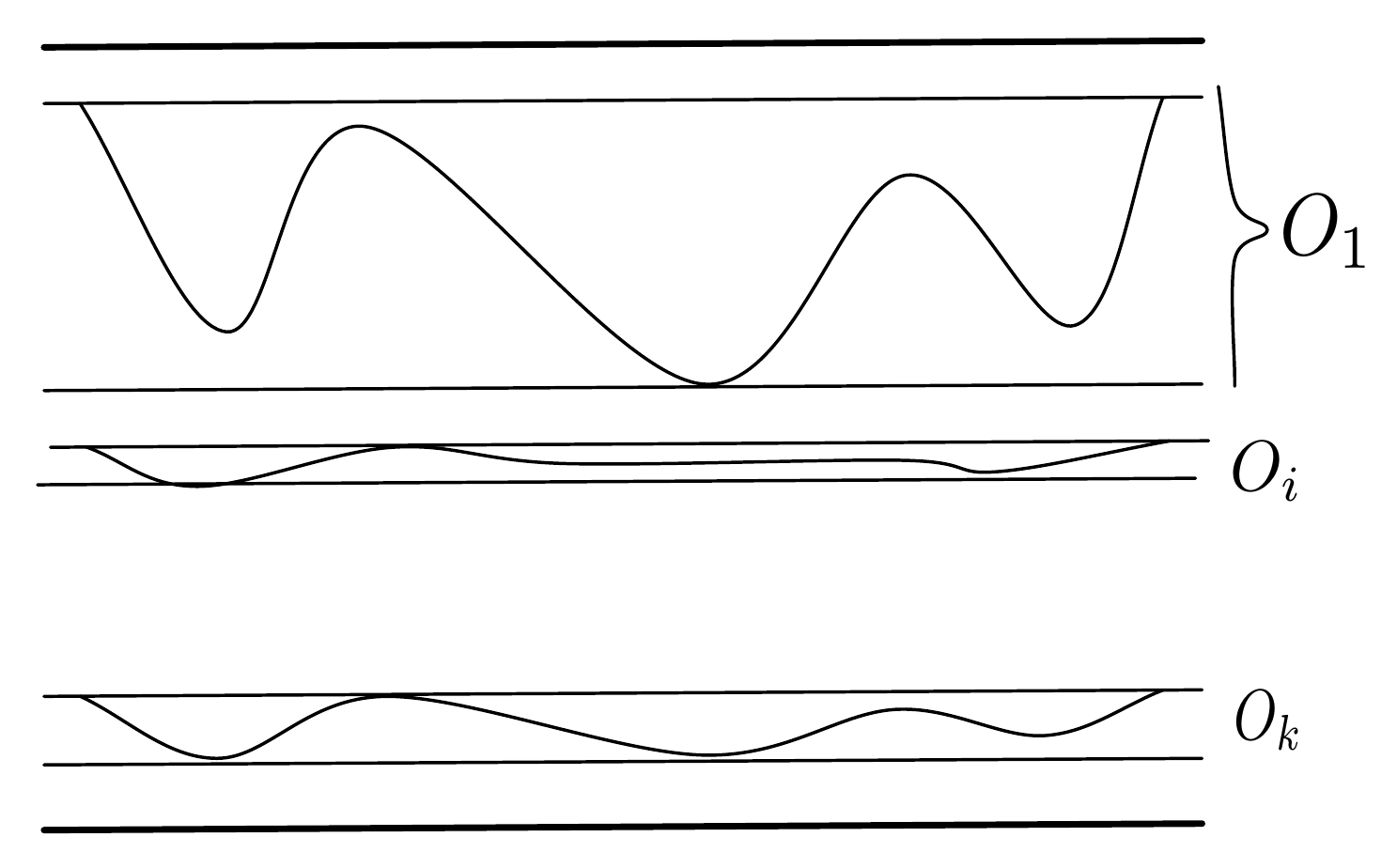}
\caption{Illustration of our hypotheses about oscillations. }
\end{center}
\end{figure}

\begin{proof}

Let $\omega$ be  the  solution of the problem
\begin{eqnarray*}
\left\{
\begin{split}
&\puccin (\omega(x)) =0, \quad x \in B_{1}(0) \\
&\omega (x)= \ueind{1} (x), \quad x \in \partial B_{1}(0) \\
\end{split}
\right.
\end{eqnarray*}
Since 
$\ueind{1}-\omega$
 is a subsolution for the positive Pucci extremal operator 
\begin{eqnarray*}
\puccip(\ueind{1}-\omega) \geq \puccin(\ueind{1}) + \puccip (-\omega) \geq 0,
\end{eqnarray*}
and 
$  \ueind{1} - \omega + \sum_{i\neq 1} (M_{i}- \ueind{i} ) $ 
is a supersolution for the negative Pucci extremal operator,
\begin{eqnarray*}
\puccin \left(    \ueind{1} + \sum_{i\neq 1} ( M_{i}- \ueind{i})  - \omega \right) &\leq& \puccin\left(   \ueind{1} + \sum_{i\neq 1} ( M_{i} - \ueind{i}) \right) +\puccip (-\omega) \\
&\leq & \puccin(   \ueind{1} )+ \sum_{i\neq 1}\puccip ( M_{i} - \ueind{i} )\\
&  \leq &\puccin (  \ueind{1} ) - \sum_{i\neq 1}\puccin  (\ueind{i})  \leq 0,
\end{eqnarray*}
by the maximum principle for viscosity solutions and our hypotheses we get,
\begin{eqnarray*}
\ueind{1} (x) \leq \omega (x) \leq \ueind{1}(x) +  \sum_{i \neq 1}(M_{i} - \ueind{i} (x)) \leq \ueind{1} + \delta O^{1}_{1}.
\end{eqnarray*}
Since $\omega \in S^{\star}(\lambda, \Lambda, 0),$ $\omega$ decays, namely
\begin{eqnarray*}
\osc_{x \in B_{\frac{1}{2}}(0)} \omega(x) \leq \mu  \: \osc_{x \in B_{1}(0)} \omega(x)
\end{eqnarray*}
 and  from this inequalities, since
 $$
 \max \omega- \min \omega \leq \max \ueind{1} + \delta O^{1}_{1} -\min \omega \leq  \max \ueind{1} + \delta O^{1}_{1} - \min \ueind{1},
 $$
 we can conclude that
$$
\osc_{B_{r}(0)} \omega \leq  \osc_{B_{r}(0)} \ueind{1}  +\delta O^{1}_{1}.
$$
In a analogous way, 
$$
\osc_{B_{r}(0)} \ueind{1} \leq \osc_{B_{r}(0)} \omega +\delta O^{1}_{1} .
$$
So, then,
\begin{eqnarray*}
\osc_{B_{\frac{1}{2}}(0)} \ueind{1} \leq \osc_{B_{\frac{1}{2}}(0)} \omega +\delta O^{1}_{1} \leq \mu  \: \osc_{x \in B_{1}(0)} \omega(x)+\delta O^{1}_{1} \leq  \mu  \left(   \osc_{B_{1}(0)} \ueind{1}  +\delta O^{1}_{1}\right)+\delta O^{1}_{1}.
\end{eqnarray*}
Simplifying, 
\begin{eqnarray*}
\osc_{B_{\frac{1}{2}}(0)} \ueind{1} \leq   \left(   \mu   \left(  1  +\delta\right) +\delta \right) O^{1}_{1},
\end{eqnarray*}
which concludes the proof by taking $\delta$ sufficiently small.
\end{proof}

\subsection{Proof of Theorem \ref{theorem: regularity of solutions}}

In this section we finally present the proof of Theorem  \ref{theorem: regularity of solutions} stated in section \ref{sec: main results}.

\begin{proof}

We prove  this theorem iteratively.
We  will prove that the oscillation of $\uue$ will  decay, by some constant factor $\tilde{\mu} <1,$ independent of $\epsilon$  when it goes from $B_{1}(0)$ to $B_{\lambda}(0)$  
for some $\lambda<1$ also independent of $\epsilon$. Meaning that, we will prove that  there exist two constants  $0<\lambda,\tilde{\mu} <1,$  independent of $\epsilon,$ such that for all $i=1,\ldots, d,$ we have 
\begin{eqnarray*}
\osc_{x \in B_{\lambda}(0)} \ueind{i}(x) \leq \tilde{\mu}  \: \osc_{x \in B_{1}(0)} \ueind{i}(x).
\end{eqnarray*}
The $C^{\alpha}$ regularity of each function will follow from here in  a standard way using Lemma 8.23, in \cite{gilbarg_elliptic_2001}. Since, what matters is the ratio between oscillations 
\begin{eqnarray*}
\frac{\osc_{B_{\lambda}(0)} \ueind{i}(x)}{\osc_{B_{1}(0)}\ueind{i}(x)} \leq \tilde{\mu},
\end{eqnarray*}
the result will hold true if we prove this decay for the normalized functions, which satisfy the same equation with a different value of $\epsilon.$ Thus, consider    $\ueind{1}, \: \ueind{2}, \cdots, \ueind{d} $ solutions  of  Problem (\ref{eq: main problem}) on $B_{1}(0)$ and the renormalized functions $\overline{\ueind{1}}, \cdots, \overline{\ueind{d}}, $ 
$$
\overline{\ueind{i}} (x) =  \rho\: \ueind{i } \left(x\right), \quad x \in B_{1}(0),  \qquad i=1, \cdots,d,
$$
with $\rho=\frac{1}{max_{x \in B_{1}(0), k=1, \ldots, d}\:  \ueind{k} (x)}. $
These functions are bounded from above by one and satisfy 
\begin{eqnarray*}
\puccin(\overline{\ueind{i}}(x)  ) = \underbrace{\frac{1 }{\rho \: \epsilon }}_{\overline{\epsilon}} \overline{\ueind{i} } (x) \sum_{i \neq l}\overline{ \ueind{l}}(x),  \qquad i=1, \cdots,d.
\end{eqnarray*}

Briefly, the iterative process consists of the following.  We  prove  that in $B_{\frac{1}{4}}(0)$ at least the largest oscillation decays. Without loss of generality consider $\overline{\ueind{1}}$ the function with the largest oscillation. Then, there exists $\overline{\mu}<1$ such that
$$
\osc_{x \in B_{\frac{1}{4}}(0)} \overline{\ueind{1}}(x) \leq \overline{\mu}  \: \osc_{x \in B_{1}(0)}\overline{ \ueind{1}}(x).
$$
Then, we  consider the renormalization  by the dilation in $x:$ 
$$
\overline{\overline{\ueind{i}}} (x) =\overline{\rho} \: \overline{\ueind{i }} \left(\frac{1}{4}x\right), \quad x \in B_{1}(0),  \qquad i=1, \cdots,d,
$$
with 
\begin{eqnarray*}
\overline{\rho} =\frac{1}{\max_{x \in B_{\frac{1}{4}}(0), k=1, \ldots, d}\:  \overline{\ueind{k}} (x)}>1.
\end{eqnarray*}
Observe that  these  functions  are solutions of  the system 
\begin{eqnarray*}
\puccin\left(\overline{\overline{\ueind{i}}} (x)  \right) = \underbrace{ \frac{1}{\overline{\rho}\, 4^{2} \: \overline{\epsilon} }}_{\overline{\overline{\epsilon}}} \overline{\overline{\ueind{i} } }(x) \sum_{i \neq l}\overline{ \overline{ \ueind{l}}} (x),  \qquad i=1, \cdots,d.
\end{eqnarray*}
\begin{figure}[h]
\begin{center}
\includegraphics[scale=0.4]{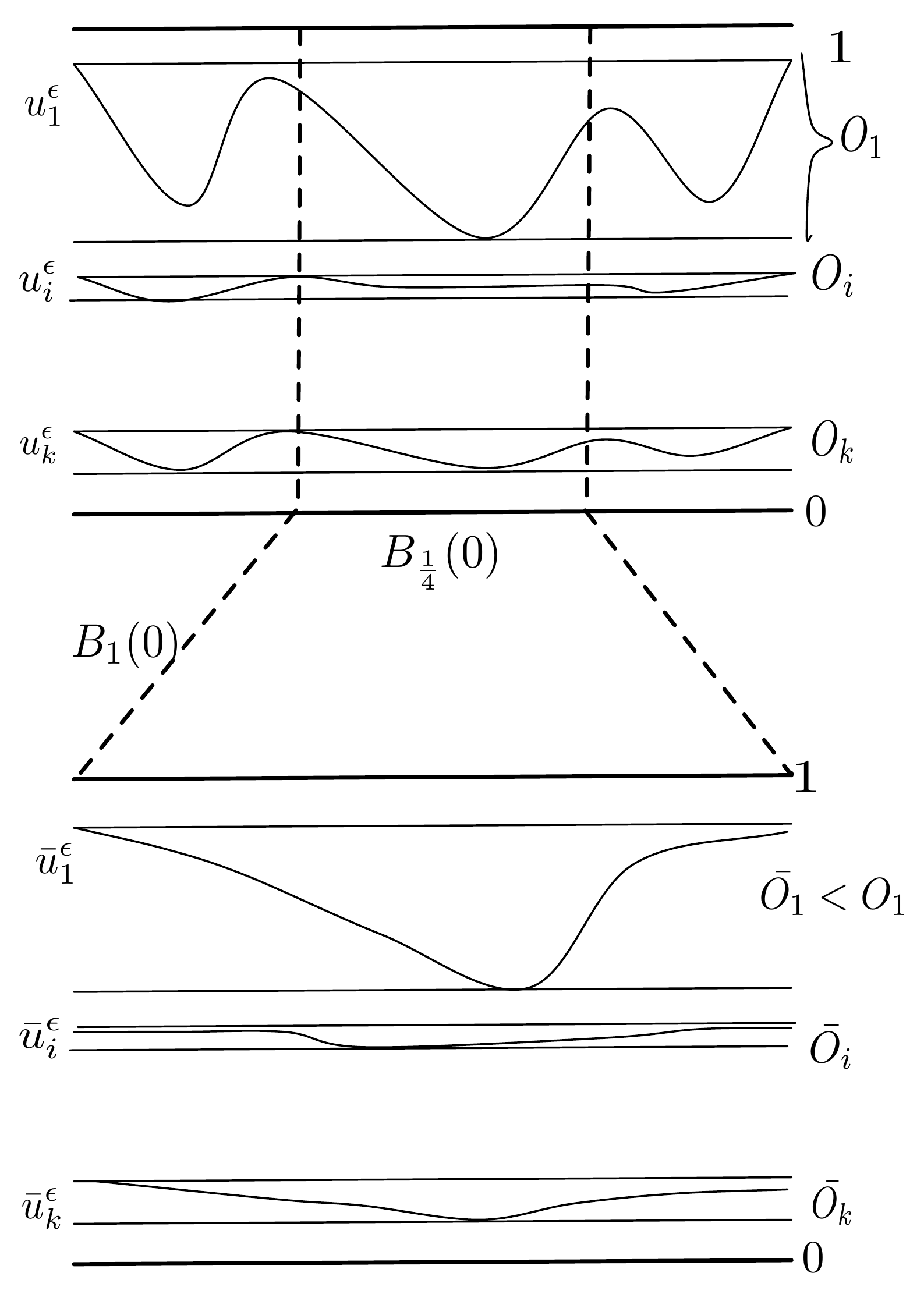}
\caption{Decay iteration. After the renormalization the oscillation of the first function decays while the others remain the same. In the original configuration we register that decay and we proceed with the next renormalization.}
\label{fig:Decay iteration}
\end{center}
\end{figure}

So, basically, they are the solutions of an equivalent system with  a different $\epsilon,$ still defined on $B_{1}(0).$

We start all over to prove that we have again the reduction of  the next largest oscillation,   when we are in $B_{\frac{1}{4}}(0).$ We call the new function with the largest oscillation $\overline{\overline{\ueind{i}}}$. So we have that there exists $ \overline{ \overline{\mu}} <1,$ independent of $\epsilon,$ such that 
\begin{eqnarray*}
 \osc_{x \in B_{\frac{1}{4}}(0)} \overline{\overline{ \ueind{i}}}(x)  \leq \overline{ \overline{\mu}}  \:   \osc_{x \in B_{1}(0)}\overline{ \overline{\ueind{i}}}(x) \Rightarrow \osc_{x \in B_{\frac{1}{4^{2}}}(0)}  \overline{\ueind{i }} (x ) \leq \overline{ \overline{\mu}}\: \osc_{x \in B_{\frac{1}{4}}(0)} \overline{ \ueind{i } }(x ) .
\end{eqnarray*}
Considering what has happened in the previous step we can have several scenarios. Either we have the reduction of the oscillation from $B_{1}(0)$ to $B_{\frac{1}{4^{2}}}(0)$ of just one of the functions, or two or more.
(Later, when there is no possible confusion, the function with the largest oscillation at each iteration will always be denoted by  $\overline{\overline{\ueind{1}}}$).

We repeat this process taking the renormalizations
$$
\overline{\overline{\ueind{i}}} (x) =\frac{1}{max_{z \in B_{1}(0), j}\: \overline{ \ueind{j}} \left(\frac{1}{4^{k}}z\right)}  \: \overline{\ueind{i }} \left(\frac{1}{4^{k}}x\right), \quad x \in B_{1}(0),  \qquad i=1, \cdots,d,
$$
until eventually we have the reduction of all the oscillations, or until the largest oscillation is much larger then the other oscillations. If that is the case, then eventually we will have that, for $O_{j}=\osc_{B_{1}} \overline{\overline{\ueind{j}}} (x),$
$$
\sum_{j\neq 1} O_{j} \leq \delta O_{1}.
$$
In that case, using the Lemma \ref {lemma: oscillations argument}, $O_{1}$ must also decay and we have the result. Thus, after repeating this iterative process a finite number of times we obtain that for some $\lambda <1,$ $\tilde{\mu}<1$ and every $i= 1, \ldots, d,$
$$
 \osc_{x \in B_{\lambda}(0)} \overline{ \ueind{i }} (x ) \leq \tilde{\mu} \: \osc_{x \in B_{1}(0)} \overline{ \ueind{i } }(x ).
$$
 
For simplicity, we will still refer to the renormalized functions  bounded by one, as  $\ueind{i}, $ and also $ \frac{1}{\overline{\rho}\, 4^{k} \:\overline{ \epsilon} } $
 will still be denoted by $\epsilon$ in each new step. Although,  $\epsilon $ can be bigger than the one in the first steps,  depending on the number of steps needed, it will eventually remain smaller than one  since the renormalization after each dilation  will be a multiplication by a factor smaller than one.

In what follows $\ueind{j_{0}}$ denotes at each renormalization, the function that achieves the maximum value $1$ and $\ueind{1}$ is  the function that has maximum oscillation. Naturally, they can be or not the same function. 

Below follows the proof of decay for the renormalized functions in all possible cases.

\underline{Case 1:} 
Let us assume that $\frac{1}{\epsilon} >1$.  \\
  We are going to use the following argument:  observe that,  if there exists $k$  such that:
$$
\abs{\{x \in B_{\frac{1}{4}}: \ueind{k} (x )\geq \gamma_{0}\}}\geq \gamma_{0},  
$$
 then since  for any $j \neq k$
$$
\puccin(\ueind{j} (x)  ) =  \frac{1}{ \epsilon} \ueind{j}  (x) (\ueind{k}+ \cdots)  
$$
we have,
$$
\abs{\left\{x \in B_{\frac{1}{4}}: \puccin (\ueind{j}) \geq \gamma_{0} \ueind{j}\right\}}\geq \gamma_{0}.
$$
And so by (3) in  Lemma  \ref{lemma: cases of decay} we can conclude that $O_{j},$ for all $j\neq k$ decays. 
Let $0<\gamma<\frac{1}{4}$ be a fixed constant.
\begin{enumerate}
\item   If  $\max_{x} \ueind{1} \geq \gamma$ and $O_{1}$ does not decay then by  (1) in Lemma \ref{lemma: cases of decay} we can conclude that there exists $\gamma_{0}$  a positive small constant such that
$$
\abs{\left\{x \in B_{\frac{1}{4}}:\ueind{1} (x) \geq \max_{x} \ueind{1}  - \gamma_{0}O_{1} \geq \frac{\gamma}{2} \right\}}\geq 1- \gamma_{0} \geq \frac{1}{2},
$$
so then as we observed before, we have for all $j\neq 1$  that,
$$
\abs{\left\{x \in B_{\frac{1}{4}}: \puccin (\ueind{j}) \geq  \frac{\gamma}{2} \ueind{j}\right\}}\geq \frac{\gamma}{2}.
$$
And so by (3) in  Lemma  \ref{lemma: cases of decay} we can conclude that $O_{j},$ for all $j\neq 1$ decays.
\item  If  $\max_{x} \ueind{1} < \gamma$ then, since $O_{1}\leq \max_{x} \ueind{1} < \gamma,$ all oscillations are smaller than $\gamma$ (see Figure \ref{fig:case2}).

\begin{figure}[h]
\begin{center}
\includegraphics[scale=0.4]{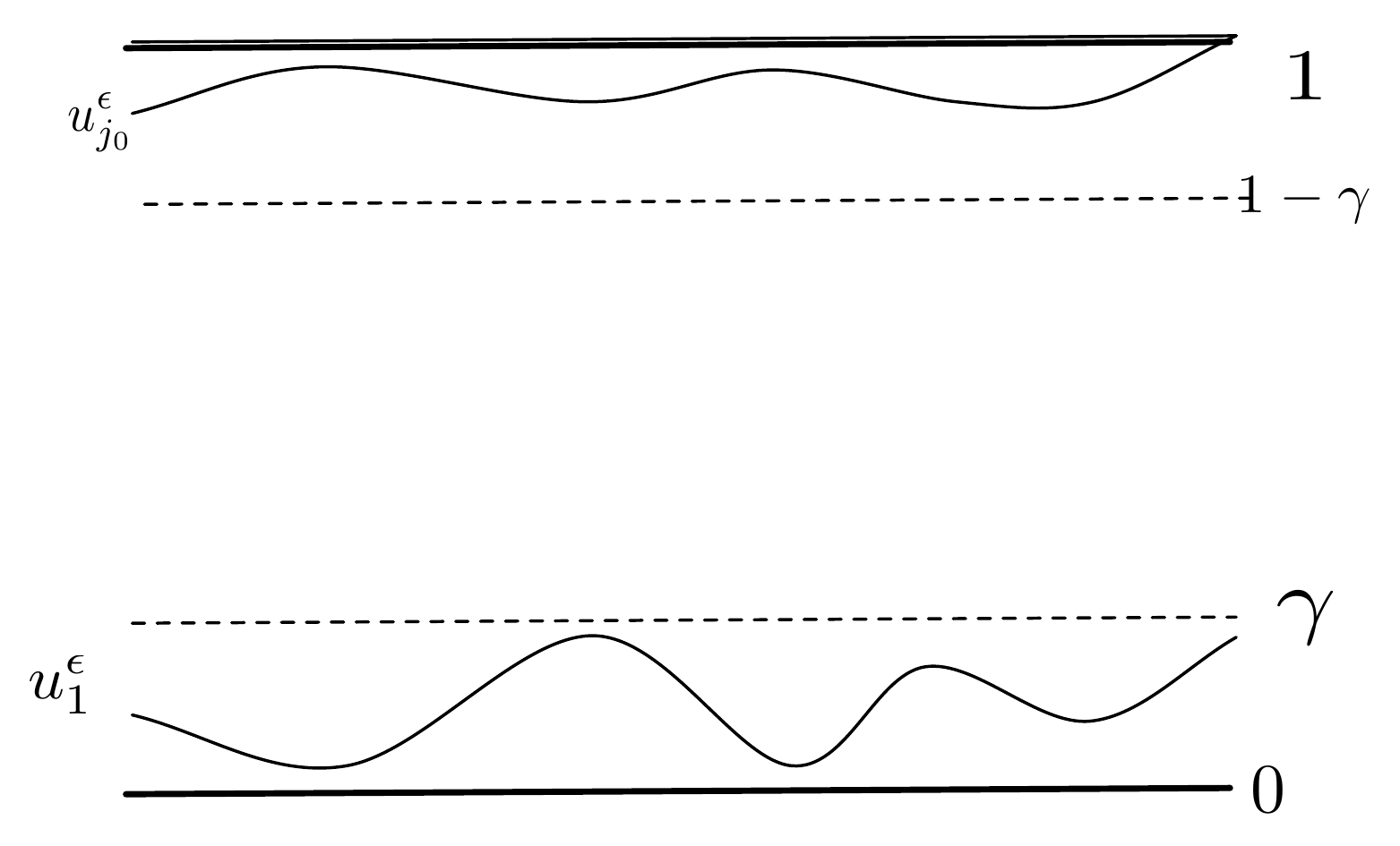}
\caption{All oscillations are smaller than $\gamma$. }
\label{fig:case2}
\end{center}
\end{figure}

Then in particular the function $\ueind{j_{0}}$ has oscillation smaller than $\gamma.$ Thus, 
$$
\ueind{j_{0}} (x) \geq 1-\gamma
$$
 and so for all $j \neq j_{0},$ we have for all $j\neq j_{0}, $
$$
 \puccin (\ueind{j}) \geq (1-\gamma) \ueind{j}.
$$
So again by (3) in Lemma  \ref{lemma: cases of decay} we can conclude that for all $j\neq j_{0}, $ $O_{j}$ decays. In particular, the largest oscillation decayed.
\end{enumerate}

\underline{Case 2:} Let us assume that $\frac{1}{\epsilon} <1$. Let $\theta=\frac{1}{\epsilon}.$  \\
Observe that, since all the functions are bounded from above by one and are positive, we have that for all $i \neq i_{0},$ and $i_{0}$ an arbitrarily  fixed indix,
\begin{equation}
\label{major estimate1}
   \theta \ueind{i_{0}}  \ueind{i} \leq \puccin(\ueind{i}) \leq  \theta \ueind{i} (d-1).
\end{equation}
In a more general point of view we have for any $i=1,\ldots, d,$
\begin{equation}
\label{major estimate2}
0 \leq  \puccin(\ueind{i}) \leq  d,
\end{equation}
We will use one or another expression according to convenience. There are two possible cases: either (1) $ \frac{1}{4} \leq O_1\leq1  $ or (2) $O_{1} <\frac{1}{4}$. We prove that in the first case all the big oscillations decay. So after a finite number of steps, (and  if all the functions did not decay yet in the mean time) we will have that all oscillations are less than $\frac{1}{4}.$ In this case, since the function that attains the maximum will have also oscillation less than $\frac{1}{4}$, we can prove that all the other functions decay. In this way, we will  eventually be again in the case of 
$$
\sum_{j\neq j_{0}} O_{j} \leq \delta O_{j_{0}}.
$$
and as before we have the result. The proof of the decay in each case follows.

\begin{enumerate}
\item Assume that $ \frac{1}{4} < O_{1} \leq 1. $ There exists an interior point  $y \in B_{\frac{1}{4}}(0),$  such that 
$$ \min \ueind{1}+ \frac{1}{8} \leq \ueind{1}(y) \leq \max \ueind{1} -\frac{1}{8}.$$  By (\ref{major estimate2}) we conclude that the equation for $\ueind{1}$ has right hand side continuous and bounded. Then, by regularity,  Proposition \ref{prop: interior holder},   we can conclude that there exists a universal constant  $C$ such that for all $x \in B_{\frac{1}{16\,C}}(y)$  
$$
\abs{\ueind{1}(x)-\ueind{1}(y)}\leq \norm{\nabla \ueind{1}}{L^{\infty}} \frac{1}{16\,C} \leq \norm{ \ueind{1}}{C^{1,\alpha}} \frac{1}{16\,C}\leq   \frac{1}{16}.
$$
And so,  for all $x  \in B_{\frac{1}{16\, C}}(y)$ we have that 
 $$\ueind{1} (x) \leq   \max \ueind{1} -\frac{1}{16}.$$
 Then by (1) of Lemma \ref{lemma: cases of decay}, we can conclude that
\begin{eqnarray*}
\osc_{B_{\frac{1}{4}}(0)}  \ueind{1} \leq \mu \:  \osc_{B_{1}(0)} \ueind{1}   \qquad \mbox {with} \: \quad \mu<1.
\end{eqnarray*}
With this argument we can conclude that all the functions with oscillations bigger than $\frac{1}{4}$ decay. Moreover, repeating the same argument we can conclude that all the functions with oscillations bigger than $\frac{1}{8}$ decay.

\item  Assume that  $O_{1} \leq \frac{1}{4}.$  Observe that in this case $\ueind{j_{0}}(x) \geq \frac{3}{4}$ for all $x \in B_{1}(0)$

Then for all $i \neq  j_{0}$ the function $\ueind{i}$ satisfies the equation
$$
   \frac{3}{4} \theta \, \ueind{i} \leq \puccin(\ueind{i}) \leq  \theta \, \ueind{i} (d-1)
$$
Meaning that,
$$
 \puccin(\ueind{i}) \sim \theta \, \ueind{i} 
$$
Consider $v_{i},$ $i \neq j_{0},$ the renormalized functions $\ueind{i},$ by the renormalization
$$
v_{i}(x) =\frac{1}{\max_{x \in B_{1}(0)}\:  \ueind{i} (x)}  \: \ueind{i } \left(x\right).
$$
Observe that $v_{i}$ has maximum $1$ and satisfies in $B_{1}(0),$
$$
   \theta \frac{3}{4} v_{i} \leq \puccin(v_{i}) \leq  \theta v_{i} (d-1).
$$

For each function $v_{i}$, $i\neq j_{0},$ we can have  two situations. Either (a) $v_{i} (x) \geq \frac{1}{2}$ for all $x$ or (b) $v_{i} (x) < \frac{1}{2}$ for some $x$ in $B_{1}(0).$    In both cases we will prove that $v_{i},$ $i \neq j_{0}$ decay in $B_{\frac{1}{4}}(0).$

\begin{enumerate}
\item If $v_{i} (x) \geq \frac{1}{2}$  for all $x \in B_{1}(0).$ Then 
$$
O_{i}= \osc_{x \in B_{1}(0)}v_{i} (x)  \leq \frac{1}{2}.
$$ 
{\it{We claim that: there exists a universal constant $N$ such that $\theta \leq N O_{i}.$}}

Assuming that the claim is true, observe that 
$$
 0 \leq \puccin(v_{i}) \leq d N O_{i},
$$
which implies that we can consider the function $\omega$ with oscillation 1  in $B_{1}(0)$ defined by  
\begin{eqnarray*}
\omega(x)= \frac{v_{i}(x)-\min_{z \in B_{1}(0)} v_{i}(z)}{O_{i} }.
 \end{eqnarray*}
 and that satisfies in $B_{1}(0)$
$$
0< \puccin\left(\omega \right) \leq d N .
$$ 
By  regularity, Proposition \ref{prop: interior holder}, there exists a universal constant $C$ depending on $N$ and $d$ such that,
\begin{eqnarray*}
\abs{\omega(x)-\omega(y)}\leq \norm{ \omega}{C^{1,\alpha}}\abs{x-y}\leq C(N,d)\abs{x-y}.
\end{eqnarray*}

 Assuming that $\omega $ didn't decay, let $y \in B_{\frac{1}{4}}(0)$ be such that $\omega (y) < \frac{1}{2}.$ Then, if \mbox{$\sigma= \min  \left( {\dist(y, \partial B_{\frac{1}{4}}),\frac{1}{4 C(N,d)}}\right)$}, for all $x \in B_{\sigma}(y)$
 $$
\abs{\omega(x)-\omega(y)}\leq \frac{1}{4} 
$$
And so, for $x \in B_{\sigma}(y),$
$$
\omega(x) \leq \frac{1}{4}+ \frac{1}{2} \leq \frac{3}{4}.
$$
Then by (1) of Lemma \ref{lemma: cases of decay}, we can conclude that for all $ i \neq j_{0}$, there exist $\mu<1$
$$
\osc_{B_{\frac{1}{4}}(0)}  v_{i} (x) \leq \mu \:  \osc_{B_{1}(0)}  v_{i} (x) \Rightarrow  \osc_{B_{\frac{1}{4}}(0)}  \ueind{i}(x) \leq \mu \: \osc_{B_{1}(0)}   \ueind{i}(x) (x).
$$

{{\it Proof of the claim:}}  consider by contradiction that $
\theta > 2 n \lambda \, 24\:  O_{i}.
$
So, for all $ x \in B_{1}(0),$
$$
 \frac{3}{4} v_{i} (x)  \theta > \frac{3}{8} \theta > 2 n \lambda \,  9\, O_{i}.
$$
Observe that, if   $m=\min_{z \in B_{1}(0)} v_{i}(z),$  $m < t < m + \osc_{B_{1}(0)} v_{i} (x)$ is a  positive constant to be chosen later   and 
$$
P(x)=  8\, O_{i} \abs{x}^{2} +  t
$$
we have that for any $x,$ and any $t,$
$$
\puccin (v_{i}) \geq \frac{3}{4} v_{i} (x)  \theta  > 2 n \lambda \, 9\:  O_{i} > \puccin (P(x))= 2 n \lambda \,8\, O_{i} .
$$
and so $P$ can not touch $v_{i}$ from above at any $y,$ since it would contradict  that the $v_{i}$ is a subsolution.
But  since for $t=m$ and $\abs{x}=1,$ 
$$
P(x) \geq m + 8 O_{i} \quad \mbox{and} \quad P(0)=m,
$$
$P$ crosses the function $v_{i}$, and so it is possible to find $t$ such that $P$ would touch the function $v_{i}$ from above at some point, which is a contradiction.
\item  If  $v_{i} (y) < \frac{1}{2}$ for some interior point $y$ of $B_{\frac{1}{4}}(0)$ we proceed as before and use regularity.  If the function never attains a value less than $\frac{1}{2}$ in $B_{\frac{1}{4}}(0)$ then  the function has decayed. 
\end{enumerate}
\end{enumerate}
\end{proof}

With  this uniform bound in the  Banach space \cspace{\alpha}, one can conclude that these sequence converges uniformly (up to a subsequence) to a vector of functions $\uu. \: $

\section{Characterization of limit problem: Proof of Theorem \ref{theorem: limit problem}}
\label{sec: limit}

In this section we will assume without loss of generality that $\lambda=1.$
Observe that if $\uue$ is a viscosity  solution of Problem (\ref{eq: main problem}) then there exists a subsequence still indexed by $\epsilon$ and function $\uu \in \left( \cspace{\alpha}{} \right)^{d}$ such  that  
$$
 \uue \rightarrow \uu \qquad \mbox{ uniformly.}
$$

The following  Lemma  characterizes the Laplacian of the limit solution.

\begin{lemma}
\label{lemma: proof supersolution}
If  $\uu \in \left( \cspace{\alpha}{} \right)^{d}$  is the limit of a solution of (\ref{eq: main problem})  then 
 $ \Delta u_{i} $ are positive measures.
\end{lemma}

\begin{proof}
 Let $\phi$ be positive a test function. Then
\begin{eqnarray*}
0\leq \int \puccin (\ueind{i}) \phi \leq \int \D\ueind{i} \phi = \int \D \phi \ueind{i}  \rightarrow  \int \D \phi \uind{i} = \int \phi \D \uind{i} 
\end{eqnarray*}
and so $\D \uind{i}$ is a positive distribution which implies that it is a nonnegative Radon measure.
\end{proof}

Now we are ready to prove Theorem \ref{theorem: limit problem} (stated in Section \ref{sec: main results}).

\begin{proof} (1)
 Observe that
\begin{eqnarray*}
\puccin \left(    \ueind{i} - \sum_{k\neq i}  \ueind{k} \right) &\leq & \puccin(\ueind{i} ) + \puccip \left(- \sum_{k\neq i}  \ueind{k}\right) \leq \puccin(\ueind{i} ) -\puccin \left(\sum_{k\neq i}  \ueind{k}\right) \\
&\leq& \puccin(\ueind{i} ) -\sum_{k\neq i} \puccin(\ueind{k})\leq 0.
\end{eqnarray*}
As $ \ueind{i} - \sum_{k\neq i}  \ueind{k} \rightarrow \uind{i} - \sum_{k\neq i}  \uind{k}  $  when $\epsilon \rightarrow 0$ uniformly and $\overline{S}(\lambda, \Lambda, 0)$ is closed under uniform convergence,
$ \uind{i} - \sum_{k\neq i}  \uind{k}   $
is a supersolution of \puccin :
$$
\puccin \left(    \uind{i} - \sum_{k\neq i}  \uind{k} \right) \leq 0
$$

(2) If $\uind{i}( x_{0})= \alpha_{0}>0$ for any  $i$, $i=1, \ldots, d$ then for $\delta < \frac{\alpha_{0}}{2}$ there exists an $\epsilon_{0}$ such that for $\epsilon < \epsilon_{0}$,  $ \frac{ \alpha_{0}}{2}< \alpha_{0} - \delta <\ueind{i}(x_{0})< \alpha_{0} + \delta< \frac{ 3 \alpha_{0}}{2}$. So 
 by \holder   continuity there exist $h>0$ such that:
  \begin{enumerate}
\item[a)] $\vert \ueind{i} (y)- \ueind{i}(x_{0}) \vert  \leq \frac{\alpha_{o}}{4}$ in $B_{2h}(x_{0});$\\
\item[b)] $\ueind{i} (y) > \frac{\alpha_{0}}{4}$ in a ball of radius $2h$ and center $x_{0}.$
\end{enumerate}
Observe that applying Green's Identity with a function $u$ and the fundamental solution  \begin{eqnarray*}
\tilde{ \Gamma}(x_{0}) = \frac{1}{n \omega_{n}(2-n)} \left( \frac{1}{\abs{x-x_{0}}^{n-2}} - \frac{1}{\abs{2\,h}^{n-2}} \right),
\end{eqnarray*}
we obtain the inequality
\begin{eqnarray*}
  h^{2}  \fint_{ B_{h}(x_{0})}  \D u \:  \diff{x}  \leq C
 \fint_{\partial B_{2h}(x_{0})}\left( u(x) - u(x_{0})\right) \diff{S}, 
\end{eqnarray*}
where $C$ is a constant just depending on $n$. From the equation for $\ueind{i}$, we obtain
\begin{eqnarray*}
 \fint_{ B_{h}(x_{0})}   \frac{\alpha_{0}\,\sum_{k \neq i}\ueind{k}}{4 \epsilon} \diff{x} &\leq&
\fint_{ B_{h}(x_{0})}   \frac{\ueind{i}\, \sum_{k \neq i}\ueind{k}}{\epsilon} \diff{x} 
=
 \fint_{ B_{h}(x_{0})}  \puccin (\ueind{i} ) \diff{x} \\
 &\leq& 
 \fint_{ B_{h}(x_{0})}  \D \ueind{i} \diff{x} 
 \leq
 \frac{C}{h^{2}} \fint_{\partial B_{2h}(x_{o})}\left( \ueind{i} (y) - \ueind{i} (x_{o})\right) \diff{S}\\
 & \leq& 
 \frac{C \,\alpha_{0}}{4\,h^{2}}.
\end{eqnarray*}
Thus, 
$$
 \fint_{ B_{h}(x_{o})}   \frac{\alpha_{0}\,\sum_{k \neq i}\ueind{k}}{2} \diff{x} \leq
 \frac{\epsilon\, C \,\alpha_{0}}{4\,h^{2}},
$$
which implies that, 
$$ \fint_{ B_{h}(x_{o})} \sum_{k \neq i}\ueind{k} \rightarrow 0$$ when \etz.   
By subharmonicity, 
$$
\sum_{k \neq i}\ueind{k}(x_{0}) \leq \fint_{ B_{h}(x_{0})} \sum_{k \neq i}\ueind{k} \rightarrow 0
$$
and so $\sum_{k \neq i}\uind{k} (x_{0})=0$  and this proves the result.

(3) To prove that $\puccin  (u_{i}) = 0,$ when $ u_{i}>0$ assume the set up in the beginning of (2) for $\uind{i}(x_{0}).$ We need to prove  that 
$  \frac{\ueind{i}\, \sum_{k \neq i}\ueind{k}}{\epsilon} \rightarrow 0$ uniformly when \etz  $\:$ in order to use the closedness of $S(\lambda, \Lambda)$. Since,
$$
 \D \left( \frac{ \sum_{k \neq i}\ueind{k}}{\epsilon} \right) \geq  \puccin   \left( \frac{ \sum_{k \neq i}\ueind{k}}{\epsilon} \right) \geq     \frac{1}{\epsilon} \sum_{k \neq i} \puccin \ueind{k}   \geq  0,
$$
$ \frac{ \sum_{k \neq i}\ueind{k}}{\epsilon}$ is subharmonic. So,  if we prove that $  \frac{ \sum_{k \neq i}\ueind{k}}{\epsilon} \rightarrow 0$ in $L^{1}(B_{h}(x_{0}))$, then for \mbox{$y \in \overline{B}_{h-\delta' }(x_{o}),$} 
$$
 \frac{ \sum_{k \neq i}\ueind{k} (y)}{\epsilon} \leq \fint_{B_{\delta'}(y)} \frac{ \sum_{k \neq i}\ueind{k}}{\epsilon} \diff{x} \rightarrow 0,
$$
and so $ \frac{ \sum_{k \neq i}\ueind{k}}{\epsilon} \rightarrow 0$ convergences uniformly in a compact set contained in  $B_{h}(x_{0})$. 
Recall  that we proved in (2) that $ \sum_{k \neq i}\ueind{k} \rightarrow 0$ uniformly in a compact set contained in  $B_{h}(x_{0})$ and we have that
\begin{eqnarray*}
  \sum_{k \neq i}\ueind{k}  \rightarrow 0 \Rightarrow \D \left(\sum_{k \neq i}\ueind{k} \right)  \rightarrow 0 \: \mbox{in the sense of distributions.}
  \end{eqnarray*}
 So, since
\begin{eqnarray*}
 \frac{\alpha_{0}}{2} \frac{\, \sum_{k \neq i}\ueind{k}}{\epsilon}&  \leq & \frac{\ueind{i}\, \sum_{k \neq i}\ueind{k}}{\epsilon} = \puccin (\ueind{i}) \leq  \sum_{k \neq i} \puccin \ueind{k} \\
 &\leq& \puccin \left(\sum_{k \neq i}\ueind{k}\right) \leq \D \left(\sum_{k \neq i}\ueind{k}\right), 
 \end{eqnarray*}
  we conclude that
 $$  \frac{ \sum_{k \neq i}\ueind{k}}{\epsilon} \rightarrow 0 \: \mbox{ in} \: L^{1}.$$
Then, as we said, we have that  
   $$  \frac{ \sum_{k \neq i}\ueind{k}}{\epsilon} \rightarrow 0 \:  \mbox{uniformly in a compact set contained in} \:  B_{h}(x_{0}).$$

 As  $\ueind{i} \rightarrow u_{i}$ uniformly and are bounded, we finally conclude that
 $$
 \ueind{i}\,\frac{ \sum_{k \neq i}\ueind{k} }{\epsilon} \rightarrow  0  \: \mbox{uniformly in a compact set contained in} \: B_{h}(x_{0}).
 $$
 Proceeding analogously with $u_{k}$, $k=1, \cdots, n $ we conclude that the limit problem is
$$
\puccin  (u_{i}) = 0, \qquad u_{i}>0 \qquad i=1, \ldots, d.
$$

(4) To prove the last statement  we will construct an upper and lower barrier.

\begin{enumerate}
\item[(a)] Consider as  upper barriers the solutions of the $d$ problems, ($i=1, \ldots, d$):
$$
\puccin(u^{\star}_{i})=0 \quad  \mbox{in} \: \Omega \qquad \mbox{and} \qquad u^{\star}_{i}(x)= \phi_{i}(x)\chi_{\{\phi_{i}(x)\neq 0\}} \quad  \mbox{in} \: \partial \Omega
$$
So we have that  for all $i =1, \ldots, d, $
$$
\puccin(u^{\star}_{i})=0 \leq   \frac{1}{\epsilon} \ueind{i}\, \sum_{k \neq i}\ueind{k} = \puccin(\ueind{i})  \quad \mbox{in} \: \Omega
$$
and 
$$
u^{\star}_{i}(x)= \ueind{i}(x) \quad  \mbox{for all } \quad x \in \partial \Omega.
$$
So by the comparison principle, for all $i$ and for all $\epsilon$ we have the upper bound
$$
u^{\star}_{i}(x)\geq  \ueind{i}(x) \quad  \mbox{for all } \quad x \in  \overline { \Omega}.
$$
Taking limits in $\epsilon$ we can deduce that for all $i$
$$
\phi_{i}(x) \geq  \uind{i}(x)  \quad  \mbox{for all } \quad x \in  \partial \Omega.
$$
\item[(b)]  Now consider as  a lower barrier the solution of the problem:
$$
\puccin(\omega_{i})=0 \quad  \mbox{in} \: \Omega \qquad \mbox{and} \qquad \omega_{i}(x)= \phi_{i}(x) -\sum_{j \neq i}\phi_{j}(x) \quad  \mbox{in} \: \partial \Omega
$$
So we have that  for all $i =1, \ldots, d, $
$$
\puccin(  \ueind{i} - \sum_{i \neq j} \ueind{j}   ) \leq 0 =    \puccin(\omega_{i})  \quad \mbox{in} \: \Omega
$$
and 
$$
  \ueind{i}(x) - \sum_{i \neq j} \ueind{j}(x)= \omega_{i}(x) \quad  \mbox{for all } \quad x \in \partial \Omega.
$$
So by the comparison principle, for all $i$ and for all $\epsilon$ we have the lower bound
$$
 \ueind{i}(x) - \sum_{i \neq j} \ueind{j}(x) \geq  \omega_{i}(x) \quad  \mbox{for all } \quad x \in  \overline { \Omega}.
$$
Taking the limit in $\epsilon$ we can deduce that for all $i$
$$
\uind{i}(x) - \sum_{i \neq j} \uind{j}(x) \geq\phi_{i}(x) -\sum_{j \neq i}\phi_{j}(x)   \quad  \mbox{for all } \quad x \in  \partial \Omega.
$$
Since by hypothesis we know that $\phi_{i}$ have disjoint supports, when $\phi_{i}(x)\neq 0$, 
$$
\uind{i}(x) \geq \phi_{i}(x),
$$
and this proves the statement. 
\end{enumerate} 
\end{proof}

\section{Lipschitz regularity for the free boundary problem: proof of Theorem \ref{theorem: lips}}
\label{sec: lipschitz}

The  regularity theorem of this section  is a very important result relating the growth of one  of the functions in terms of the  distance of the function to the free boundary. In the proof we will need to use barriers, properties of subharmonic functions and the monotonicity formula introduced in \cite{alt_variational_1984}. In order to simplify the reading of the paper, the linear decay to the boundary is done in the beginning of this section and uses the construction of fundamental barriers that is done in Appendix E. The monotonicity formula is stated in Appendix \ref{apd for lips: monotonicity} and the  study of the $L^{\infty}$ decay for subharmonic functions supported in a small domains is presented in Appendix \ref{apd for lips: subhar}.

\begin{lemma}[Linear decay to the boundary normalized]
\label{wall of barriers} 
Let $v$ be a non-negative continuous function defined in $\Omega= B_{\sigma}(z_{0})\backslash B_{1}(0),$ where  $\sigma \leq \frac{1}{2}$ and $z_{0}$ is, without loss of generality, a point on $\partial B_{1}(0)$ with $\frac{z_{0}}{\abs{z_{0}}}=e_{n}, $ $e_{n} $ the unit vector. Assume that,
\begin{enumerate}
\item $\puccip (v) \geq 0$ in $\Omega,$
\item $ v(x) \leq U \sigma $ in $\Omega,$
\item $ v(x)=0 $ on $\partial B_{1}(0).$
\end{enumerate}
 then, there exists a universal constant $\tilde{C},$ $\tilde{C}=\frac{8}{5} \frac{\alpha}{\frac{1}{5}-\left(\frac{1}{5}\right)^{\alpha+1}}, $ such that,
 $$
 v(x)\leq \tilde{C}U \dist(x, \partial B_{1}), 
 $$
 when $x \in S_{\sigma}$ with $S_{\sigma}:= \left( B_{1+\frac{\sigma}{4}}(0)\backslash B_{1}(0)\right) \cap \{ x=(x',x_{n}): \abs{x'-z'_{0}}<\frac{\sigma}{2} \}.$ In particular,
  $$
 v(x)\leq \tilde{C} \, U \dist(x, \partial B_{1}), 
 $$
  when $x \in B_{\frac{\sigma}{4}}(z_{0}).$ 
 \end{lemma}

\begin{proof} Consider the function $v$ extended by zero to all $B_{\sigma}(z_{0}).$ Observe that the extension still satisfies the hypotheses. We will take a barrier $\phi$ as in Lemma \ref{lemma: barrier super} with $r=\frac{5 \sigma}{8},$ $\frac{a}{b}=\frac{1}{5}$ and $M= U \frac{8}{5},$ that will be used as a model and will be sliding  tangentially along $\partial B_{1}(0) $ in order to construct a wall of barriers (see Figure \ref{fig:barriers_wall_cropped}).
With that purpose, take a family of balls $\{ B_{\frac{\sigma}{8}}(y)\}_{y}$ such that  $y \in B_{\frac{ 3 \sigma}{8}} (z_{0})\cap  \partial B_{1-\frac{\sigma}{8}}(0) .$  So the balls $B_{\frac{\sigma}{8}}(y)$ are tangent to $\partial B_{1}(0) $ and are inside $B_{1}(0),$ where $v$ is zero. 
 For each ball consider $\phi$ is such that: 
 \begin{enumerate}
 \item[(a)] $\phi (x) = U\sigma$ for $x \in \partial B_{\frac{5\sigma}{8}}(y);$
  \item[(b)] $\phi (x) =0 $ for $x \in \partial B_{\frac{\sigma}{8}} (y);$
 \item[(c)] $\puccip (\phi) \leq 0$ in $B_{ \frac{5 \sigma}{8}   }(y) \backslash B_{\frac{\sigma}{16}}(y);$
 \item[(d)] $\dd{ \phi }{\nu }(x)= U \frac{8}{5} \frac{\alpha}{\frac{1}{5}-\left(\frac{1}{5}\right)^{\alpha+1}}$ when $\abs{x-y}=\frac{\sigma}{8}.$
 \end{enumerate}
 Note that  $\cup_{y} B_{ \frac{5 \sigma}{8}   }(y) \subset B_{\sigma}(z_{0})$ and so for  all $y$ defined previously and for $x \in \partial \left( \cup_{y} B_{ \frac{5 \sigma}{8}   }(y)\right)$
we know by hypothesis that   
 $$v(x)\leq U\sigma.$$ We now apply  the comparison principle  for each barrier depending on $y $ and respective ring $B_{ \frac{5 \sigma}{8}   }(y) \backslash B_{\frac{ \sigma}{8}}(y),$  
 since $v$ is a subsolution and $\phi $
 a supersolution for $\puccip,$  and we obtain that 
 \begin{eqnarray*}\phi(x) \geq v(x) , \quad   x \in \partial B_{ \frac{5 \sigma}{8}   }(y) \cup \partial B_{\frac{\sigma}{8}}(y)   \Rightarrow  \phi(x) \geq v(x) , \quad  x \in B_{ \frac{5 \sigma}{8}   }(y) \backslash B_{\frac{\sigma}{8}}(y).
\end{eqnarray*}
Hence,  repeating this for all $y$ we obtain that
$$
v(x) \leq \phi(x),
$$
for all $ x \in   S_{\sigma}= \left( B_{1+\frac{\sigma}{4}}(0)\backslash B_{1}(0)\right) \cap \{ x=(x',x_{n}): \abs{x'-z'_{0}}<\frac{\sigma}{2} \}.$
Taking in account that $\phi$ is radially concave, we  also obtain that,
\begin{equation}
\label{bound: surface for v}
v(x) \leq U \frac{8}{5} \frac{\alpha}{\frac{1}{5}-\left(\frac{1}{5}\right)^{\alpha+1}} \dist (x, \partial B_{1}(0)).
\end{equation}
For the final remark, observe that $B_{\frac{\sigma}{4}}(z_{0}) \subset S_{\sigma}.$
\end{proof}

\begin{coro}[Linear decay to the boundary]
\label{wall of barriers normalized}
Let $v$ be a non-negative continuous function defined in $\Omega_{t_{0}}= B_{\tilde{\sigma}}(t_{0}z_{0})\backslash B_{t_{0}}(0),$ where  $\tilde{\sigma} \leq \frac{t_{0}}{2}$ and $t_{0}z_{0}$ is, without loss of generality, a point on $\partial B_{t_{0}}(0)$ with $\frac{z_{0}}{\abs{z_{0}}}=e_{n}, $ $e_{n} $ the unit vector. Assume that,
\begin{enumerate}
\item $\puccip (v) \geq 0$ in $\Omega_{t_{0}},$
\item $ v(x) \leq \tilde{U} \tilde{\sigma} $ in $\Omega_{t_{0}},$
\item $ v(x)=0 $ on $\partial B_{t_{0}}(0).$
\end{enumerate}
 then, there exists a universal constant $\tilde{C},$ $\tilde{C}=\frac{8}{5} \frac{\alpha}{\frac{1}{5}-\left(\frac{1}{5}\right)^{\alpha+1}}, $ such that,
 $$
 v(x)\leq \tilde{C} \tilde{U} \dist(x, \partial B_{t_{0}}), 
 $$
 when $x \in S_{\tilde{\sigma}}$ with $S_{\tilde{\sigma}}:= \left( B_{t_{0}+\frac{\tilde{\sigma}}{4}}(0)\backslash B_{t_{0}}(0)\right) \cap \{ x=(x',x_{n}): \abs{x'-t_{0}z'_{0}}<\frac{\tilde{\sigma}}{2} \}.$ In particular,
  $$
 v(x)\leq \tilde{C} \, \tilde{U} \dist(x, \partial B_{t_{0}}), 
 $$
  when $x \in B_{\frac{\tilde{\sigma}}{4}}(t_{0}z_{0}).$ 
\end{coro}

\begin{proof}
Consider the function 
$$u(x) =\frac{1}{\tilde{U} t_{0}}v(x\,t_{0})$$ 
defined for $x \in \Omega= B_{\sigma}(z_{0})\backslash B_{1}(0)$, with $\sigma=\frac{\tilde{\sigma}}{t_{0}}$. Observe that, $u$ satisfies the hypotheses of \mbox{Lemma \ref{wall of barriers}} with $U=1$ and notice that $\sigma\leq \frac{1}{2}.$  Then,  
$$
 u(x)\leq \tilde{C} \, \dist(x, \partial B_{1}), 
 $$
 when $x \in S_{\sigma}.$ Substituting, we obtain
 \begin{eqnarray*}
\frac{1}{\tilde{U} t_{0}}v(x\,t_{0})  \leq \, \frac{\tilde{C}  }{t_{0}}\dist(t_{0}x, \partial B_{t_{0}}) \quad \Leftrightarrow \quad v(y)   \leq\, \tilde{C} \,U\dist(y, \partial B_{t_{0}}), 
 \end{eqnarray*}
for $y \in S_{\tilde{\sigma}}.$
\end{proof}

We finally present the proof of Theorem \ref{theorem: lips}.

\begin{proof}
 (1) The proof is by contradiction. Let $x_{0}, z_{0} \in \partial (\supp \: \uind{1})$ be two interior points to be characterized later. We will find a smooth function $\eta $  such that in  \mbox{$B_{\delta}(z_{0}) \subset B_{R}(x_{0})$} touches the supersolution 
$$\uind{1}-\sum_{k \neq 1} u_{k}$$
 from below at $z_{0},$  and simultaneously satisfies 
$$
\puccin (\eta (z_{0})) > 0,
$$
and this contradicts the definition of supersolution. In fact, we will find a positive universal constant $C_{0}$ such that  if $\uind{1}(y)= CR $  for some $y \in B_{R}(x_{0})$ and  $C>C_{0}$ then we have a contradiction.

To simplify the notation let $u=\uind{1}$ and $ v=\sum_{k \neq 1} u_{k}.$
Let us assume that the free boundary intersects the ball centered at the origin, $\partial (\supp \:u) \cap B_{\frac{1}{2}}(0) \neq \emptyset. $

By contradiction, assume  without loss of generality that $u$ grows above any linear function in a ball centered at $x_{0}$; more precisely, assume that for any constant  $M'$
and
$$
x_{0} \in \partial (\supp \: u) \cap B_{\frac{1}{2}}(0),
$$
there exists  $y$ such that
$$
y \in B_{R}(x_{0}) \qquad \mbox{and} \quad u (y) = M' R,
$$
where $R<\frac{1}{4}$. Also,  we may assume that $R \leq 2d$ where $d=\dist(y, \partial \supp u)>0$ (if not we can always pick another $x_{0}$) and so we have
$$
y \in B_{R}(x_{0}) \qquad \mbox{and} \quad u (y) = M' 2 d = Md.
$$
For a later purpose we fix $z_{0} \in \partial B_{d}(y)\cap \partial (\supp \: u)$ (the closest point to $y$ in the free boundary).

As $u \geq 0$ and $u \in \sspace{*}(\lambda, \Lambda, 0)$ we can apply Harnack for viscosity solutions on $B_{d}(y)$, 
$$
\sup_{B_{\frac{d}{2}}(y)}  \, u \leq c \inf_{B_{\frac{d}{2}}(y)} u
$$
and conclude that for any $x \in B_{\frac{d}{2}}(y)$
\begin{equation}
\label{eq: harnack on u1}
\frac{1}{c}\, M\, d \leq u (x) \leq c\,M\,d
\end{equation}
Observe now that if $\omega (z)$ defined on $B_{d}(y)$ is a solution to $\puccin(\omega)=0$ then due to the  invariance under translation by $y$ (\ref{eq:1}), rotation by $R$ (\ref{eq:2}) and dilation by $\frac{1}{d}$ and rescaling by $d$ (\ref{eq:3}):
\begin{equation}
\label{eq:1}  \overline{\omega}(x)=\omega (\underbrace{x+y}_{z}), \quad   x \in B_{d}(0) \Rightarrow \puccin(\overline{\omega}(x))= \puccin (\omega(z)), \quad x \in B_{d}(0), 
\end{equation}
\begin{equation}
\label{eq:2}  \overline{\omega}(x)=\omega (Rx), \quad   x \in B_{d}(0) \Rightarrow \puccin(\overline{\omega}(x))= \puccin (\omega(z)), \quad x \in B_{d}(0), 
\end{equation}
\begin{equation}
\label{eq:3}  \overline{\omega}(x)= \frac{1}{d}\omega (dx), \quad   x \in B_{1}(0) \Rightarrow \puccin(\overline{\omega}(x))= d \puccin (\omega(z)), \quad x \in B_{1}(0),
\end{equation}
$\overline{\omega}(x)$ defined in $B_{1}(0)$ is still a solution to $\puccin(\overline{\omega})=0$ with the direction $e_{n}$ as we want.

So  we will prove this theorem using  translation, rotation, dilation, and rescaling arguments on $(\uind{1}, \uind{2},\cdots, \uind{d})$. In order to simplify the notation the new functions will always be denoted by the same name. 

Consider the functions $ u, v$ and $u-v $  satisfying 
\begin{eqnarray*}
\puccin (u) \geq 0, \quad \puccin (v) \geq 0, \quad \puccin ((u-v) (x)) \leq 0, 
\end{eqnarray*} 
(see 2 in Lemma  \ref{lemma: proof supersolution}), and    defined  on an appropriate domain by translation (\ref{eq:1}), rotation (\ref{eq:2}), dilation and rescaling (\ref{eq:3}) such that $y$ is the new origin,  the direction $\frac{z_{0}-y}{\abs{ z_{0}-y}}$ is now the direction $e_{n}$ and $d$ is now 1. We will still call $z_{0}$ the point in  $\partial B_{1}(0)\cap \partial \supp u.$  
 By (\ref{eq: harnack on u1}) we now have
$$
\frac{1}{d}\frac{1}{C}\, M\, d \leq u(x) \leq \frac{1}{d} C\,M\,d, \qquad x \in B_{\frac{1}{2}}(0), 
$$
 that is,
\begin{equation}
\label{bounds}
\frac{1}{C}\, M \leq u (x) \leq  C\,M,  \qquad x \in B_{\frac{1}{2}}(0).
 \end{equation}
The rest of the proof consists of the following steps:

\underline{First step:} we will prove that for a certain  $\overline{M},$ a positive large constant,
 $$
 u(x)\geq \overline{M} \underbrace{ d(x, \partial B_{1}(0))}_{1-\abs{x}} \qquad x \in B_{1}(0) \backslash  B_{\frac{1}{2}}(0).
 $$
\underline{Second step:}
 we will prove that, for a small $\rho,$
\begin{eqnarray}
\label{linear bound}
 v(x)\leq \frac{C}{\overline{M}} \underbrace{ d(x, \partial B_{1}(0))}_{\abs{z}-1} \qquad  x \in \mathrm{S_{\frac{\rho}{2}}},
 \end{eqnarray}
 where $S_{\frac{\rho}{2}}:=\left(B_{1+\frac{\rho}{8}}(0)\backslash B_{1}(0)\right)\cap \{(x',x_{n}) \in \rr^{n}: \abs{x'-z_{0}'}<\frac{\rho}{4}\}$ (see Figure \ref{fig:barriers_wall_cropped}).

\begin{figure}[h]
\begin{center}
\includegraphics[trim=0cm 0cm 27cm 0cm, clip=true, scale=0.4]{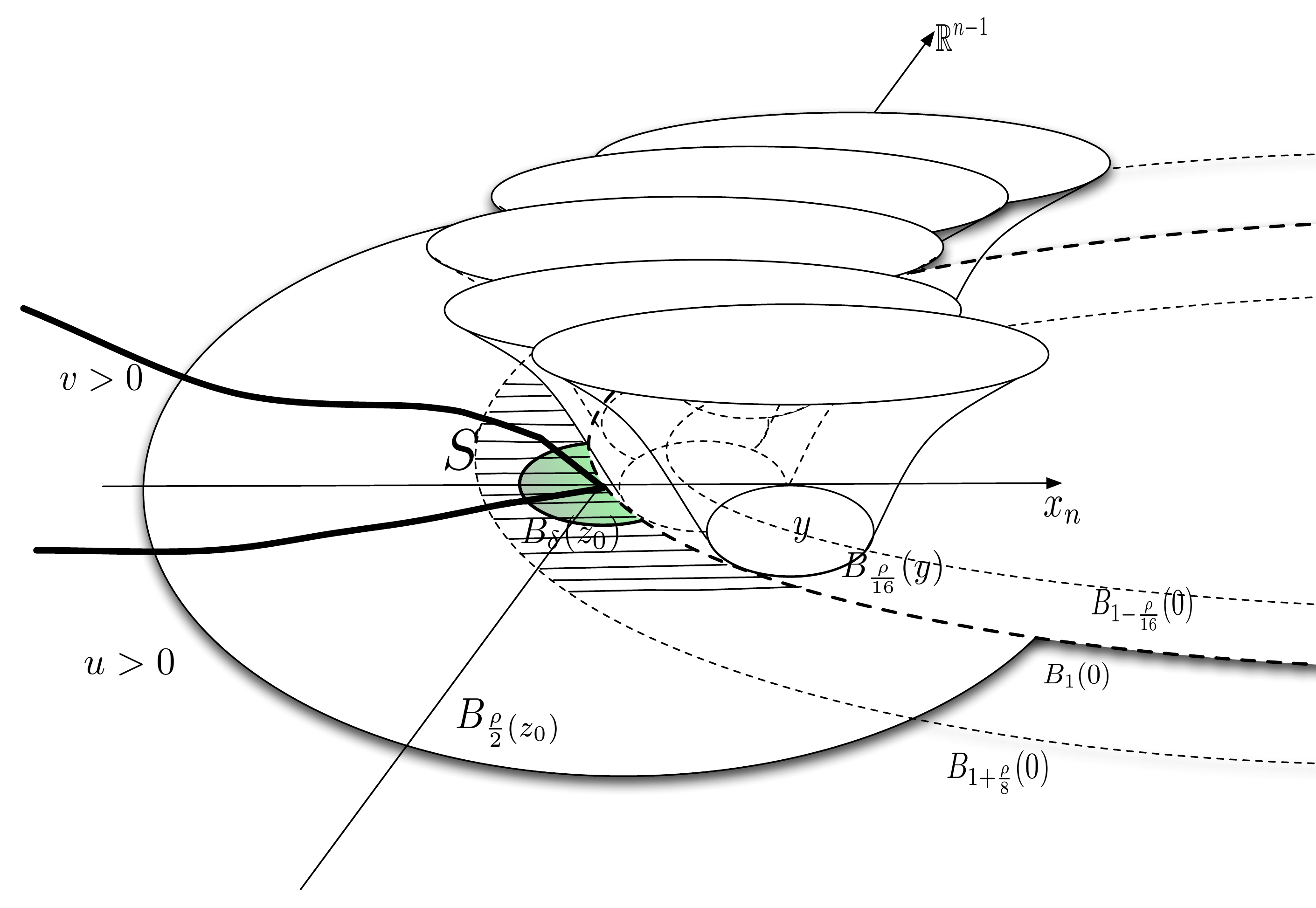}
\caption{Barriers to control $v$. In the picture $S$ is $S_{\frac{\rho}{2}}$ defined in Lemma \ref{wall of barriers}.  }
\label{fig:barriers_wall_cropped}
\end{center}
\end{figure}

 \underline{Third step:} Finally we will construct the function $\eta$ (see Figure \ref{fig:final barrier for u-v}) to obtain the contradiction, since we will have for $\delta \leq \frac{\rho}{8}$: 
\begin{itemize}
\item[-] $ \puccin (u -v)\leq 0 \qquad \mbox{for all} \: x \in B_{\delta}(z_{0})$;
\item[-] $ \eta (x) \leq \left( u -v \right) (x), \qquad \mbox{for all} \: x \in B_{\delta}(z_{0})$;
\item[-] $ \eta  (z_{0}) =  \left( u -v \right) (z_{0}) ;$
\item[-] $\eta$ is a smooth function such that $\puccin (\eta (z_{0})) > 0.$
\end{itemize}
and by definition of supersolution this is impossible and the result follows.

\begin{figure}[h]
\begin{center}
\includegraphics [scale=0.3]{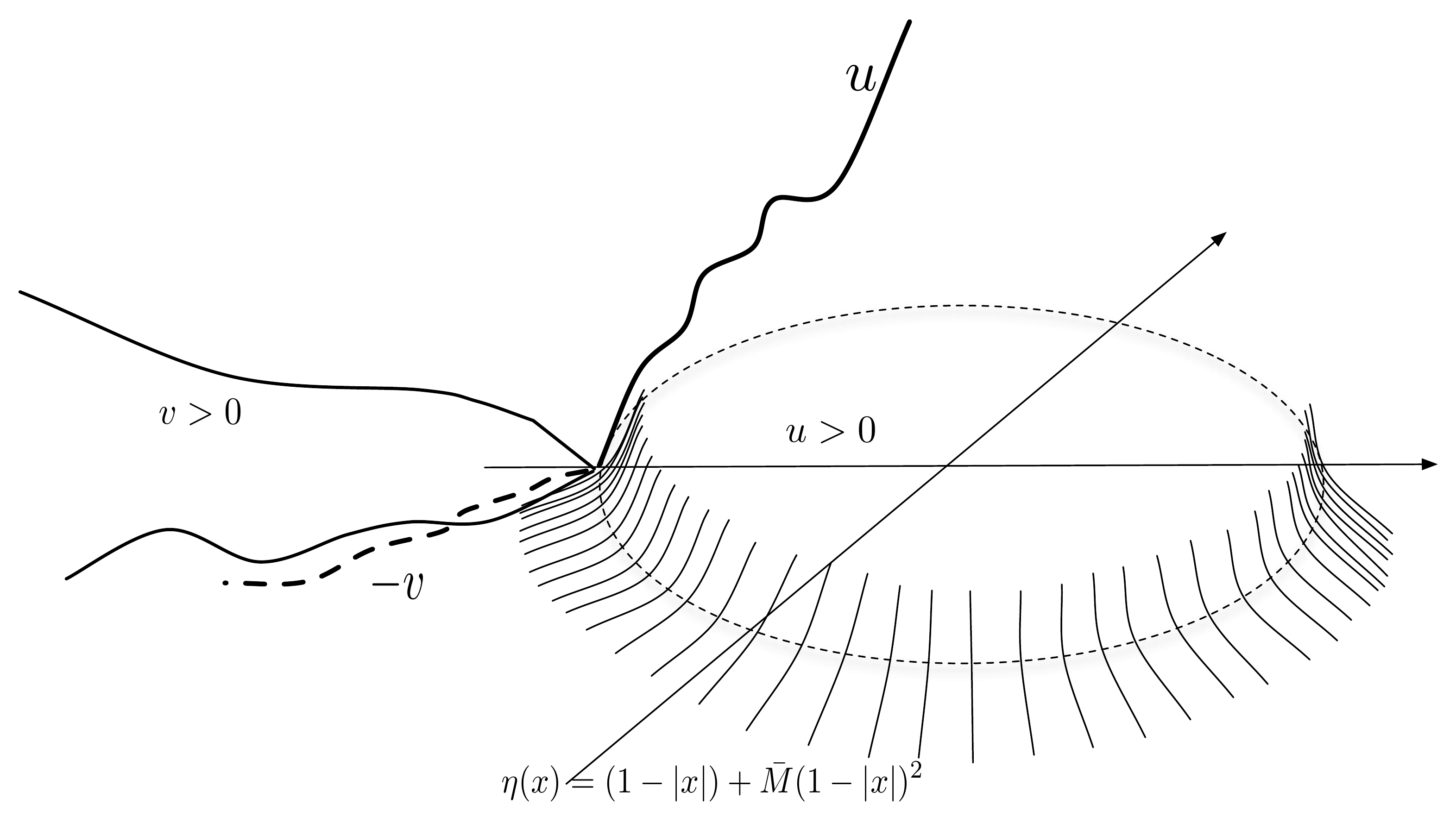}
\caption{Barrier function that touches $u-v$ from below at $z_{0}.$  }
\label{fig:final barrier for u-v}
\end{center}
\end{figure}

 \underline{First step:}
   By Lemma \ref{lemma: barrier sub} ($r=1, a=1, b=2$) there exist a subsolution  $\psi$ of $ \puccin$  on the ring $B_{1}(0)\backslash B_{\frac{1}{2}}(0)$ such that:

\begin{enumerate}
\item[(a)] $\psi (x) = 0 $ for  $x \in \partial B_{1}(0);$
\item[(b)] $\psi (x)= \frac{1}{C}M$ for $ x \in \partial B_{\frac{1}{2}}(0);$
\item[(c)] $\puccin (\psi ) \geq 0$ for $ x \in  B_{1}(0)\backslash  B_{\frac{1}{4}}(0);$
\item[(d)] $\dd{ \psi (x)}{\nu}= - \frac{\alpha}{2^{\alpha}-1}\frac{M}{C} $ for  $x \in \partial B_{1}(0),$ where $\nu$ is the outer normal direction.
\end{enumerate}

Due to the comparison principle applied in the ring $B_{1}(0) \backslash B_{\frac{1}{2}}(0)$ and by (\ref{bounds}) one can conclude that
 $$
 u (x) \geq \psi (x), \quad   x \in \partial B_{1}(0) \cup  \partial B_{\frac{1}{2}}(0)   \Rightarrow u(x) \geq \psi(x), \quad  x \in B_{1}(0) \backslash B_{\frac{1}{2}}(0),
 $$
and also, since $\psi$ is convex in the radial direction,  and by (d), 
\begin{equation}
\label{bound: surface for u}
u(x) \geq \underbrace{ \frac{\alpha}{2^{\alpha}-1}\frac{M}{C}}_{\overline{M}}  \dist(x, \partial B_{1}(0)) \qquad x \in B_{1}(0) \backslash B_{\frac{1}{2}}(0).
\end{equation}

\underline{Second Step}:
We do not know anything about the free boundary or the shape of the support of the other densities $v=\sum_{i\neq 1} \uind{i}$. But we know that if $z_{0}$ is the point on the free boundary closest to $0$  then  by the monotonicity formula 
\begin{eqnarray*}
\left( \frac{1}{\rho^{2}} \int_{B_{\rho}(z_{0})} \frac{\vert  \nabla u \vert^{2}}{\vert x- z_{0} \vert^{n-2}} \diff{x}\right) 
  \left( \frac{1}{\rho^{2}} \int_{B_{\rho}(z_{0})} \frac{\vert  \nabla  v \vert^{2}}{\vert x -z_{0}\vert^{n-2}}  \diff{x} \right) \leq N
\end{eqnarray*}
since the norm of the densities is bounded, say by the fixed constant $N$:
$$
 C(n)  \norm{(u,v)}{L^{2}(B_{1}(0))}^{4} \leq N.
$$

Our goal in this step is to prove that $v$ has to grow very slowly below   a linear function with small slope,  away from the free boundary. 

Let us consider the two possible cases around the point $z_{0}$:

\begin{enumerate}
\item[]\underline{Second Step a)}: For a sequence of small radii,  the measure of the intersection of the support of $v$ with each ball is almost zero: there is a sequence of radii $\rho_{k},$ $\rho_{k} \rightarrow 0,$  such that 
$$
\abs{\{ v \neq 0 \}\cap B_{\rho_{k}}(z_{0})} < \epsilon \abs{B_{\rho_{k}}(z_{0})},
$$
or,
\item[]\underline{Second Step b)}:  the measure of the support of $v$ contained in a ball with any radius is not small, say, if for any radius $\rho>0$
$$
\abs{\{ v \neq 0 \}\cap B_{\rho}(z_{0})}  > \epsilon \abs{B_{\rho}(z_{0})}.
$$
\end{enumerate}
The proof proceeds separately in each case  but in both cases we will prove (\ref{linear bound}).

\begin{enumerate}
\item[]\underline{ Proof of Second Step a)}:  If  there exists  a sequence of radius $(\rho_{k})_{k \in \mathbb{N}},$ $\rho_{k} \rightarrow 0,$ such that 
$$
\abs{\{ v \neq 0 \}\cap B_{\rho_{k}}(z_{0})} < \epsilon \abs{B_{\rho_{k}}(z_{0})},
$$
we can consider a subsequence of $(\rho_{k})_{k}$ such that 
$$... \rho_{k+1} < \frac{\rho_{k}}{2}\leq \rho_{k} < \frac{\rho_{k-1}}{2}\leq \rho_{k-1} \cdots \leq \rho_{1}\leq 1.$$
Since $v$ is bounded in all domain, we have that there exists $\tilde{N}_{1}$ such that
\begin{eqnarray*}
\sup_{x \in B_{\rho_{1}} (z_{0})} v(x) \leq \tilde{N_{1}} \rho_{1}.
\end{eqnarray*}
Then, since $v$  is subharmonic,  by  Proposition \ref{reference decay for ro}, we have that
\begin{eqnarray*}
\sup_{x \in B_{\frac{\rho_{1}}{2}}(z_{0})} v(x) \leq \epsilon\, \tilde{N}_{1} \rho_{1} 2^n.
\end{eqnarray*}
Then, since
$$
\puccip (v)\geq \puccin(v) \geq \sum_{i \neq 1} \puccin(\uind{i}) =0,
$$
by  Corollary \ref{wall of barriers normalized}  with $\tilde{U} = \epsilon\, \tilde{N}_{1} 2^{n+1},$   $\tilde{\sigma} =\frac{\rho_{1}}{2}$ and $t_{0}=1,$ we obtain that  there exists a universal constant $\tilde{C}:$
 $$
 v(x)\leq \tilde{C}\, \epsilon\, \tilde{N}_{1} \,2^{n+1} \dist(x, \partial B_{1}), 
 $$
 when $x \in S_{\frac{\rho_{1}}{2}}$ with $S_{\frac{\rho_{1}}{2}}:= \left( B_{1+\frac{\rho_{1}}{8}}(0)\backslash B_{1}(0)\right) \cap \{ x=(x',x_{n}): \abs{x'-z'_{0}}<\frac{\rho_{1}}{4} \}.$
 Let $\epsilon:=\epsilon_{0}$ be such that 
 \begin{eqnarray}
 \label{def epsilon}
 \epsilon_{0} \, \tilde{C}\,2^{n+1} \leq \frac{1}{2}.
 \end{eqnarray}
 Then, for $x \in S_{\frac{\rho_{1}}{2}}$
  $$
 v(x)\leq \frac{\tilde{N}_{1}}{2} \dist(x, \partial B_{1}). 
 $$
 Take as the next radii in the subsequence $\rho_{2}$ such that $B_{\rho_{2}}(z_{0}) \subset S_{\frac{\rho_{1}}{2}}.$
  Then, 
 \begin{eqnarray*}
\sup_{x \in B_{\rho_{2}} (z_{0})} v(x) \leq \frac{\tilde{N_{1}}}{2} \rho_{2}.
\end{eqnarray*}
 So we can repeat the process, and so, again by Proposition \ref{reference decay for ro}  with  $\tilde{N}_{2}=\frac{\tilde{N_{1}}}{2},$  we have that
\begin{eqnarray*}
\sup_{x \in B_{\frac{\rho_{2}}{2}}(z_{0})} v(x) \leq \epsilon_{0} \frac{\tilde{N}_{1}}{2} \rho_{2}\, 2^n.
\end{eqnarray*}
Then, by  Corollary \ref{wall of barriers normalized},  with $\tilde{U} = \epsilon_{0} \,\tilde{N_{1}} \,2^n,$   $\tilde{\sigma}=\frac{\rho_{2}}{2}$ and $t_{0}=1,$ and by (\ref{def epsilon}), 
 $$
 v(x)\leq \tilde{C} \epsilon_{0} \, \tilde{N_{1}} \,2^n \dist(x, \partial B_{1}) \leq \frac{\tilde{N_{1}}}{4}  \dist(x, \partial B_{1}), 
 $$
 when $x \in S_{\frac{\rho_{2}}{2}}$ with $S_{\frac{\rho_{2}}{2}}:= \left( B_{1+\frac{\rho_{2}}{8}}(0)\backslash B_{1}(0)\right) \cap \{ x=(x',x_{n}): \abs{x'-z'_{0}}<\frac{\rho_{2}}{4} \}.$
 
Again  take as the next radii in the subsequence $\rho_{3}$ such that $B_{\rho_{3}}(z_{0}) \subset S_{\frac{\rho_{2}}{2}}.$
  Then, 
 \begin{eqnarray*}
\sup_{x \in B_{\rho_{3}} (z_{0})} v(x) \leq \frac{\tilde{N_{1}}}{4} \rho_{3}.
\end{eqnarray*}
Repeating a finite number of times this process, we have that, for $x \in S_{\frac{\rho_{k}}{2}}$
 $$
 v(x)\leq  \frac{\tilde{N_{1}}}{2^k}  \dist(x, \partial B_{1}).
 $$
Therefore,  it is possible to find $\rho_{l}$ such that 
$ \frac{\tilde{N}_{1}}{2^{l}} \leq C \frac{1}{\overline{M}} $, and so
\begin{eqnarray*}
 v(x)\leq C\frac{1}{\overline{M}} \underbrace{ d(x, \partial B_{1}(0))}_{\abs{x}-1} \qquad  x \in \mathrm{S_{\frac{\rho_{l}}{2}}}.
 \end{eqnarray*}
 In particular,
 \begin{eqnarray}
\label{linear bound a)}
 v(x)\leq C\frac{1}{\overline{M}} \underbrace{ d(x, \partial B_{1}(0))}_{\abs{x}-1} \qquad  x \in \mathrm{B_{\frac{\rho_{l}}{8}}(z_{0})}.
 \end{eqnarray}

\item[]\underline{Proof of Second Step b)}:   If for any $\rho > 0$ as small  as we want
 \[
 \abs{\{ v \neq 0 \}\cap B_{\rho}(z_{0}) }  > \epsilon \abs{B_{\rho}(z_{0})}. \]
Since   
\[
\abs{\{ v \neq 0 \}\cap B_{\rho}(z_{0})}=  \abs{\{ u = 0 \}\cap B_{\rho}(z_{0})} > \epsilon \abs{B_{\rho}(z_{0})}, \] 
we can apply Poincaré-Sobolev inequality to the function $u$.  Also, by construction 
$$\abs{\{ u \neq 0 \}}=  \abs{\{ v = 0 \}} \geq \abs{B_{\rho}(z_{0})\cap B_{1}(0)} > \epsilon \abs{B_{\rho}(z_{0})},  $$ 
so we can also apply the Poincaré-Sobolev inequality to $v$.
Hence,
\begin{eqnarray*}
 \int_{B_{\rho}(z_{0})} v^{2}(x) \diff{x} \leq C(n, \epsilon) \rho^{2} \int_{B_{\rho}(z_{0})} \vert  \nabla v(x)   \vert^{2} \diff{x},
\end{eqnarray*}
\begin{eqnarray*}
 \int_{B_{\rho}(z_{0})} u^{2}(x) \diff{x} \leq C(n, \epsilon)  \rho^{2} \int_{B_{\rho}(z_{0})} \vert  \nabla u(x)   \vert^{2} \diff{x}.
\end{eqnarray*}
Then,
\begin{eqnarray*}
 \int_{B_{\rho}(z_{0})} v^{2} (x) \diff{x} \leq C(n, \epsilon)  \rho^{2} \int_{B_{\rho}(z_{0})} \vert \nabla v (x)   \vert^{2} \frac{\rho^{n-2} }{\vert x-z_{0}\vert^{n-2}}\diff{x},
\end{eqnarray*}
\begin{eqnarray*}
 \int_{B_{\rho}(z_{0})} u^{2} (x) \diff{x} \leq C(n, \epsilon)  \rho^{2}  \int_{B_{\rho}(z_{0})}\vert \nabla u (x)   \vert^{2} \frac{ \rho^{n-2} }{\vert x-z_{0}\vert^{n-2}}\diff{x},
\end{eqnarray*}
and 
\begin{eqnarray*}
 \left( \int_{B_{\rho}(z_{0})} v^{2} (x) \diff{x} \right)\left(  \int_{B_{\rho}(z_{0})} u^{2} (x) \diff{x} \right) \leq C(n, \epsilon) \,\rho^{2n-4}\rho^{8}\, N.
\end{eqnarray*}
\end{enumerate}
But we also know that $u$ is controlled from below by the barrier function from the First Step, so if $A = B_{1}(0)\backslash B_{\frac{1}{2}}(0)$
\begin{eqnarray*}
 \int_{B_{\rho}(z_{0})\cap A } \psi^{2} (x) \diff{x}\leq  \int_{B_{\rho}(z_{0})\cap A } u^{2} (x) \diff{x}
   \leq \int_{B_{\rho}(z_{0})} u^{2} (x) \diff{x}.
\end{eqnarray*}

To estimate the integral of the barrier function we can use a linear function defined all over $\partial B_{1}(0)$ by the slope evaluated  at $z_{0}.$

As in the first step, take $\overline{M} =  \frac{\alpha}{2^{\alpha}-1}\frac{M}{C} $ to be the absolute value of the slope of the barrier function. Note that  $M $ is arbitrarily big and so $\frac{1}{\overline{M}}$ will be arbitrarily small.
Assume that the volume of the cone defined by the barrier function is bounded from below by $ c_{1}\rho^{2} \rho^{n} \omega_{n}\frac{1}{n} $  where $c_{1}$ depends on the area of the base of the cone, then, 
\begin{eqnarray*}
\frac{ c_{1} {\overline{M}}^{2}\rho^{2} \rho^{n} \omega_{n} }{n}\leq \int_{B_{\rho}(z_{0})\cap A } \psi^{2} (x) \diff{x}
\end{eqnarray*}
therefore,
 \begin{eqnarray*}
 \left( \int_{B_{\rho}(z_{0})} v^{2} (x) \diff{x} \right)\left( \frac{ c_{1} {\overline{M}}^{2}\rho^{2} \rho^{n} \omega_{n} }{n} \right) \leq  C(n, \epsilon) \,\rho^{2n-4} \rho^{8}N,
\end{eqnarray*}
so
\begin{eqnarray*}
\left( \int_{B_{\rho}(z_{0})} v^{2} (x) \diff{x} \right) \leq \frac{ C(n, \epsilon) \,\rho^{2n-4} \rho^{8} N}{c_{1} {\overline{M}}^{2}\rho^{2} \rho^{n} \omega_{n} } =  C( \alpha, n, c_{1}, \epsilon) \rho^{n+2}\frac{  N}{ {\overline{M}}^{2}}.
\end{eqnarray*}

Let $y \in B_{\frac{7 \rho}{8}}(z_{0})$ such that $B_{\frac{\rho}{16}}(y) \subset B_{\rho}(z_{0}).$ Due to the subharmonicity and positivity of $v$
\begin{eqnarray*}
v (y) \leq \fint_{B_{\frac{\rho}{16}}(y)} v (x)\diff{x} \leq \left( \frac{\rho}{\frac{\rho}{16}}\right)^{n} \fint_{B_{\rho}(z_{0})} v(x)\diff{x}. 
\end{eqnarray*}
On the other hand, the  \holder inequality yields,
\begin{eqnarray*}
\left(
\fint_{B_{\rho}(z_{0})} v(x)\diff{x} 
\right)^{2}
\leq
 \fint_{B_{\rho}(z_{0})}v (x)   ^{2} \diff{x}.
\end{eqnarray*}
All  together, for any $y \in B_{\frac{7 \rho}{8}}(z_{0})$, we have,
\begin{eqnarray*}
v^{2} (y) \leq 16^{2n}  \fint_{B_{\rho}(z_{0})}v (x)   ^{2} \diff{x} =16^{2n}  \frac{1}{\rho^{n} \omega_{n}}   \int_{B_{\rho}(z_{0})}v (x)   ^{2} \diff{x},
\end{eqnarray*}
and hence
\begin{eqnarray*}
v^{2} (y) \leq \frac{ 16^{2n }}{\rho^{n} \omega_{n}}  C( \alpha, n , c_{1}, \epsilon) \rho^{n+2}\frac{  N}{ {\overline{M}}^{2}}= \underbrace{16^{2n}\, C( \alpha, n, c_{1}, \epsilon)}_{c_{2}^{2}}  \rho^{2}  \frac{  N}{   {\overline{M}}^{2}}.
\end{eqnarray*}
 We conclude that  $v(y)\leq c_{2} \frac{\sqrt{N}}{\overline{M}} \frac{7 \rho}{8},$ for  $y \in B_{\frac{7 \rho}{8}}(z_{0})$. Note that doing the same type of argument for $r=\rho t,$ $0<t<1$  we can conclude that 
 \begin{equation}
 \label{eq: upper bound v}
 v(y)\leq c_{2} \sqrt{N}\frac{1}{\overline{M}} \abs{y-z_{0}}.
 \end{equation}
 Since, $v$ a subsolution of the positive extremal Pucci operator
$$
\puccip (v)\geq \puccin(v) \geq \sum_{i \neq 1} \puccin(\uind{i}) =0,
$$
and we have by (\ref{eq: upper bound v}) that
  \begin{eqnarray*}
 v(y)\leq c_{2} \sqrt{N}\frac{1}{\overline{M}} \frac{\rho}{2}, \qquad x \in B_{\frac{\rho}{2}}(z_{0}),
 \end{eqnarray*}
 by Corollary \ref{wall of barriers normalized}, with $\tilde{U}= c_{2} \frac{ \sqrt{N}}{\overline{M}}$,  $\tilde{\sigma}=\frac{\rho}{2},$ and $t_{0}=1, $ we obtain that
 \begin{eqnarray*}
 v(y) \leq \tilde{C} \frac{ c_{2} \sqrt{N}}{\overline{M}} \dist(y, \partial B_{1}(0)), \quad \mbox{when} \quad y \in S_{\frac{\rho}{2}}.
\end{eqnarray*}
As before, $S_{\frac{\rho}{2}}$ is the portion of an annulus around $B_{1}(0)$ contained in $B_{\frac{\rho}{2}}(z_{0}), $
$$S_{\frac{\rho}{2}}=\left(B_{1+\frac{\rho}{8}}(0)\backslash B_{1}(0)\right)\cap \{(x',x_{n}) \in \rr^{n}: \abs{x'-z_{0}'}<\frac{\rho}{4}\}.$$ 
More explicitly,  we have that for $y \in B_{\frac{\rho}{8}}(z_{0}), $
\begin{equation}
\label{bound: surface for v}
v(y) \leq \underbrace{c_{2} \sqrt{N}\frac{8}{5} \frac{\alpha}{\frac{1}{5}-\left(\frac{1}{5}\right)^{\alpha+1}}}_{C} \frac{1}{\overline{M}} \underbrace{\dist (y, \partial B_{1}(0))}_{\abs{y}-1}.
\end{equation}

 \underline{Third step:}
Consider $\rho $  the radii in Second Step a). Consider $\epsilon$ to be $\epsilon_{0}$ in Second Step a), and $C$ in (\ref{linear bound a)}) to be 
\begin{eqnarray*}
C=c_{2} \sqrt{N}\frac{8}{5} \frac{\alpha}{\frac{1}{5}-\left(\frac{1}{5}\right)^{\alpha+1}}.
\end{eqnarray*}
Putting together (\ref{bound: surface for u}), (\ref{linear bound a)}) and (\ref{bound: surface for v}), if $A= B_{1}(0) \cap B_{\delta}(z_{0})$ and $B= S \cap B_{\delta}(z_{0}) $
$$
u(x) \geq \overline{M}(1-\abs{x}) \quad x \in A \qquad \mbox{and} \qquad -v(x)\geq - \frac{C}{\overline{M}}(\abs{x}-1) \quad x \in B.
$$

So letting $\eta$ be the radial symmetric smooth function defined by (see Figure \ref{fig:final barrier for u-v})
$$
\eta(x) = (1-\abs{x}) + \overline{M}(1-\abs{x})^{2},
$$
we can conclude that 
$$
(u-v)(x) \geq \eta (x) \qquad x \in B_{\delta}(z_{0})
$$
and that 
$$
(u-v)(z_{0}) = \eta (z_{0}).
$$
But notice that the Hessian of $\eta$ is given by
$$
H(\eta)(x)=
\left[
\begin{array}{ccccc}
 2 \overline{M}&0&0&\cdots&0\\
0&\frac{-1 + 2\overline{M}(r-1)}{r}&0&\cdots&0\\
\vdots&\vdots&\vdots&\vdots&\vdots\\
0&\cdots&0&\frac{-1 + 2\overline{M}(r-1)}{r}&0\\
0&0&\cdots&0 &\frac{-1 + 2\overline{M}(r-1)}{r}\\
\end{array}
\right]
$$
(where $r=\abs{x}$) and so at 
$z_{0}$ because $r=1$ we obtain
$\puccin(\eta (z_{0}))= \lambda 2 \overline{M} - \Lambda (n-1)> 0,$
 since $\overline{M}$ is as big as we wish.  But this contradicts the fact that $u-v$ is a supersolution of \puccin and this finishes the proof of (1).

(2) Arguing in a similar way as in Proposition 4.13, in \cite{caffarelli_fully_1995}, we should consider:

\begin{enumerate}
\item[a)] When $2 \abs{x-y}< h$, $h=\max (\dist(x, \partial \supp \uind{1}),\dist(y, \partial \supp \uind{1})),$ $y \in B_{\frac{h}{2}}(x)$ so by interior regularity (Proposition \ref{prop: interior holder} with $f=0$) for properly scaled balls gives that 
\[
 \abs{\uind{1}(x)-\uind{1}(y)}\leq C \norm{ \uind{1}}{L^{\infty}(B_{1}(0))} \abs{x-y};
 \]

\item[b)] When $2 \abs{x-y} \geq  h$  note that for $x$ in the support of $\uind{1}$ and $\overline{x}$ is the closest point to $x$  on the free boundary we have from the previous result that
$$
 \abs{\uind{1}(x)-\uind{1}(\overline{x})} \leq \abs{\uind{1}(x)}\leq Ch.
$$
And so, by adding and subtracting $\uind{1}(\overline{x})$ and $\uind{1}(\overline{y})$, where $\overline{y}$ is the closest point to $y$  on the free boundary,  and by triangular inequality,
\begin{eqnarray*}
 \frac{\abs{\uind{1}(x)-\uind{1}(y)} }{\abs{x-y}}
 &\leq &\frac{\abs{\uind{1}(x)-\uind{1}(\overline{x})}}{\abs{x-y}} 
 +\frac{\abs{\uind{1}(\overline{x})-\uind{1}(\overline{y})}}{\abs{x-y}}
 +\frac{\abs{\uind{1}(\overline{y})-\uind{1}(y)}}{\abs{x-y}}\\
& \leq&
\frac{2\abs{\uind{1}(x)}}{h} 
 +\frac{2 \abs{\uind{1}(y)}}{h}\leq 4C.
\end{eqnarray*}
\end{enumerate}
Thus,
\begin{eqnarray*}
\Vert \uind{1} \Vert_{Lip(B_{\frac{1}{4}}(x_{0}))} = \norm{ u_{1}}{ C^{0,1}(B_{\frac{1}{4}}(x_{0}))}\leq \tilde{C},
\end{eqnarray*}
and  the result follows.
\end{proof}

\textbf{Acknowledgements: } To my professor Luis Caffarelli for introducing  me to this problem and guide me through all the work presented here. To my professor Diogo Gomes  and to my colleagues and friends Alessio Figalli, Betul Orcan, Diego Farias, Emanuel Indrei, Farid Bozorgnia, Fernando Charro, Nestor Guillen and Ray Yang, for all their suggestions and encouragement.

\section*{Appendices}

\appendix{}
 
\section{Pucci Operators. General properties }
\label{general pucci}

For the sake of completeness, we recall some properties of the Pucci operators in the following:
$$
  \puccin  (u) =  \inf_{A \in \mathcal{A}_{\lambda, \Lambda}}    a_{ij}D_{ij}(u)= \inf_{A \in \mathcal{A}_{\lambda, \Lambda}} \tr (A D^{2} u) =\Lambda \sum_{  e_{i}<0 } e_{i}  + \lambda \sum_{ e_{i}>0 } e_{i}\\
 $$   
    where:
 \begin{itemize}
    \item[]  $ \mathcal{A}_{\lambda, \Lambda}$ is the set of  symmetric matrices $(n \times n)$ with with eigenvalues   in $[\lambda, \Lambda]$ for $0<\lambda<\Lambda $;
    \item[] $e_{i} $ eigenvalue of the matrix  $D^{2} u$;
    \end{itemize}

\begin{remark}
  \item[] Observe that $\inf_{A \in \mathcal{A}_{\lambda, \Lambda}}    a_{ij}D_{ij}(u) \leq  a_{ij}D_{ij}(u) \leq  \sup_{A \in \mathcal{A}_{\lambda, \Lambda}}    a_{ij}D_{ij}(u)$.
  
  \item[]   Let $u, v $ be  smooth functions and $0<\lambda<1<\Lambda :$ 
\begin{itemize}
\item[-] $\puccin (u) \leq  \Lambda \D u \leq  \puccip (u)$  ;
\item[-] $\puccin (-u) = - \puccip(u)$;
\item[-] $\puccin(u)+\puccin (v) \leq \puccin(u+v) \leq \puccip(u) + \puccin (v)$ and so  $\puccin$ is concave;
\item[-] $\puccip(u)+\puccin (v) \leq \puccip(u+v) \leq \puccip(u) + \puccip (v)$ and so $\puccip$ is convex ;
\item[-] $0 \leq \puccin (\ueind{i}) \leq  \Lambda \D \ueind{i} \leq  \puccip (\ueind{i}) \Rightarrow  \: \ueind{i} \: \mbox{subharmonic in the viscosity sense}$  ;
\item[-] $ 0 \leq \sum_{i} \puccin (\ueind{i}) \leq \puccin (\sum_{i} \ueind{i}) \leq  \D \left(  \sum_{i} \ueind{i} \right) \Rightarrow  \:  \sum_{i} \ueind{i} \: \mbox{subharmonic in the viscosity sense}$;
\end{itemize}

\end{remark}

\section{Fabes and Strook Inequality}
\label{app: F S inequality}

Here we state without proof part of Theorem 2 in \cite{fabes_lp_1984}. For more details see also Appendix B  in \cite{caffarelli_homogenization_2005}.

\begin{lemma}
\label{lemma: F S ineq}
Let $G(x,y)$ denote the Green's function for a linear operator L with measurable coefficients $ Lu=a_{ij}(x)D_{ij} u$, $\lbrack a_{ij}(x)\rbrack \in  \mathcal{A}_{\lambda, \Lambda}$. Then, there exist universal constants  $C$ and $\beta$ such that whenever $E \subset B_{r},$ and $B_{r} \subset B_{\frac{1}{2}}$ the following holds:
\begin{eqnarray*}
\left\lbrack \frac{\abs{E}}{ \abs{B_{r}}}\right\rbrack^{\beta} \int_{B_{r}} G(x,y)\dy \leq C \int_{E} G(x,y) \dy \qquad \mbox{for all $x \in  B_{1}$.}
\end{eqnarray*}
\end{lemma}

\section{Properties of subharmonic functions in thin domains.}
\label{apd for lips: subhar}

\begin{lemma}[$L^{\infty}$ decay for subharmonic functions supported in small domains]
\label{lemma: decay}
Let $u$ be a non-negative subharmonic function in a domain that contains $B_{1}(0).$ If
for some small $\epsilon_{0}>0,$
$$
\sup_{B_{1} (0)} u \leq 1 \quad \mbox{ and} \quad \frac{ \abs{\{u\neq 0\} \cap B_{1}(0)}}{\abs{B_{1}(0)}} \leq \epsilon_{0}
$$  
then, 
$$
\sup_{B_{\frac{1}{2}}(0)} u \leq \epsilon_{0}\, 2^n.
$$
\end{lemma}

\begin{proof}
Let $y$ be an arbitrary point in the ball $B_{\frac{1}{2}}(0). $ Due to the subharmonicity and the fact that $u$ is a non-negative function 
\begin{eqnarray*}
\begin{split}
u(y) &\leq \fint_{B_{\frac{1}{2}}(y)} u(x) \diff{x}
 \leq \frac{1}{\abs{B_{\frac{1}{2}}(y)}} \int_{B_{1}(0)} u(x) \diff{x} \leq \frac{\abs{B_{1}(0)\cap \supp \, u }}{\abs{B_{\frac{1}{2}}(y)}} \sup_{B_{1}(0)} u(x)
\end{split}
\end{eqnarray*}
and so by hypotheses, 
\begin{eqnarray*}
u(y) \leq \frac{ \epsilon_{0}\, \omega_{n}}{\left(\frac{1}{2}\right)^{n}\omega_{n}}
  =  \epsilon_{0}\, 2^n
\end{eqnarray*}
which gives the result.
\end{proof}

\begin{prop}
\label{reference decay for ro}
Let $u$ be a non-negative subharmonic function in a domain that contains $B_{1}(0).$ If
for some $\rho \leq 1$ and for some constants $N,\epsilon_{0} \geq 0$ 
\begin{eqnarray*}
\sup_{x \in B_{\rho} (0)} u(x) \leq N \rho \quad \mbox{ and} \quad \frac{ \abs{\{u\neq 0\} \cap B_{\rho}(0)}}{\abs{B_{\rho}(0)}} \leq \epsilon_{0},
\end{eqnarray*}
then, 
\begin{eqnarray*}
\sup_{x \in B_{\frac{\rho}{2}}(0)} u(x) \leq  N \,\rho \, \epsilon_{0}\, 2^n.
\end{eqnarray*}
\end{prop}

\begin{proof}
Consider the function $v$ defined on $B_{1}(0)$ by 
$$v(x)= \frac{1}{N \rho}u( \rho\,x) .$$
The new function $v$ satisfies
$$
\sup_{x \in B_{1}(0)} v (x) \leq 1,
$$ 
and so by Lemma \ref{lemma: decay}, we have that
$$
\sup_{x \in B_{\frac{1}{2}}(0)} v(x) \leq  \epsilon_{0}\, 2^n.
$$
Substituting $v$ by its definition in terms of $u$ gives that, 
\begin{eqnarray*}
\sup_{x \in B_{\frac{1}{2}}(0)}  \frac{1}{N \rho}u( \rho\,x) \leq \epsilon_{0}\, 2^n 
\Leftrightarrow 
\sup_{y \in B_{\frac{\rho}{2}}(0)}   u(y) \leq N \,\rho\, \epsilon_{0}\, 2^n
\end{eqnarray*}
which gives the final result.
\end{proof}

\section{Monotonicity formula}
\label{apd for lips: monotonicity}

\begin{lemma}
Let $u,$ $v$ be two \holder continuous functions defined on $B_{1}(0)$, nonnegative,  subharmonic when positive and with disjoints supports. Assume that $0$ belongs to the intersection of the boundary of their support. 
 Let $z_{0}$ be the point in the free boundary close to $0.$ Then, the following quantity is increasing with the radius $\rho$ and uniformly bounded,
\begin{eqnarray*}
\left( \frac{1}{\rho^{2}} \int_{B_{\rho}(z_{0})} \frac{\vert  \nabla u \vert^{2}}{\vert x- z_{0} \vert^{n-2}} \diff{x}\right) 
  \left( \frac{1}{\rho^{2}} \int_{B_{\rho}(z_{0})} \frac{\vert  \nabla  v \vert^{2}}{\vert x -z_{0}\vert^{n-2}}  \diff{x} \right) \leq N,
\end{eqnarray*}
since the norm of the densities is bounded, say by the fixed constant $N$;
$$
 C(n)  \norm{(u,v)}{L^{2}(B_{1}(0))}^{4} \leq N.
$$
\end{lemma}


For the proof see page 214 on \cite{caffarelli_geometric_2005} and the proof of Lemma 7 (a) in \cite{caffarelli_geometry_2009}.

\section{Existence of Barriers}
\label{apd for lips: barriers}

\begin{lemma}
\label{lemma: barrier sub}
Given constants $0<\lambda< \Lambda$ and $M, r ,a,b, \rho \geq 0,$  $\frac{a\,r}{b}<r<\rho$ there exist a smooth function defined on  $B_{\rho} (0) \backslash B_{\frac{a\,r}{2b}}(0)$  and a constant $c=c(\alpha, a,b)$ and a universal constant $\alpha$ such that:

\begin{enumerate}
\label{lemma: barrier sub}
\item $\psi (x)= r\, M $ for $x \in \partial B_{\frac{a\,r}{b}}(0)$ 
\item $\psi (x) =0 $ for $x \in \partial B_{r}(0)$
\item $\puccin (\psi) \geq 0 $ for $x \in B_{r}(0) \backslash B_{\frac{a\,r}{2 b}}(0)$ 
\item $\frac{\partial \psi}{ \partial \nu} = c M $ when $x \in \partial B_{r}(0),$ and $c=-\alpha \frac{a^{\alpha}}{b^{\alpha}-a^{\alpha}}. $ 
\end{enumerate}
\end{lemma}

\begin{proof}
Consider first $r=1$ and $M=1,$  we will  rescale afterwards. 
Let $\alpha, M_{2}>0,  $  $\alpha > n-2 $ and
$$
\varphi (x)= M_{1} + M_{2}\frac{1}{\abs{x}^{\alpha}}. 
$$
where  $M_{1},$ $M_{2}$ and $\alpha$ are such that the following conditions are satisfied:
\begin{enumerate}
\item $\varphi (x)= 0$ when $\abs{x}=1$\\
\item $\varphi (x)= 1$ when $\abs{x}=\frac{a}{b}<1$\\
\item $\puccin (\varphi) \geq 0$ in $B_{1}(0) \backslash B_{\frac{a}{2b}}(0)$
\end{enumerate}
In detail,
\begin{enumerate}
\item $\varphi (x)= 0$ when $\abs{x}=1$ $\Rightarrow$ $M_{1}=-M_{2}$ \\
\item $\varphi (x)= 1$ when $\abs{x}=\frac{a}{b}$ $\Rightarrow$   $M_{2}=\frac{1}{\left(\frac{a}{b}\right)^{\alpha}-1}= \frac{a^{\alpha}}{b^{\alpha}-a^{\alpha}}\geq 0$ \\
\end{enumerate}
and so have that,
$$
\varphi (x)= -M_{2} + M_{2}\frac{1}{\abs{x}^{\alpha}} \quad \mbox{with} \quad M_{2}=\frac{a^{\alpha}}{b^{\alpha}-a^{\alpha}}. 
$$
Note that if $ \frac{a}{b}$ is very small then $M_{2}$ is very small too.
On the other hand, the second derivatives of $\varphi$ are given by:
$$
\partial_{ij} \varphi (x ) = - \alpha M_{2} \abs{x}^{-\alpha -2} \delta_{ij}  -\alpha (-\alpha -2)M_{2} x_{i}x_{j}\abs{x}^{-\alpha -4}.
$$
Evaluating the Hessian of $\varphi$ at a point $(r, 0, \cdots, 0)$ one obtains:
\begin{eqnarray*}
\begin{split}
&\partial_{ij} \varphi  = 0 \quad i \neq j \\
&\partial_{11} \varphi = M_{2} \alpha (\alpha +1 )r^{-\alpha -2}\\
&\partial_{ii} \varphi = - \alpha M_{2} r^{-\alpha -2} \quad i>1\\
\end{split}
\end{eqnarray*}
And so by radial symmetry 
\begin{eqnarray*}
\begin{split}
&\partial_{ij} \varphi (x ) = 0 \quad i \neq j\\
&\partial_{11} \varphi (x ) = M_{2}\left(  \alpha (\alpha +1 )\right)\abs{x}^{-\alpha -2}\\
&\partial_{ii} \varphi (x ) = - \alpha M_{2}\abs{x}^{-\alpha -2} \quad i>1\\
\end{split}
\end{eqnarray*}
which implies that the Pucci extremal operator is given by
\begin{eqnarray*}
\begin{split}
\puccin (\varphi (x))&= M_{2}  \lambda  \alpha (\alpha +1 ) \abs{x}^{-\alpha -2} -\Lambda \, (n-1)\, \alpha M_{2} \abs{x}^{-\alpha -2}\\
&=  M_{2} \alpha\, \abs{x}^{-\alpha -2}  \left( \lambda (\alpha +1 ) -\Lambda \, (n-1) \right)
\end{split}
\end{eqnarray*}
In order to satisfy (3) one needs that:
$$
\alpha \geq \frac{\Lambda (n-1) - \lambda }{\lambda}
$$
which gives that
$$
\varphi (x)= -M_{2} + M_{2}\frac{1}{\abs{x}^{\alpha}} \quad \mbox{with } \: M_{2}=\frac{a^{\alpha}}{b^{\alpha}-a^{\alpha}} \:  \mbox{and} \quad  \alpha \geq \frac{\Lambda (n-1) - \lambda }{\lambda}. 
$$
Notice that  the normal derivative is:
$$
\dd{ \varphi }{\nu }(x) = -\alpha \, M_{2} \frac{1}{\abs{x}^{\alpha +1}}.
$$
And so when $r=\abs{x}=1,$
$$
\dd{ \varphi }{\nu }(x) = -\alpha \, M_{2}= - \alpha \frac{a^{\alpha}}{b^{\alpha}-a^{\alpha}} .
$$
Let $c= - \alpha \frac{a^{\alpha}}{b^{\alpha}-a^{\alpha}}. $\\
Now let u s consider a dilation and obtain the result for general $r $ and $M=1.$
Let $\tilde{\varphi}(x)= r \varphi(\frac{x}{r}).$ Then $\tilde{\varphi}$ defined on
$B_{r}(0) $ satisfies:
\begin{enumerate}
\item $\tilde{\varphi} (x)= 0$ when $\abs{x}=r$\\
\item $\tilde{\varphi} (x)= r$ when $\abs{x}=\frac{a\,r}{b}$\\
\item $\puccin (\tilde{\varphi}) \geq 0$ in $B_{r}(0) \backslash B_{\frac{a\,r}{2b}}(0)$
\item $\dd{ \tilde{\varphi} }{\nu }(x)= c $ when $\abs{x}=r$
\end{enumerate}
Finally for an arbitrary $M,$ let $\psi (x)= M \tilde{\varphi}(x).$
Now is easy to check that the barrier function $\psi$ satisfies:
\begin{enumerate}
\item $\psi (x)= 0$ when $\abs{x}=r$\\
\item $\psi (x)= r M$ when $\abs{x}=\frac{a\, r}{b}$\\
\item $\puccin (\psi) \geq 0$ in $B_{r}(0) \backslash B_{\frac{a\, r}{b}}(0)$
\item $\dd{ \psi }{\nu }(x)= c M$ when $\abs{x}=r$
\end{enumerate}
 \end{proof}

\begin{lemma}
\label{lemma: barrier super}
Given constants $0<\lambda< \Lambda$ and $M, r ,a,b, \rho \geq 0$,  $\frac{a}{b}r<r< \rho$ there exist a smooth function  defined on $B_{\rho} (0) \backslash B_{\frac{a\, r}{2b}}(0)$ and a constant $c=c(\alpha, a,b)$ and a universal constant $\alpha$ such that:

\begin{enumerate}
\item $\psi (x)=r\, M $ for $x \in \partial B_{r}(0)$ 
\item $\psi (x) =0 $ for $x \in \partial B_{r\frac{a}{b}}(0)$
\item $\puccip (\psi) \leq 0 $ for $x \in B_{r}(0) \backslash B_{r\frac{a}{2b}}(0)$ 
\item $\frac{\partial \psi}{ \partial \nu} = c M $ when $x \in \partial B_{r\frac{a}{b}}(0),$ where $c=\alpha \frac{1}{ \frac{a}{b}-\left( \frac{a}{b}\right)^{\alpha +1}}$
\end{enumerate}
\end{lemma}

\begin{proof}
Consider first $r=1$  and $M=1,$ we will rescale afterwards.
Let $\alpha, M_{2}>0$   $\alpha > n-2 $ and
$$
\varphi (x)= M_{1} - M_{2}\frac{1}{\abs{x}^{\alpha}}. 
$$
where  $M_{1},$ $M_{2}$ and $\alpha$ are such that the following conditions are satisfied:
\begin{enumerate}
\item $\varphi (x)= 0$ when $\abs{x}=\frac{a}{b}$\\
\item $\varphi (x)= 1$ when $\abs{x}=1$\\
\item $\puccip (\varphi) \leq 0$ in $B_{1}(0) \backslash B_{\frac{a}{2 b}}(0)$
\end{enumerate}
In detail, (1) and (2) imply that
 $$
 M_{1}=1+M_{2} 
 \quad \mbox{and }  M_{2}=\frac{a^{\alpha}}{ b^{\alpha}-a^{\alpha}}
 $$ 
and so have that,
$$
\varphi (x)= 1+M_{2} - M_{2}\frac{1}{\abs{x}^{\alpha}} \quad \mbox{with} \quad M_{2}=\frac{a^{\alpha}}{ b^{\alpha}-a^{\alpha}} . 
$$
On the other hand, the second derivatives of $\varphi$ are given by:
$$
\partial_{ij} \varphi (x ) = \alpha M_{2} \abs{x}^{-\alpha -2} \delta_{ij}  + \alpha (-\alpha -2)M_{2} x_{i}x_{j}\abs{x}^{-\alpha -4}.
$$
Evaluating the Hessian of $\varphi$ at a point $(r, 0, \cdots, 0)$ one obtains:
\begin{eqnarray*}
\begin{split}
&\partial_{ij} \varphi  = 0 \quad i \neq j \\
&\partial_{11} \varphi = - M_{2} \alpha (\alpha +1 )r^{-\alpha -2}\\
&\partial_{ii} \varphi =  \alpha M_{2} r^{-\alpha -2} \quad i>1\\
\end{split}
\end{eqnarray*}
And so by radial symmetry 
\begin{eqnarray*}
\begin{split}
&\partial_{ij} \varphi (x ) = 0 \quad i \neq j\\
&\partial_{11} \varphi (x ) = - M_{2} \alpha (\alpha +1 ) \abs{x}^{-\alpha -2}\\
&\partial_{ii} \varphi (x ) = \alpha M_{2}\abs{x}^{-\alpha -2} \quad i>1\\
\end{split}
\end{eqnarray*}
which implies that the Pucci extremal operator is given by
\begin{eqnarray*}
\begin{split}
\puccip (\varphi (x))&=  \Lambda \, (n-1)\, \alpha M_{2} \abs{x}^{-\alpha -2}- \lambda  M_{2}   \alpha (\alpha +1 ) \abs{x}^{-\alpha -2}\\
&=  M_{2} \alpha\, \abs{x}^{-\alpha -2}  \left( \Lambda \, (n-1) -\lambda (\alpha +1 ) \right)
\end{split}
\end{eqnarray*}
In order to satisfy (3) one needs that:
$$
\Lambda \, (n-1) -\lambda (\alpha +1 ) \leq 0       \Leftrightarrow      \alpha \geq \frac{\Lambda (n-1) - \lambda }{\lambda}
$$
which gives that
$$
\varphi (x)= 1+M_{2} - M_{2}\frac{1}{\abs{x}^{\alpha}} \quad \mbox{with } \: M_{2}=\frac{a^{\alpha}}{ b^{\alpha}-a^{\alpha}}\:  \mbox{and} \quad  \alpha \geq \frac{\Lambda (n-1) - \lambda }{\lambda}. 
$$
Notice that  the normal derivative is:
$$
\dd{ \varphi }{\nu }(x) = \alpha \, M_{2} \frac{1}{\abs{x}^{\alpha +1}}.
$$
And so when $\abs{x}=\frac{a}{b},$
$$
\dd{ \varphi }{\nu }(x) = \alpha \, M_{2} \frac{1}{\left(\frac{a}{b}\right)^{\alpha+1}}=\frac{\alpha}{        \frac{a}{b}-\left( \frac{a}{b}\right)^{\alpha +1}  } .   $$
Let $c=\frac{\alpha}{        \frac{a}{b}-\left( \frac{a}{b}\right)^{\alpha +1} }  . $\\
Now let us consider a dilation and obtain the result for general $r $ and $M=1.$
Let $\tilde{\varphi}(x)= r \varphi(\frac{x}{r}).$ Then $\tilde{\varphi}$ defined on
$B_{r}(0) $ satisfies:
\begin{enumerate}
\item $\tilde{\varphi} (x)= 0$ when $\abs{x}=r\frac{a}{b}$,\\
\item $\tilde{\varphi} (x)= r$ when $\abs{x}=r$,\\
\item $\puccip (\tilde{\varphi}) \leq 0$ in $B_{r}(0) \backslash B_{r\frac{a}{2b}}(0)$,
\item $\dd{ \tilde{\varphi} }{\nu }(x)= c $ when $\abs{x}=r \frac{a}{b}.$
\end{enumerate}
Finally for an arbitrary $M,$ let $\psi (x)= M \tilde{\varphi}(x).$
Now, it is easy to check that the barrier function $\psi$    satisfies
\begin{enumerate}
\item $\psi (x)= 0$ when $\abs{x}=r\frac{a}{b}$,\\
\item $\psi (x)= r M$ when $\abs{x}=r$,\\
\item $\puccip (\psi) \leq 0$ in $B_{r}(0) \backslash B_{r\frac{a}{b}}(0),$
\item $\dd{ \psi }{\nu }(x)= c M$ when $\abs{x}=r\frac{a}{b}.$
\end{enumerate}
 \end{proof}


\bibliography{articlearxivevquitalo}
\bibliographystyle{plain}

\end{document}